\documentclass[11pt]{article}
\usepackage[margin=1in]{geometry}
\usepackage{amsfonts,amsmath,amssymb,amsthm}
\usepackage{aliascnt}
\usepackage{comment}
\usepackage{mathtools,bm,mathrsfs,mleftright}
\usepackage{xcolor,enumitem,autobreak}
\usepackage{needspace}
\allowdisplaybreaks[4]
\usepackage{graphicx,longtable}
\usepackage{placeins}
\usepackage[authoryear,round]{natbib}
\setcitestyle{authoryear,round}
\usepackage[colorlinks,citecolor=blue,linkcolor=blue,urlcolor=blue]{hyperref}
\usepackage{xr-hyper}
\usepackage[nameinlink,noabbrev]{cleveref}
\crefname{theorem}{Theorem}{Theorems}
\Crefname{theorem}{Theorem}{Theorems}
\crefname{lemma}{Lemma}{Lemmas}
\Crefname{lemma}{Lemma}{Lemmas}
\crefname{corollary}{Corollary}{Corollaries}
\Crefname{corollary}{Corollary}{Corollaries}
\crefname{proposition}{Proposition}{Propositions}
\Crefname{proposition}{Proposition}{Propositions}
\crefname{definition}{Definition}{Definitions}
\Crefname{definition}{Definition}{Definitions}
\crefname{remark}{Remark}{Remarks}
\Crefname{remark}{Remark}{Remarks}
\crefname{example}{Example}{Examples}
\Crefname{example}{Example}{Examples}
\crefname{assumption}{Assumption}{Assumptions}
\Crefname{assumption}{Assumption}{Assumptions}
\crefname{condition}{Condition}{Conditions}
\Crefname{condition}{Condition}{Conditions}
\crefname{section}{Section}{Sections}
\Crefname{section}{Section}{Sections}
\crefname{subsection}{Subsection}{Subsections}
\Crefname{subsection}{Subsection}{Subsections}
\crefname{subsubsection}{Subsection}{Subsections}
\Crefname{subsubsection}{Subsection}{Subsections}

\crefformat{equation}{(#2#1#3)}
\crefrangeformat{equation}{(#3#1#4)--(#5#2#6)}
\crefmultiformat{equation}{(#2#1#3)}{ and (#2#1#3)}{, (#2#1#3)}{, and (#2#1#3)}

\newcommand{\ep}{\varepsilon}

\newcommand{\R}{\mathbb{R}}
\newcommand{\E}{\mathbb{E}}
\newcommand{\bbS}{\mathbb{S}}

\newcommand{\bbP}{\mathbb{P}}

\newcommand{\caH}{\mathcal{H}}
\newcommand{\caI}{\mathcal{I}}
\newcommand{\caJ}{\mathcal{J}}
\newcommand{\caK}{\mathcal{K}}
\newcommand{\caT}{\mathcal{T}}
\newcommand{\caU}{\mathcal{U}}
\newcommand{\caV}{\mathcal{V}}
\newcommand{\caX}{\mathcal{X}}
\newcommand{\caZ}{\mathcal{Z}}
\newcommand{\caL}{\mathcal{L}}

\renewcommand{\Pr}{\bbP}

\DeclareMathOperator*{\argmin}{arg\,min}
\DeclareMathOperator*{\esssup}{ess\,sup}
\DeclareMathOperator*{\supp}{Supp}
\DeclareMathOperator{\cl}{cl}
\DeclareMathOperator{\conv}{conv}

\DeclareMathOperator{\ran}{ran}
\DeclareMathOperator{\sign}{sign}

\DeclareMathOperator{\Tr}{Tr}
\DeclareMathOperator{\Var}{Var}
\providecommand{\TV}{\mathrm{TV}}

\newcommand{\mfC}{\mathfrak{C}}
\newcommand{\uC}{\mfC^{(\beta)}}
\newcommand{\uCt}{\uC_{\theta_0}}

\newcommand{\xk}[1]{\left(#1\right)}
\newcommand{\zk}[1]{\left[#1\right]}
\newcommand{\dk}[1]{\left\{#1\right\}}
\newcommand{\xkx}[1]{(#1)}
\newcommand{\zkx}[1]{[#1]}
\newcommand{\dkx}[1]{\{#1\}}
\providecommand{\ang}[1]{\left\langle{#1}\right\rangle}

\providecommand{\abs}[1]{\left\lvert{#1}\right\rvert}
\providecommand{\norm}[1]{\left\lVert{#1}\right\rVert}

\providecommand{\normx}[1]{\lVert{#1}\rVert}

\newcommand{\xkm}[1]{\mleft(#1\mright)}

\newcommand{\dkm}[1]{\mleft\{#1\mright\}}

\providecommand{\ind}[1]{\mathbf{1}\{#1\}}

\providecommand{\diam}{\operatorname{diam}}
\providecommand{\HS}{\operatorname{HS}}

\theoremstyle{plain}
\newtheorem{theorem}{Theorem}[section]
\newaliascnt{lemma}{theorem}
\newtheorem{lemma}[lemma]{Lemma}
\aliascntresetthe{lemma}
\newaliascnt{corollary}{theorem}
\newtheorem{corollary}[corollary]{Corollary}
\aliascntresetthe{corollary}

\theoremstyle{definition}
\newaliascnt{proposition}{theorem}
\newtheorem{proposition}[proposition]{Proposition}
\aliascntresetthe{proposition}
\newaliascnt{definition}{theorem}
\newtheorem{definition}[definition]{Definition}
\aliascntresetthe{definition}
\newaliascnt{remark}{theorem}
\newtheorem{remark}[remark]{Remark}
\aliascntresetthe{remark}
\newtheorem{example}{Example}[section]
\newtheorem{assumption}{Assumption}
 \title{Exact Asymptotic Efficiency under zCDP: Diameter-Constrained Information Geometry}
\author{T. Tony Cai\thanks{Department of Statistics and Data Science, The Wharton School, University of Pennsylvania. Email: \texttt{tcai@wharton.upenn.edu}.}
\and Yicheng Li\thanks{KLATASDS-MOE, School of Statistics, East China Normal University, Shanghai, China. Email: \texttt{ycli@sfs.ecnu.edu.cn}.}}
\date{\today}
\newenvironment{frontmatter}{}{}
\begin{document}
\begin{frontmatter}
\maketitle
\begin{abstract}
      Protecting individual privacy has become a central and urgent concern in modern data analysis, given the vast quantities of data now generated and processed. In this paper,
we develop a systematic theory of exact asymptotic efficiency for regular parametric estimation under central zero-concentrated differential privacy.
In the privacy regime, the governing object is a diameter-constrained information region: the set of information matrices generated by statistics with diameter at most one.
For weighted quadratic loss, we show that the exact local minimax risk is an inverse information variational functional over this region.

More generally, a mixed information region yields a unified efficiency constant across different regimes,
covering the privacy regime and classical Fisher efficiency.
A matching estimator releases the empirical mean of a nearly optimal bounded statistic with Gaussian noise and locally inverts its population moment map.
Our theory differs from classical efficiency theory in its set-valued information geometry and loss-dependent efficient estimator.

As examples, we provide closed-form constants and optimal procedures for various concrete models,
including one-dimensional regular families, Gaussian means, categorical probability vectors, and regression models among others.
The theory also transfers to Gaussian differential privacy through an exact parameter rescaling.
     \end{abstract}


\begingroup\renewcommand{\thefootnote}{}\footnotetext{The authors are listed in alphabetical order.}\endgroup
\end{frontmatter}

  \section{Introduction}
\label{sec:introduction}

Differential privacy (DP) provides a principled framework for statistical analysis while limiting the disclosure of individual records \citep{dwork2006_CalibratingNoise,dwork2014_AlgorithmicFoundations}.
Since the introduction of DP,
numerous private methods have been developed for a broad range of estimation problems.
In parallel, a large body of work has been devoted to establishing minimax optimal rates under privacy for specific problems, including mean estimation, regression, and distribution estimation \citep{duchi2018_MinimaxOptimal,acharya2020_DifferentiallyPrivate,cai2021_CostPrivacy}.
These results characterize the order of the privacy cost, but these rate bounds do not determine the exact leading constant.

A central statistical question remains unresolved: how to characterize \emph{efficiency} under DP, namely the \emph{minimax optimal constant}.
The key challenge is that privacy imposes a global constraint on the estimation procedure, not the statistical model itself, so the classical efficiency theory based on Fisher information does not simply apply.
This leads to our central questions: what information geometry under privacy governs efficiency, what exact local minimax constant does it imply, and can that constant be attained by an explicit private estimator?

We answer this question by developing a general efficiency theory for private parametric estimation.
Consider \(n\) independent observations from a regular parametric model indexed by \(\theta\in\Theta\subset\R^p\).
In this paper, we focus on central record-level zero-concentrated differential privacy (zCDP)
~\citep{dwork2016_ConcentratedDifferential,bun2016_ConcentratedDifferential} with privacy budget \(\rho_n\),
while our results also extend directly to Gaussian DP~\citep{dong2022_GaussianDifferential}.
Beginning with the high privacy regime in which the privacy error dominates,
we identify the exact asymptotic minimax constant over the entire class of private mechanisms and construct optimal estimators that attain it.

We find that the information geometry governing private estimation is fundamentally different from ordinary Fisher information.
For example, consider scalar Gaussian mean estimation, \(N(\theta,\sigma^2)\), whose ordinary Fisher information is
\(\E_\theta s_\theta(X)^2 = \sigma^{-2}\), where \(s_\theta(X)=(X-\theta)/\sigma^2\) is the score function.
Under zCDP, the optimal private mechanism is not a Gaussian release of the sample mean, and the corresponding local minimax risk is not determined by the Fisher information.
Instead, the optimal private mechanism is a noisy threshold count:
we first count \( C_n = \sum_{i=1}^n \ind{X_i \ge\theta_0} \) and privatize it as \( \widetilde{C}_n=C_n +Z_n\) for \(Z_n \sim N(0,1/(2\rho_n))\);
then project \(\widetilde{C}_n/n\) onto a fixed interval inside the image neighborhood of the mean map \(g_{\theta_0}(\theta)=P_\theta(X\ge\theta_0)\) and apply its local inverse.
It can be shown that the resulting private risk is governed by the squared \(L^1\) score quantity \(\xk{\E_{\theta_0} \abs{s_{\theta_0}(X)}/2}^2=(2\pi\sigma^2)^{-1}\), and this risk is in fact the optimal local zCDP constant for this model; see the scalar specialization of \cref{prop:gauss-mean}.
This example already shows that privacy changes the underlying local information geometry, not merely the effective sample size.

The preceding construction generalizes to a simple but fundamental class of private estimators,
which we call \emph{private moment-map inverse estimators}.
Let \(\Phi:\caX\to\R^p\) be a statistic with diameter at most one, that is, \(\diam(\Phi) \coloneqq \sup_{x,x'\in\caX} \norm{\Phi(x)-\Phi(x')}_2 \le1\).
Its empirical mean has replacement sensitivity at most \(1/n\), so the Gaussian mechanism \citep{dwork2016_ConcentratedDifferential,bun2016_ConcentratedDifferential} gives a \(\rho_n\)-zCDP release
\[
  Y_{\Phi,n}
  =
  \frac{1}{n}\sum_{i=1}^n \Phi(X_i)+G_n,
  \quad
  G_n \sim N\xkm{0,(2n^2 \rho_n)^{-1} I_p},
\]
If the population moment map \(g_\Phi(\theta)= \E_\theta \Phi(X)\) is continuously differentiable with nonsingular derivative at \(\theta_0\), the inverse function theorem gives a local inverse \(\Psi_\Phi\).
Fix a closed Euclidean ball \(\mathcal{B}_\Phi\) centered at \(g_\Phi(\theta_0)\) and compactly contained in the domain of \(\Psi_\Phi\), and let \(\Pi_{\mathcal{B}_\Phi}\) denote Euclidean projection onto it.
We estimate \(\theta\) by the projected local inversion
\[
  \widehat{\theta}_{\Phi,n}
  =
  \Psi_\Phi \left(\Pi_{\mathcal{B}_\Phi} Y_{\Phi,n} \right).
\]
We call \(\widehat{\theta}_{\Phi,n}\) the private moment-map inverse estimator based on \(\Phi\).
Consequently, the private noise \(G_n\) contributes to the error through the inverse of the relevant Jacobian
\[
  A_{\Phi,\theta_0}
  \coloneqq
  Dg_\Phi(\theta_0)
  =
  \E_{\theta_0} \zk{\Phi(X)s_{\theta_0}(X)^\top},
\]
or more precisely, through the induced positive semidefinite information matrix \(A_{\Phi,\theta_0}^*A_{\Phi,\theta_0}\),
where \(A_{\Phi,\theta_0}^*\) is the adjoint of \(A_{\Phi,\theta_0}\).

Motivated by this construction, we introduce the \emph{diameter-constrained information region}
\[
  \caJ_{\theta_0}
  =
  \cl
  \dkm{A_{\Phi,\theta_0}^*A_{\Phi,\theta_0}:
  \diam(\Phi)\le1},
\]
where \(\Phi\) ranges over diameter-one statistics.
We refer to \cref{def:diameter-score} for a precise definition.
Rather than selecting a single maximal information matrix, \(\caJ_{\theta_0}\) reflects the local information matrices generated by all such statistics.
Our minimax lower bound shows that the same region constrains the information of every zCDP output experiment, and hence the first-order risk of every zCDP estimator.
Its matrix-valued nature is essential in several dimensions, where the optimal statistic depends on the loss and on how information is allocated across parameter directions.

For a quadratic loss matrix \(W\succeq0\), define
\[
  \mfC_{\theta_0}(W)
  =
  \frac{1}{2}\inf_{J\in\caJ_{\theta_0}}\caI_W(J),
\]
where \(\caI_W(J)=\Tr(WJ^{-1})\) for nonsingular \(J\); the precise definition for all \(J\succeq0\) is given in \cref{def:diameter-score-constant}.
Under the privacy regime and the regularity conditions, our main theorem establishes the exact local minimax risk:
\begin{equation}
  \label{eq:intro-constant}
  \lim_{n\to\infty}
  n^2 \rho_n
  \inf_{\widehat{\theta} : \rho_n \text{-zCDP}}
  \sup_{\theta\in\Theta_n}
  \E_\theta\normx{\widehat{\theta}-\theta}_W^2
  =
  \mfC_{\theta_0}(W),
\end{equation}
where \( \Theta_n \) is a shrinking neighborhood of \( \theta_0 \) and the infimum is over all zCDP estimators.
The optimal constant is attained by private moment-map inverse estimators based on nearly optimal \( \Phi \) from the information region.
In addition, we show the asymptotic normality of the attaining estimator.
For \(W\succ0\), let \(J_W\) be the unique optimizing information matrix, and an attaining sequence can be chosen so that the law of \(\tau_n(\widehat{\theta}_n-\theta)\) converges uniformly over the local neighborhood to \(N(0,J_W^{-1}/2)\).

A distinctive feature of the private efficiency bound is that the optimizing information matrix, and hence the choice of an efficient diameter-one statistic, can depend on the loss matrix \(W\).
If \(J_1 \succeq J_2\), then \(\caI_W(J_1)\leq\caI_W(J_2)\) for every \(W\succeq0\).
However, \(\caJ_{\theta_0}\) need not possess a greatest element in the Loewner order, so there need not be a single feasible information matrix, or a single statistic \(\Phi\), that is optimal for every loss.
\Cref{prop:common-efficiency} shows that no single zCDP estimator sequence can attain both loss-specific constants in this case.
This contrasts with classical regular efficiency, where the same Fisher information matrix governs every quadratic loss and a single asymptotically efficient estimator---such as the maximum likelihood estimator---attains the efficiency bound for every fixed \(W\).

Moreover, our theory also characterizes efficiency across the different asymptotic regimes, rather than only the privacy-dominated regime.
\Cref{tab:intro-asymp-regimes} summarizes these regimes and their risk behavior.
When \(n^2\rho_n\to\infty\), a unified scale and mixed information region yield exact efficiency constants, retaining both the sampling and privacy covariance of the same statistic;
see \cref{thm:local-minimax,thm:compact-minimax}.
In particular, when the privacy budget is large, the efficiency constant is just the ordinary Fisher constant.
When the privacy cost is comparable to the statistical cost,
we characterize efficiency through a joint sampling--privacy variational constant in \cref{def:unified-info}.
On the other hand, in the strong privacy regime where \(n^2\rho_n=O(1)\), consistency is impossible;
see \cref{prop:strong-privacy-lower}.

\begin{table}[t]
  \caption{Asymptotic regimes and corresponding results under zCDP.}
  \label{tab:intro-asymp-regimes}
  \centering
  \scriptsize
  \setlength{\tabcolsep}{2pt}
  \renewcommand{\arraystretch}{1.7}
  \begin{tabular}{p{0.1\textwidth}p{0.17\textwidth}p{0.23\textwidth}p{0.2\textwidth}p{0.2\textwidth}}
    \hline
    & Strong privacy & {Privacy dominated} & Transition & {Sampling dominated} \\
    \hline
    Relation & \(n^2\rho_n=O(1)\) & \(n^2\rho_n\to\infty,\ n\rho_n\to0\) &
    \(n\rho_n\to\lambda\in(0,\infty)\),\par
    \(  \beta \coloneqq \lambda/(1+\lambda) \)
    & \(n\rho_n\to\infty\) \\
    Error
      & \(\geq c > 0\)
      & \( (n^2 \rho_n)^{-1} \mfC_{\theta_0}(W)\)
      & \(n^{-1}  \uCt(W)/\beta \)
      & \(n^{-1}\Tr\zkx{W I(\theta_0)^{-1}}\) \\
    \hline
  \end{tabular}
\end{table}

Our efficiency theory applies to arbitrary fixed-dimensional regular parametric models under zCDP and weighted quadratic losses.
As applications, we obtain exact constants in closed form and construct matching private estimators for several important classes of models.
These include one-dimensional regular families, Gaussian and Poisson means, categorical probability vectors, bounded mean estimation, and generalized linear regression under Gaussian design.
\Cref{appendix-sec:model-class-reductions,appendix-sec:concrete-model-constants} give a complete catalogue of further exact constants, including classical location and scale families, spherically symmetric location models, and additional regression settings.
These examples illustrate the broad applicability of our theory and the diversity of information geometries that arise under privacy.

Finally, our efficiency results extend exactly to \(\mu_n\)-Gaussian differential privacy (GDP) \citep{dong2022_GaussianDifferential}.
Upon setting \(\rho_n=\mu_n^2/2\), the same effective scale and optimal constants apply, since every \(\mu_n\)-GDP mechanism is \(\rho_n\)-zCDP, while the Gaussian releases attaining the zCDP bounds are themselves \(\mu_n\)-GDP.

\subsection{Related work}
\label{subsec:related-work}

The local minimax formulation follows the classical strategy of deriving an information bound and constructing a procedure that attains it.
Our lower bound uses a matrix van Trees inequality, following the multivariate Bayesian Cram\'er--Rao formulation of \citet{gill1995_ApplicationsVan}.
The new issue created by privacy is not the van Trees inequality itself, but the characterization of the Fisher information matrices attainable by a private output experiment.

\subsubsection{Efficiency and exact optimality under local DP}
\label{subsubsec:optimality-local-dp}

Local DP (LDP) requires each observation to be privatized before aggregation, so its mechanisms have a simpler structure than central DP mechanisms, which can process the entire sample jointly.
For regular parametric models under a fixed LDP level, \citet{steinberger2024_EfficiencyLocal} establishes local asymptotic mixed normality for arbitrary sequentially interactive mechanisms and derives convolution and local asymptotic minimax results.
For scalar parameters, the paper identifies the minimal asymptotic variance as the reciprocal of the maximal Fisher information over marginal LDP channels and shows that a two-step private MLE attains it.
For multidimensional parameters, the convolution and local asymptotic minimax bounds are stated in terms of the Fisher information of the privatized experiment, without identifying an optimal information matrix or optimal constant.

Within this scalar LDP framework, \citet{nikita2025_EfficientEstimation} solve the channel optimization in closed form for the unit-variance Gaussian location model when \(\epsilon\le 1.04\): randomized response applied to the sign of \(X-\theta\) yields Fisher information \(\frac{2}{\pi}\left(\frac{e^\epsilon-1}{e^\epsilon+1}\right)^2\), and a two-stage estimator attains the reciprocal variance.
However, their results are restricted to the Fisher information for the Gaussian model.
Recent work further studies extremal and staircase representations for Fisher information and privacy--utility optimization under LDP \citep{amorino2025_FactorizationExtremal,amorino2026_SymmetryStaircase}.

For a distribution on a fixed \(k\)-point alphabet under noninteractive \(\epsilon\)-LDP, \citet{ye2018_OptimalLocally} show that, for every \(1\le p\le 2\), the worst-case \(\ell_p^p\) risk of their scheme and estimator divided by the minimax risk converges to one as the sample size grows.
This exact result is confined to global worst-case risk for a fixed finite alphabet model with fixed \(\epsilon\) and noninteractive observation-wise privatization.

\subsubsection{Optimal rates and information contraction under privacy}
\label{subsubsec:optimal-rates-privacy}

A substantial literature studies how privacy changes minimax rates and develops lower-bound methods for establishing those rates.
Under local privacy, \citet{duchi2018_MinimaxOptimal} derive private divergence bounds and rate-optimal procedures for several estimation problems, while \citet{rohde2020_GeometrizingRates} characterize private rates through moduli of continuity and construct rate-optimal channels.
At the instance-specific level, \citet{duchi2024_RightComplexity} show that LDP local minimax risk is governed by a total-variation modulus and define, for one-dimensional regular models, the \(L^1\)-information \(J_\theta=\E_\theta|s_\theta(X)|\), equal to \(2a_\theta\) in our scalar notation.
Their lower bound covers sequentially interactive locally R\'enyi-private channels, and matching constructions recover this local risk scale in several settings, but the comparison is only up to universal numerical constants rather than an exact leading constant.
Under central DP, \citet{smith2011_PrivacyPreserving} constructs private versions of a broad class of estimators that retain their classical asymptotic distributions, and \citet{awan2024_OneStepEfficient} preserve efficient summary statistics in differentially private synthetic data.
For minimax risk under central DP, \citet{cai2021_CostPrivacy} establish matching rates for mean estimation and linear regression in low- and high-dimensional settings, and \citet{acharya2020_DifferentiallyPrivate} develop private analogues of Le Cam, Fano, and Assouad methods.
Building on tracing attacks, \citet{cai2023_ScoreAttack} use model scores to construct a general attack statistic and obtain privacy-constrained minimax lower bounds and optimal rates up to logarithmic factors for several parametric and nonparametric models.
These results characterize the order of the privacy cost, or provide general tools for doing so, without by themselves identifying a loss-specific leading local constant.

Fisher information contraction offers a more local route to private lower bounds.
\citet{barnes2019_FisherInformation} study Fisher information under blackboard communication protocols, and \citet{barnes2020_FisherInformation} derive Fisher information bounds under local differential privacy.
Under central zCDP, \citet{cai2026_MinimaxAdaptive} prove a trace contraction for the Fisher information of any \(\rho\)-zCDP statistic and combine it with a multivariate van Trees inequality.
However, their contraction does not yield a sharp constant and fails to capture how the optimal constant depends on the loss.

\subsubsection{Optimal constants and fixed-query mechanisms}
\label{subsubsec:optimal-constants-privacy}

A related line of work fixes a query and optimizes its release mechanism.
For count queries, \citet{ghosh2009_UniversallyUtilitymaximizing} establish the universal optimality of the geometric mechanism for their class of users and losses.
For real-valued queries with bounded sensitivity under pure differential privacy, \citet{geng2014_OptimalMechanism} characterize optimal staircase noise.
\citet{geng2019_OptimalNoiseAdding} study optimal noise addition under additive differential privacy, and \citet{kulesza2025_GeneralStaircasea} develop general staircase mechanisms.
Under \(f\)-DP, \citet{awan2024_OptimizingNoise} optimize additive noise through anti-concentration and stochastic dominance.
These formulations optimize a release channel after the query and its sensitivity geometry have been specified.
A population estimation problem additionally requires choosing which statistic should encode the model score and proving a lower bound over all admissible mechanisms.

\begingroup
\emergencystretch=1em
Recent work has used staircase mechanism arguments to claim exact constants for one-dimensional mean estimation and for variance and covariance estimation \citep{kulesza2024_MeanEstimationa,takakura2025_OptimalVariance} under high privacy.
However, fixed-query staircase characterizations do not by themselves yield lower bounds over arbitrary central DP mechanisms.
\par\endgroup

\subsection{Organization}
\label{subsec:organization}

\Cref{sec:efficiency-under-zcdp} formalizes the local zCDP experiment, introduces the diameter-constrained and mixed information regions, and states the privacy-dominated and unified local and compact efficiency theorems.
\Cref{sec:attain-bound-gauss-releases} constructs estimators that attain the bound, and \cref{sec:lower-bounds-all-zcdp-mechanisms} proves the matching lower bound over all zCDP mechanisms by combining uniform information contraction with a matrix local minimax argument.
\Cref{sec:examples} develops closed-form constants for several model classes.
\Cref{sec:extensions} treats Gaussian differential privacy and differentiable targets, and \cref{sec:conclusion} summarizes the implications, limitations, and remaining questions.
The appendices contain the complete technical proofs and further example calculations.

Write \(\bbS^p\) for the space of real symmetric \(p\times p\) matrices and \(\bbS_+^p\) for its positive semidefinite cone.
Use \(\preceq\) and \( \prec \) to denote the Loewner order on \(\bbS^p\).
Denote the \(k\times k\) identity matrix by \(I_k\).
Write \(A^\top\) for the transpose of a matrix and \(A^*\) for the adjoint of a linear operator.
For any matrix \(A\), \(A^\dagger\) denotes its Moore--Penrose inverse.
We use \(\Tr\) for the trace and \(\norm{\cdot}_2\) for the Euclidean norm, with other norms indicated by their subscripts.
For \(W\succeq0\), write \(\norm{v}_W=(v^\top Wv)^{1/2}\) for the induced seminorm.
   \section{Efficiency under zCDP}
\label{sec:efficiency-under-zcdp}

\subsection{Setup}
\label{subsec:setup}

Let \(\{P_\theta:\theta\in\Theta\}\) be a parametric model on a standard Borel sample space \(\caX\),
where \(\Theta\subset\R^p\) is open and \(p\) is fixed.
Suppose that the model is dominated by a common \(\sigma\)-finite measure \(\mu\), with densities \(p_\theta\).
We have independent observations \(X_1,\ldots,X_n\) from \(P_\theta\).

Throughout this paper, we adopt the notion of \(\rho\)-zCDP from \citet{bun2016_ConcentratedDifferential}, which is a convenient relaxation of \((\epsilon,\delta)\)-DP.
We use the Rényi divergence of order \(\alpha>1\) to quantify the distinguishability of two distributions \(P\) and \(Q\):
\( D_\alpha(P\|Q) = \frac{1}{\alpha-1} \log \int \left(\frac{dP}{dQ}\right)^\alpha dQ \).
We say that two data sets \(S,S'\in\caX^n\) are adjacent if they differ in one coordinate, i.e., if there exists an index \(i\) such that \(S_j=S'_j\) for all \(j\ne i\) and \(S_i \ne S'_i\).
The following definition is from \citet{bun2016_ConcentratedDifferential,dwork2016_ConcentratedDifferential}.

\begin{definition}[\(\rho\)-zCDP]
  \label{def:zcdp}
  A randomized algorithm \(M:\caX^n \to\caZ\) is \(\rho\)-zCDP if, for every adjacent pair \(S,S'\in\caX^n\),
  \[
    D_\alpha(M(S)\|M(S'))\le \alpha\rho,\quad \forall \alpha > 1.
  \]
\end{definition}

The divergence formulation of zCDP is more convenient and elegant for our efficiency analysis than the probability formulation of
 \((\epsilon,\delta)\)-DP, though the latter is more common in the literature.
Moreover, it is well known that \(\rho\)-zCDP implies \((\epsilon,\delta)\)-DP for any \(\delta\in(0,1)\) with \(\epsilon=\rho+2\sqrt{\rho\log(1/\delta)}\), while \((\epsilon,0)\)-DP implies \((\epsilon^2/2)\)-zCDP~\citep{bun2016_ConcentratedDifferential}.
However, these translation formulas may not be tight, and the exact efficiency constants under \((\epsilon,\delta)\)-DP may differ from those under zCDP.
Gaussian differential privacy admits an exact transfer because every \(\mu\)-GDP mechanism is \((\mu^2/2)\)-zCDP and the Gaussian mechanisms used here are themselves \(\mu\)-GDP; see \cref{subsec:gauss-differential-privacy}.
We will discuss possible extensions of our results to \((\epsilon,\delta)\)-DP in \cref{sec:conclusion}.

We impose two basic regularity conditions on the statistical model.
The first is classical regularity assumption~\citep{vaart1998_AsymptoticStatistics} ensuring the existence of score functions and Fisher information.
The second ensures that the score functions are well-behaved in \(L^1\) so that they can be used to define our diameter-constrained information for private estimation.

\begin{assumption}[DQM regularity]
  \label{ass:regular}
  The model is differentiable in quadratic mean (DQM) at every \(\theta\in\Theta\) with score \(s_\theta\) for one observation satisfying
\[
    \E_\theta s_\theta(X)=0, \qquad \E_\theta \norm{s_\theta(X)}_2^2<\infty .
  \]
\end{assumption}

\begin{assumption}[\(L^1\) smoothness]
  \label{ass:score-density-l1}
  The density map \(\theta\mapsto p_\theta\) is continuously differentiable as an \(L^1(\mu)\)-valued map on \(\Theta\), with derivative
  \begin{math}
    \dot{p}_\theta(x)=p_\theta(x)s_\theta(x).
  \end{math}
\end{assumption}

\cref{ass:regular,ass:score-density-l1} are mild regularity conditions that hold for many parametric models.
In particular, both assumptions hold for regular natural exponential families with common support and an open natural parameter space.
Common examples include the Gaussian location family, Bernoulli family, and Cauchy location family among others.

\subsection{Diameter-constrained private information}
\label{subsec:diameter-constrained-info}

The classical theory of asymptotic efficiency is based on the Fisher information matrix
\[
  I(\theta) \coloneqq \E_\theta \zk{s_\theta(X)s_\theta(X)^\top},
\]
which defines the local information metric of the statistical model and quantifies the infinitesimal distinguishability of nearby distributions under full observation.
Under privacy, however, \(I(\theta)\) alone does not characterize the information attainable by the output experiment, because privacy restricts the class of statistics that can be released.
More concretely, if we use the Gaussian mechanism (see \cref{lem:gauss-composition}),
then we can only transmit information through a statistic with bounded diameter.
This motivates the following diameter-constrained information region.

Suppose \cref{ass:regular} holds, and fix \(\theta\in\Theta\).
For a real Hilbert space \(\caH\) and a statistic \(\Phi:\caX\to\caH\) satisfying \(\E_\theta \norm{\Phi(X)}_{\caH}^2<\infty\),
the diameter of \(\Phi\) is
\begin{equation}
  \label{eq:def-diameter}
  \diam(\Phi)\coloneqq
  \sup_{x,x'\in\caX} \norm{\Phi(x)-\Phi(x')}_{\caH}.
\end{equation}
We introduce the score--statistic covariance operator
\begin{equation}
  \label{eq:score-stat-operator}
  A_{\Phi,\theta}:\R^p \to\caH,
  \qquad
  A_{\Phi,\theta} h
  =\E_\theta \zk{\Phi(X)\ang{s_\theta(X),h}}
  \qquad h\in\R^p.
\end{equation}
The adjoint of \(A_{\Phi,\theta}\) is \(A_{\Phi,\theta}^*:\caH\to\R^p\), given by
\begin{math}
  A_{\Phi,\theta}^*v
  =
  \E_\theta \zk{s_\theta(X)\ang{\Phi(X),v}_{\caH}},~
   v\in\caH.
\end{math}

\begin{definition}
  \label{def:diameter-score}
  The diameter-constrained information region at \(\theta\) is
  \begin{equation}
    \label{eq:diameter-score-info-set}
    \caJ_\theta
    \coloneqq
    \cl\dk{A_{\Phi,\theta}^*A_{\Phi,\theta}:
    d<\infty,\ \Phi: \caX\to\R^d,\ \diam(\Phi)\le1
    }
    \subset\bbS_+^p,
  \end{equation}
  where the closure is taken in \(\bbS^p\).
  We identify the positive semidefinite cone \(\bbS_+^p\) with the set of self-adjoint positive semidefinite operators on \(\R^p\).
\end{definition}

We emphasize that \(\caJ_\theta\) is a set of positive semidefinite matrices, not a single number or matrix.
In comparison with the Fisher information, the geometry under privacy could be more complicated, and the information region may contain multiple extreme points, so a single matrix may not be sufficient.
The set \(\caJ_\theta\) characterizes the information attainable by any statistic with diameter at most one, and it is the fundamental object that governs the local minimax risk under zCDP.
We may reduce it by taking the trace to obtain a scalar quantity:
\begin{equation}
  \label{eq:diameter-score-coefficient}
  \Gamma_\theta
  \coloneqq
  \sup_{J\in\caJ_\theta} \Tr J
  =
  \sup_{d<\infty,\ \Phi:\caX\to\R^d,\ \diam(\Phi)\le1}
  \norm{A_{\Phi,\theta}}_{\HS}^2.
\end{equation}
We next provide some basic geometric properties of \(\caJ_\theta\).
See also \Cref{fig:info-region}.

\begin{figure}[tbp]
  \centering
  \includegraphics[width=0.9\textwidth]{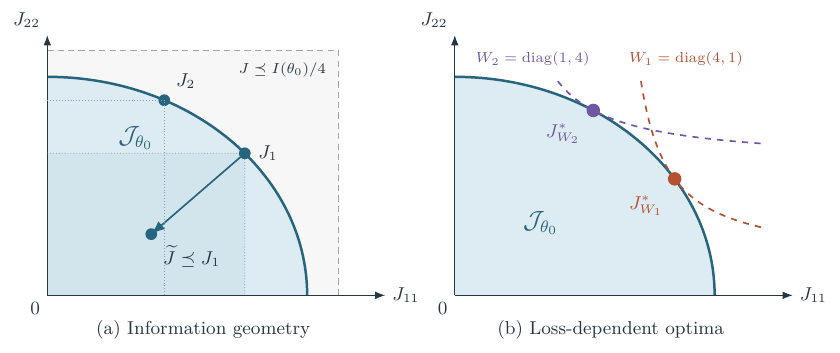}
  \caption{Illustration of the information region \( \caJ_{\theta_0} \).}
  \label{fig:info-region}
\end{figure}

\begin{proposition}
  \label{prop:dc-info-prop}
  Suppose \cref{ass:regular} holds, and fix \(\theta\in\Theta\).
  Then the following statements hold:
  \begin{enumerate}[label=(\roman*)]
    \item For every real Hilbert space \(\caH\), every statistic \(\Phi:\caX\to\caH\) with finite second moment, and every \(c\in\caH\),
    \begin{math}
      \diam(\Phi+c)=\diam(\Phi),
      ~
      A_{\Phi+c,\theta}=A_{\Phi,\theta}.
    \end{math}
    \item \(\caJ_\theta\) is a nonempty compact convex subset of \(\bbS_+^p\).
    \item The region \(\caJ_\theta\) is downward closed in the Loewner order: if \(J\in\caJ_\theta\) and \(0\preceq\widetilde{J}\preceq J\), then \(\widetilde{J}\in\caJ_\theta\).
    \item The region \(\caJ_\theta\) contains a positive definite matrix if and only if \(I(\theta)\succ0\).
  \end{enumerate}
\end{proposition}

The definition in \cref{eq:diameter-score-info-set} is characterized over all finite-dimensional statistics \( \Phi \),
but actually allowing arbitrary real Hilbert-valued statistics does not enlarge \(\caJ_\theta\).
Projection onto the range of \(A_{\Phi,\theta}\), which has dimension at most \(p\), preserves \(A_{\Phi,\theta}^*A_{\Phi,\theta}\) exactly while not increasing the diameter.

Moreover, the diameter-constrained information region is closely related to the local total variation distance between nearby distributions.
For probability measures \(P,Q\) dominated by \(\mu\), with densities \(p,q\), respectively, write
\( \TV(P,Q)
\coloneqq
\frac{1}{2}\int\abs{p-q}\,d\mu. \)

\begin{proposition}
  \label{prop:directional-envelope-tv}
  Suppose \cref{ass:regular,ass:score-density-l1} hold, and fix \(\theta\in\Theta\).
  Then, for every \(h\in\R^p\),
  \begin{equation}
    \label{eq:directional-info-envelope}
    \sup_{J\in\caJ_\theta} h^\top Jh
    =
    \frac{1}{4}
    \xk{
      \E_\theta \abs{\ang{s_\theta(X),h}}
    }^2.
  \end{equation}
  Moreover, as \(t\to0\),
  \begin{equation}
    \label{eq:local-tv-info}
    \TV(P_{\theta+th},P_\theta)
    =
    \frac{|t|}{2}
    \E_\theta \abs{\ang{s_\theta(X),h}}
    +o(|t|)
    =
    |t|
    \sqrt{\sup_{J\in\caJ_\theta} h^\top Jh}
    +o(|t|).
  \end{equation}
\end{proposition}

By \cref{eq:directional-info-envelope} and the Cauchy--Schwarz inequality,
\begin{equation}
  \label{eq:info-region-fisher-bound}
  J\preceq\frac{1}{4}I(\theta)
  \quad\text{for every }J\in\caJ_\theta,
  \qquad
  \Gamma_\theta \le\frac{1}{4}\Tr I(\theta).
\end{equation}
Thus the Fisher information provides a universal quadratic upper bound, whereas the exact directional envelope is governed by the \(L^1\) norm of the directional score and need not itself be a quadratic form.
In particular, different directions may require different score-sign statistics, which explains why the multivariate private information is naturally represented by a region rather than by a single maximal matrix.

When \( p=1 \), we have the following simple characterization of the information region.
\cref{ex:1d-info-region} shows that the diameter-constrained information region is related to the \(L^1\) norm of the score, rather than by the Fisher information \(\E_\theta s_\theta(X)^2\).
This shows that our diameter-constrained information is substantially different from the Fisher information.

\begin{example}[One-dimensional information region]
  \label{ex:1d-info-region}
  Under \(p=1\), we have
  \[
    \caJ_\theta=[0,\Gamma_\theta],
    \qquad
    \Gamma_\theta
    =
    \left(\frac{1}{2}\E_\theta \abs{s_\theta(X)}\right)^2,
  \]
  where the equality is attained by  \(\phi_*(x)=\frac{1}{2}\sign\zk{s_\theta(x)}\).
\end{example}

\subsection{Privacy-dominated local efficiency}
\label{subsec:privacy-dominated-local-efficiency}

Our first result focuses on the privacy-dominated regime, where the cost of privacy dominates the statistical error.
For a sequence of privacy budgets \(\rho_n\), write \(\tau_n \coloneqq n\sqrt{\rho_n}\).
For the privacy-dominated results, we work in the asymptotic regime where \(n\to\infty\) and
\begin{equation}
  \label{eq:privacy-dominated-regime}
  \rho_n \to0,\qquad \tau_n \to\infty,\qquad n\rho_n \to0.
\end{equation}
We recall that the optimal rates~\citep{cai2021_CostPrivacy} of convergence under zCDP are often \( n^{-1} + \tau_n^{-2} \).
Therefore,
the condition \(\rho_n \to0\) describes an increasingly stringent privacy constraint;
\(\tau_n \to\infty\) ensures consistency of the private estimator;
while \(n\rho_n \to0\) makes the statistical error negligible relative to the privacy noise.

We shall first consider a local estimation problem around a fixed center \(\theta_0 \in\Theta\).
Given any sequence \(R_n \to\infty\) with \(R_n/\tau_n \to0\), define
\begin{equation}
  \label{eq:local-param-set}
  \Theta_n = \left\{ \theta_0+\frac{u}{\tau_n}: \norm{u}_2 \le R_n \right\}.
\end{equation}
Because \(\Theta\) is open, \(\Theta_n \subset\Theta\) for all sufficiently large \(n\).

To state the theorem, we introduce a trace functional to accommodate singular matrices.
For \(W\succeq0\) and \(J\succeq0\), define
\begin{math}
   \caI_W(J) = \lim_{t\downarrow0} \Tr[W(J+tI_p)^{-1}].
\end{math}
If \(J\) is nonsingular, then \(\caI_W(J)=\Tr(WJ^{-1})\).
More generally, if \(\R^p=\ran(J)\oplus\ker(J)\), then \(\caI_W(J)<\infty\) precisely when \(W\) annihilates \(\ker(J)\), and \( \caI_W(J) \) is the trace of \(W\) against the inverse of \(J\) on \(\ran(J)\).
For each \(t>0\), the map \(J\mapsto\Tr[W(J+tI_p)^{-1}]\) is continuous on \(\bbS_+^p\), and these functions increase pointwise as \(t\downarrow0\).
Hence \(\caI_W\) is lower semicontinuous.

We can now define the exact efficiency constant.

\begin{definition}[Privacy constant]
  \label{def:diameter-score-constant}
  Define
  \begin{equation}
    \mfC_\theta(W) = \frac{1}{2}\inf_{J\in\caJ_\theta} \caI_W(J).
  \end{equation}
\end{definition}

By the compactness of \(\caJ_\theta\) and the lower semicontinuity of \(\caI_W\), the infimum is attained.
When \(I(\theta)\succ0\), it is finite because \(\caJ_\theta\) contains a positive definite matrix.
The following theorem identifies \(\mfC_{\theta_0}(W)\) as the exact efficiency bound for the private local experiment.

\begin{theorem}[Privacy-dominated constant]
  \label{thm:private-minimax}
  Fix \(W\succeq0\).
  Under \cref{ass:regular,ass:score-density-l1}, suppose \(I(\theta_0)\succ0\).
  Let \((\rho_n)\) satisfy the privacy-dominated regime in \cref{eq:privacy-dominated-regime}.
  Then, for every \(R_n \to\infty\) with \(R_n/\tau_n \to0\),
  \begin{equation}
    \label{eq:private-minimax}
    \lim_{n\to\infty} \tau_n^2 \inf_{\widehat{\theta}:\rho_n \text{-zCDP}} \sup_{\theta\in\Theta_n} \E_\theta \normx{\widehat{\theta}-\theta}_W^2 = \mfC_{\theta_0}(W),
  \end{equation}
  where \(\Theta_n\) is defined in \cref{eq:local-param-set}.
\end{theorem}

\subsection{Unified efficiency across privacy regimes}
\label{subsec:unified-local-efficiency}

The preceding theorem concerns the privacy-dominated regime, in which the sampling fluctuation is asymptotically negligible.
In the following, we establish a general efficiency theorem that covers different asymptotic regimes and unifies the efficiency constant.
Recall that the statistic error and privacy error are of scale \( n^{-1} \) and \( \tau_n^{-2} \) respectively.
For any positive sequence \((\rho_n)\), define the information ratio \(\lambda_n\), the positive local scale \(\sigma_n\), and the sampling weight \(\beta_n\) by
\begin{equation}
  \label{eq:unified-effective-info-scale}
  \lambda_n \coloneqq n\rho_n = \frac{\tau_n^2}{n},\quad
  \sigma_n^2 \coloneqq \frac1n+\frac1{\tau_n^2},\quad
  \beta_n \coloneqq \frac{\lambda_n}{1+\lambda_n}.
\end{equation}
Then, \(\sigma_n^{-2}\) is the effective information obtained by reciprocal addition of the sampling information \(n\) and the privacy information \(\tau_n^2=n^2\rho_n\).
The weights \(\beta_n\) and \(1-\beta_n\) are respectively the sampling and privacy shares of the common scale \(\sigma_n^2\).

For a statistic \(\Phi:\caX\to\R^d\), write \( V_{\Phi,\theta}=\Var_\theta(\Phi(X))  \).
For \(\beta\in[0,1]\), define
\begin{equation}
  \label{eq:weighted-variance}
  \Sigma_{\Phi,\theta,\beta}
  =
  \beta V_{\Phi,\theta}
  +
  \frac{1-\beta}{2}I_d.
\end{equation}
When \(\beta<1\), this matrix is positive definite.
The normalization in \cref{eq:unified-effective-info-scale} is chosen so that the sampling and Gaussian privacy covariances satisfy the exact identity
\( \sigma_n^2 \Sigma_{\Phi,\theta,\beta_n} =  V_{\Phi,\theta}/n + I_d/ (2\tau_n^2)\).

\begin{definition}[Unified information region and constant]
  \label{def:unified-info}
  For \(\theta\in\Theta\) and \(\beta\in[0,1]\), let
  \begin{equation}
    \label{eq:unified-info-region}
    \caU_{\theta,\beta}
    =
    \cl\dk{
      A_{\Phi,\theta}^*
      \Sigma_{\Phi,\theta,\beta}^{\dagger}
      A_{\Phi,\theta}:
      d<\infty,
      \ \Phi:\caX\to\R^d,
      \ \diam(\Phi)\le1
    }.
  \end{equation}
  The closure is taken in \(\bbS^p\).
  For \(W\succeq0\), define
  \begin{equation}
    \label{eq:unified-constant}
    \uC_\theta(W)
    =
    \inf_{H\in\caU_{\theta,\beta}}\caI_W(H).
  \end{equation}
\end{definition}

In the privacy and sampling regimes, respectively, the unified region recovers the private information region and Fisher information:
\begin{equation}
  \label{eq:unified-regime-identities}
  \caU_{\theta,0}=2\caJ_\theta,
  \qquad
  \mfC_\theta^{(0)}(W)=\mfC_\theta(W),
  \qquad
  \mfC_\theta^{(1)}(W)=\Tr\xkm{WI(\theta)^{-1}}.
\end{equation}
For \(0<\beta<1\), the covariance \(\Sigma_{\Phi,\theta,\beta}\) keeps the sampling and privacy contributions attached to the same statistic, so the constant remains a joint variational problem.

Define the unified local parameter set \( \widetilde\Theta_n \) for a sequence \( (R_n) \) with
\begin{equation}
  \label{eq:unified-local-param-set}
  \widetilde\Theta_n
  =
  \dk{
    \theta_0+\sigma_n u:
    \norm{u}_2\le R_n
  },
  \qquad
  R_n\to\infty,
  \quad
  R_n\sigma_n\to0.
\end{equation}

\begin{theorem}[Unified efficiency]
  \label{thm:local-minimax}
  Fix \(W\succeq0\).
  Under \cref{ass:regular,ass:score-density-l1}, suppose \(I(\theta_0)\succ0\).
  Let \((\rho_n)\) be a positive sequence such that
  \begin{equation}
    \label{eq:unified-asymp-regime}
    n^2\rho_n\to\infty,
    \qquad
    \beta_n\to\beta\in[0,1],
  \end{equation}
  where \(\beta_n\) is defined in \cref{eq:unified-effective-info-scale}.
  With \(\widetilde\Theta_n\) defined in
  \cref{eq:unified-local-param-set}, we have
  \begin{equation}
    \label{eq:unified-local-minimax}
    \lim_{n\to\infty}
    \sigma_n^{-2}
    \inf_{\widehat\theta:\rho_n\text{-zCDP}}
    \sup_{\theta\in\widetilde\Theta_n}
    \E_\theta
    \normx{\widehat\theta-\theta}_W^2
    =
    \uCt(W).
  \end{equation}
\end{theorem}

In particular, if the privacy constraint is weak, then we can achieve the optimal statistical error for free.

\begin{corollary}[Sampling-dominated regime]
  \label{cor:sampling-dominated-local-minimax}
  Under the conditions of \cref{thm:local-minimax}, suppose that
  \(n\rho_n\to\infty\).
  Then
  \begin{equation}
    \label{eq:sampling-dominated-local-minimax}
    \lim_{n\to\infty}
    n
    \inf_{\widehat\theta:\rho_n\text{-zCDP}}
    \sup_{\theta\in\widetilde\Theta_n}
    \E_\theta
    \normx{\widehat\theta-\theta}_W^2
    =
    \Tr\xkm{WI(\theta_0)^{-1}}.
  \end{equation}
\end{corollary}

For the intermediate case when the privacy error is of the same order as the statistical error,
the optimal constant is characterized by the variational constant \cref{eq:unified-constant} that reflects the delicate combination of two error sources,
which is not a simple weighted average in general.

\begin{corollary}[Transition regime]
  \label{cor:trans-local-minimax}
  Under the conditions of \cref{thm:local-minimax}, suppose
  \(n\rho_n\to\lambda\in(0,\infty)\), and put
  \(\beta=\lambda/(1+\lambda)\).
  Then
  \begin{equation}
    \lim_{n\to\infty}
    n
    \inf_{\widehat\theta:\rho_n\text{-zCDP}}
    \sup_{\theta\in\widetilde\Theta_n}
    \E_\theta
    \normx{\widehat\theta-\theta}_W^2
    =
    \frac{1}{\beta}
    \uCt(W).
  \end{equation}
\end{corollary}

Finally, under strong privacy, the minimax quadratic risk over a fixed parameter set with nonempty interior remains bounded away from zero.
\begin{proposition}[Impossibility under strong privacy]
  \label{prop:strong-privacy-lower}
  Under \cref{ass:regular,ass:score-density-l1}, fix \(W\succeq0\) with \(W\ne0\), and suppose  \(\Theta\) has  nonempty interior.
  If \(n^2\rho_n=O(1)\), then there exists a constant \(c>0\) such that for all sufficiently large \(n\),
  \begin{equation}
    \inf_{\widehat{\theta}:\rho_n \text{-zCDP}}
    \sup_{\theta\in \Theta}
    \E_\theta
    \normx{\widehat{\theta}-\theta}_W^2
    \ge
    c.
  \end{equation}
\end{proposition}

\subsection{Unified efficiency on compact parameter sets}
\label{subsec:unified-compact-efficiency}

The preceding local theorem treats the center \(\theta_0\) as public and known.
On a fixed compact parameter set, the same local constant can be attained without a public center by first releasing a private preliminary estimator and then selecting one of finitely many local Gaussian constructions.
Because \(\uC_\theta(W)\) may vary with \(\theta\), our uniform statement normalizes the risk pointwise by this local constant.

\begin{theorem}[Unified efficiency on a compact parameter set]
  \label{thm:compact-minimax}
  Fix \(W\succeq0\) with \(W\ne0\), and let \(K\subset\Theta\) be a nonempty compact set satisfying \(K=\overline{K^\circ}\).
  Under \cref{ass:regular,ass:score-density-l1}, suppose \(I(\theta)\succ0\) for every \(\theta\in K\) and  \(\theta\mapsto I(\theta)\) is continuous on \(K\).
  Assume the model is identifiable on \(K\), in the sense that \(P_{\theta_1} \ne P_{\theta_2}\) whenever \(\theta_1,\theta_2\in K\) and \(\theta_1\ne\theta_2\).
  Let \((\rho_n)\) be a positive sequence satisfying the asymptotic regime in
  \cref{eq:unified-asymp-regime},
  where \(\beta_n\) is defined in \cref{eq:unified-effective-info-scale}.
  Then \(\theta\mapsto\uC_\theta(W)\) is continuous and takes values in \((0,\infty)\) on \(K\), and
  \begin{equation}
    \label{eq:compact-relative-minimax}
    \lim_{n\to\infty}
    \inf_{\widehat{\theta}:\rho_n \text{-zCDP}}
    \sup_{\theta\in K}
    \frac{
      \sigma_n^{-2} \E_\theta\normx{\widehat{\theta}-\theta}_W^2
    }{
      \uC_\theta(W)
    }
    =1.
  \end{equation}
\end{theorem}

\begin{corollary}
  \label{cor:compact-minimax}
  Under the assumptions of \cref{thm:compact-minimax},
  \[
    \lim_{n\to\infty}
    \sigma_n^{-2}
    \inf_{\widehat{\theta}:\rho_n \text{-zCDP}}
    \sup_{\theta\in K}
    \E_\theta\normx{\widehat{\theta}-\theta}_W^2
    =
    \max_{\theta\in K} \uC_\theta(W).
  \]
\end{corollary}

On the other hand, if the parameter space is unbounded, then the minimax lower bound is infinite under privacy.

\begin{proposition}[Impossibility for unbounded parameter]
  \label{prop:unbounded-impossibility}
  Fix \(W\succeq0\), and suppose
  \[
    \sup_{\theta,\vartheta\in\Theta}
    \norm{\theta-\vartheta}_W^2
    =\infty.
  \]
  Then, for every \(n\ge1\) and every finite \(\rho\ge0\),
  \[
    \inf_{\widehat{\theta}:\rho \text{-zCDP}}
    \sup_{\theta\in\Theta}
    \E_\theta
    \normx{\widehat{\theta}-\theta}_W^2
    =\infty.
  \]
\end{proposition}

\subsection{Discussion}
\label{subsec:discussion}

We conclude this section with several implications of the information region and the unified local and compact efficiency theorems.

\subsubsection{Privacy regimes and optimal constants}

\cref{thm:private-minimax,thm:local-minimax,thm:compact-minimax} identify the exact minimax constant for private estimation under zCDP in their respective local and compact settings.
For every quadratic loss with \(W\succeq0\), the privacy-dominated local theorem shows that the risk, normalized by \(\tau_n^2\), converges to the variational constant \(\mfC_{\theta_0}(W)\).
The unified local theorem replaces \(\tau_n^2\) by \(\sigma_n^{-2}\) and connects the privacy-dominated, transition, and Fisher constants, while the compact theorem attains the corresponding regime-dependent local constant uniformly over a compact parameter set.
Because the minimization ranges over all \(\rho_n\)-zCDP mechanisms and all estimators based on their outputs, these results characterize a fundamental limit rather than the performance of a particular mechanism class.

The two compound scales in the theory play different roles (\Cref{tab:intro-asymp-regimes}).
The privacy-induced information scale is \(\tau_n^2=n^2\rho_n\) and the ratio \(\lambda_n=n\rho_n=\tau_n^2/n\) compares privacy and sampling information and determines their first-order balance.
Under strong privacy \(\tau_n^2 =O(1)\), \cref{prop:strong-privacy-lower} gives a nonvanishing lower bound.
When \(\tau_n^2\to\infty\) and \(\lambda_n\to0\) (\(\beta=0\)), estimation is consistent but privacy dominated.
When \(\lambda_n \to\infty\) (\( \beta=1 \)), privacy is negligible on the classical root-\(n\) scale and Fisher efficiency is recovered.
When \(\lambda_n\to\lambda\in(0,\infty)\) (\(\beta\in(0,1)\)), sampling and privacy contribute on the same scale, and their exact constant is a joint variational quantity.

Our results apply broadly across regular fixed-dimensional parametric models.
Nevertheless, the optimal variational constant \(\uC_\theta(W)\) may be difficult to compute in general because characterizing the region \(\caU_{\theta,\beta}\) requires optimization over bounded statistics.
In \cref{sec:examples}, we investigate several concrete models and give closed-form expressions for \(\uC_\theta(W)\).

\subsubsection{Set-valued private information geometry}

The variational expression for \(\mfC_\theta(W)\) also explains why the relevant private information object is a region rather than a single matrix.
Privacy does not simply reduce Fisher information uniformly across directions, but it creates trade-offs among the information attainable in different directions.
Indeed, whereas classical Fisher information describes all local directions through the single quadratic form \(h\mapsto h^\top I(\theta)h\), \cref{prop:directional-envelope-tv} identifies \(\sup_{J\in\caJ_\theta} h^\top Jh\) with the squared local total variation speed.
The squared directional speed need not be quadratic in \(h\), so no single feasible matrix need attain the supremum in every direction.
A maximizing score-sign statistic may consequently depend on \(h\), and statistics that are separately optimal in different directions need not be compatible under a common diameter constraint.
The region \(\caJ_\theta\) records precisely this compatibility by specifying which combinations of directional information can be attained jointly.

Our private information geometry is closely related to the \(L^1\)-information studied by \citet{duchi2024_RightComplexity} in local DP estimation.
Along the one-dimensional submodel \(t\mapsto\theta+th\), their \(L^1\)-information is \(\E_\theta \abs{\ang{s_\theta(X),h}}\), which is precisely the directional total variation speed in \cref{eq:local-tv-info}.
In particular, when \(p=1\), the region reduces to the interval in \cref{ex:1d-info-region}.
The distinction in several dimensions is that a local DP benchmark may treat directions separately, whereas our central zCDP problem requires a single mechanism to operate uniformly over a multidimensional local parameter set.
The region \(\caJ_\theta\) records which directional information levels are jointly attainable and thereby yields the exact minimax constant under a general quadratic loss.

The structural properties in \cref{prop:dc-info-prop} give this geometry an operational interpretation.
Combining feasible statistics in orthogonal components yields convexity, while discarding attainable information yields downward closure.
Consequently, only the upper Pareto boundary of \(\caJ_\theta\) can be relevant for efficiency.
Fisher information remains a useful outer envelope: \cref{eq:info-region-fisher-bound} gives \(J\preceq I(\theta)/4\) for every \(J\in\caJ_\theta\), and \(I(\theta)\succ0\) is equivalent to the existence of a positive definite feasible information matrix.
It determines whether nondegenerate private information is available and supplies a universal quadratic bound, but it does not generally determine the exact private efficiency constant.
In this sense, privacy changes the local information geometry rather than merely reducing the effective sample size.

The same set-valued interpretation extends to the unified information region \(\caU_{\theta,\beta}\).
Its elements record the information matrices jointly attainable when the same bounded statistic determines both the score--statistic covariance and the sampling covariance.
In the privacy and sampling regimes, \(\caU_{\theta,0}=2\caJ_\theta\), \(\mfC_\theta^{(0)}(W)=\mfC_\theta(W)\), and \(\mfC_\theta^{(1)}(W)=\Tr\xkm{WI(\theta)^{-1}}\), while for \(0<\beta<1\) the region retains both privacy and sampling contributions.
Thus, as the sampling weight changes, the unified theory changes the attainable information geometry itself rather than merely rescaling a fixed private region.

\subsubsection{Loss-dependent efficiency}

The efficiency constant in \cref{eq:unified-constant} is given by an optimization problem depending on \( W \),
and the minimizer (if it exists) need not be the same for every \( W \), particularly when \( \beta \in [0,1) \).
Consequently, there need not be a single statistic, and hence a single private estimator, that is always efficient.
This is in sharp contrast to the classical theory about efficiency, where a single MLE is efficient for every \( W \).
The following proposition formalizes this loss dependency and incompatibility.

\begin{proposition}[Loss-specific incompatibility]
  \label{prop:common-efficiency}
  Suppose \(I(\theta_0)\succ0\), fix \(\beta\in[0,1]\), and let \(W_1,W_2 \succeq0\).
  Write
  \begin{math}
    \mathcal{A}_{\theta_0,\beta}(W)
    =
    \argmin_{H\in\caU_{\theta_0,\beta}}\caI_W(H).
  \end{math}
  Then
  \begin{equation}
    \uCt(W_1+W_2)
    \ge
    \uCt(W_1)
    +
    \uCt(W_2),
  \end{equation}
  with equality if and only if
  \(\mathcal{A}_{\theta_0,\beta}(W_1)\cap\mathcal{A}_{\theta_0,\beta}(W_2)\ne\varnothing\).
  If these optimizer sets are disjoint, no single sequence of \(\rho_n\)-zCDP estimators can attain both local asymptotic minimax constants
  \(\uCt(W_1)\) and \(\uCt(W_2)\)
  along any sequence of privacy budgets satisfying
  \cref{eq:unified-asymp-regime}.
\end{proposition}

\subsubsection{Attainability and universality}

The variational problem directly determines asymptotically attaining procedures.
For a fixed \(\beta\), a statistic \(\Phi\) of diameter one generates the unified information matrix
\begin{math}
  H_{\Phi,\theta_0,\beta}
  =
  A_{\Phi,\theta_0}^*
  \Sigma_{\Phi,\theta_0,\beta}^{\dagger}
  A_{\Phi,\theta_0}.
\end{math}
Releasing its empirical mean with Gaussian noise and applying projected local inversion yields the inverse-information risk associated with this matrix.
The blockwise construction in
\cref{thm:local-attain-seq,thm:compact-attain-seq}
therefore turns this variational approximation into a single
asymptotically efficient estimator sequence in the local and compact settings.
For the transition and Fisher regimes, the proofs use the same projected inversion with the sampling covariance retained in the local risk calculation.
\Cref{thm:gauss-attain-limit} gives a local attaining sequence with uniform Gaussian limit \(N(0,H_{W,\beta}^{-1})\) for \(W\succ0\), where \(H_{W,\beta}\) denotes the unique optimizing information matrix.

Our minimax characterization establishes the first-order universality of Gaussian mechanisms for efficiency problems considered here.
Across different models, the exact minimax constant is attained by a Gaussian release of a suitably chosen bounded statistic, followed by moment-map inversion.
Given the minimax lower bound, it is unnecessary to resort to other nonadditive, nonlinear, or otherwise unrestricted private mechanisms for improvement.
The essential choice is therefore not a more elaborate noise distribution, but the statistic to which Gaussian noise is added.

\subsubsection{From a public center to uniform efficiency}

The local theorem treats \(\theta_0\) as public because the optimal statistic and its inverse mean map may depend on the center.
Under the assumptions of \cref{thm:compact-minimax}, including global identifiability on \(K\), the uniform theorem removes this requirement without first-order loss.
The bounded statistic \(T\) serves only to localize the unknown parameter, whereas the statistics used in the second stage approximate the loss-specific information optimum on finitely many local charts.
The statistic used to identify the chart therefore need not itself be efficient for final estimation.

The first-stage localization uses asymptotically negligible shares of the sample and privacy budget, so the second stage preserves the pointwise leading constant \(\uC_\theta(W)\) uniformly over \(K\).
The normalized theorem states this pointwise efficiency, while \cref{cor:compact-minimax} shows that the unnormalized compact minimax risk is governed by the least favorable local constant \(\max_{\theta\in K} \uC_\theta(W)\).
   \section{Optimal Estimator via Gaussian Release and Moment-Map Inversion}
\label{sec:attain-bound-gauss-releases}

This section develops the Gaussian mechanisms that attain the upper bounds in
\cref{thm:local-minimax,thm:compact-minimax}.
We first construct a moment-map inverse estimator from a fixed bounded statistic, and then choose the statistic to approach the optimal variational constant.
Finally, we use a private preliminary estimator to remove the assumption that the local center is publicly known.
We first recall two standard properties of \(\rho\)-zCDP
\citep{dwork2016_ConcentratedDifferential,bun2016_ConcentratedDifferential}.

\begin{lemma}
  \label{lem:gauss-composition}
  We have the following properties of \(\rho\)-zCDP.
  \begin{itemize}[leftmargin=*]
    \item \emph{Gaussian mechanism.}
    If \(f:\caX^n\to\R^d\) has replacement \(\ell_2\)-sensitivity
    \[
      \Delta
      =
      \sup_{S,S'\text{ adjacent}}
      \norm{f(S)-f(S')}_2,
    \]
    then \(f(S)+Z\), where
    \(Z\sim N(0,\Delta^2I_d/(2\rho))\), is \(\rho\)-zCDP.
    \item \emph{Adaptive composition.}
    If, conditional on any preceding outputs, the \(j\)th mechanism is
    \(\rho_j\)-zCDP for \(j=1,\ldots,k\), then their adaptive composition is
    \((\sum_{j=1}^k\rho_j)\)-zCDP.
  \end{itemize}
\end{lemma}

\subsection{Local upper bound for a bounded statistic}
\label{subsec:local-upper-bound-bounded-stat}

First, suppose we have a bounded statistic with a nonsingular mean derivative.
As introduced in \cref{sec:introduction}, this gives a private estimator by releasing its empirical mean and locally inverting its population moment map.
We recall the construction.
Fix \(\theta_0\in\Theta\), and let \(\Phi:\caX\to\R^p\) satisfy
\(\diam(\Phi)\le1\).
Suppose its mean map
\(g(\theta)=\E_\theta\Phi(X)\) is continuously differentiable near
\(\theta_0\), with nonsingular derivative
\(A_0=Dg(\theta_0)=A_{\Phi,\theta_0}\).
Let \(\Psi_0\) be the local \(C^1\) inverse of \(g\) on a small closed
Euclidean ball \(\mathcal B_0\) centered at \(g(\theta_0)\), and let
\(\Pi_{\mathcal B_0}\) denote Euclidean projection onto this ball.
Set \(F_0=\Psi_0\circ\Pi_{\mathcal B_0}\).
Define the Gaussian release and the corresponding moment-map inverse estimator by
\begin{equation}
  \label{eq:moment-map-inverse-estimator}
  \widehat\theta_{\Phi,n}
  =
  F_0(Y_n),
  \qquad
  Y_n
  =
  \frac1n\sum_{i=1}^n\Phi(X_i)+G_n,
  \qquad
  G_n
  \sim
  N\xkm{0,\frac{I_p}{2\tau_n^2}}.
\end{equation}

By \cref{lem:gauss-composition}, the estimator is
\(\rho_n\)-zCDP: the empirical mean has replacement sensitivity at most
\(1/n\), and the remaining operations are post-processing.
Its sampling covariance and Gaussian privacy covariance satisfy
\[
  \Var_\theta\xkm{Y_n-g(\theta)}
  =
  \frac{V_{\Phi,\theta}}n+\frac{I_p}{2\tau_n^2}
  =
  \sigma_n^2\Sigma_{\Phi,\theta,\beta_n}.
\]
Thus the covariance appearing in \cref{eq:weighted-variance} is
exactly the first-order covariance of the released moment on the
\(\sigma_n\)-scale.
The next theorem identifies the resulting local risk.

\begin{theorem}[Local upper bound]
  \label{thm:bounded-stat-upper}
  Under \cref{ass:regular,ass:score-density-l1}, suppose the statistic above
  has the stated continuously differentiable mean map and nonsingular
  derivative.
  Let \((\rho_n)\) satisfy the asymptotic regime in
  \cref{eq:unified-asymp-regime}.
  For every \((R_n)\) and \(\widetilde\Theta_n\) satisfying
  \cref{eq:unified-local-param-set},
  \(\widehat\theta_{\Phi,n}\) is \(\rho_n\)-zCDP and, for every
  \(W\succeq0\),
  \begin{equation}
    \label{eq:unified-fixed-stat-risk}
    \sup_{\theta\in\widetilde\Theta_n}
    \left|
      \sigma_n^{-2}
      \E_\theta
      \normx{\widehat\theta_{\Phi,n}-\theta}_W^2
      -
      \Tr\zk{
        WA_0^{-1}
        \Sigma_{\Phi,\theta_0,\beta}
        A_0^{-\top}
      }
    \right|
    \longrightarrow0.
  \end{equation}
\end{theorem}

The idea of \cref{thm:bounded-stat-upper} is simple.
After the local inverse transformation, the released moment is adjusted by
the derivative \(A_0^{-1}\) of the inverse map.
Consequently, the sampling and privacy fluctuations have first-order
covariance
\begin{math}
    A_0^{-1}
  \Sigma_{\Phi,\theta_0,\beta}
  A_0^{-\top}.
\end{math}
The projection makes the estimator globally defined without affecting this
local expansion.
The details are given in \cref{appendix-sec:attain-proofs}.

\subsection{Near-optimal statistics and local attainment}
\label{subsec:near-optimal-attainment}

Next, we show that the bound in
\cref{thm:bounded-stat-upper} can be made arbitrarily close to
\(\uCt(W)\) by choosing a nearly optimal statistic.

\begin{lemma}[Nearly optimal statistic]
  \label{lem:verify-near-optimal}
  Fix \(W\succeq0\), \(\theta_0\in\Theta\), and \(\beta\in[0,1]\).
  Under \cref{ass:regular,ass:score-density-l1}, suppose
  \(I(\theta_0)\succ0\).
  For every \(\epsilon>0\), there is a statistic
  \(\Phi_\epsilon:\caX\to\R^p\), with
  \(\diam(\Phi_\epsilon)\le1\), whose mean map is continuously
  differentiable near \(\theta_0\), whose derivative
  \(A_\epsilon=A_{\Phi_\epsilon,\theta_0}\) is nonsingular, and such that
  \begin{equation}
    \label{eq:near-optimal-stat}
    \Tr\zk{
      WA_\epsilon^{-1}
      \Sigma_{\Phi_\epsilon,\theta_0,\beta}
      A_\epsilon^{-\top}
    }
    \le
    \uCt(W)+\epsilon.
  \end{equation}
\end{lemma}

For \(\beta<1\), the argument perturbs an optimizing information matrix
toward a positive-definite feasible matrix, approximates the result by an
actual diameter-one statistic, and reduces that statistic to \(\R^p\).
At \(\beta=1\), the same conclusion follows from bounded truncated-score
statistics approaching the Fisher information limit.
The assumed \(L^1\) differentiability and the inverse function theorem then
provide the local inverse required in
\cref{thm:bounded-stat-upper}.
Details are given in \cref{appendix-sec:attain-proofs}.

Combining \cref{lem:verify-near-optimal} with
\cref{thm:bounded-stat-upper} and taking a diagonal sequence yields
the following result, which is the constructive upper half of
\cref{thm:local-minimax} under the same assumptions.

\begin{theorem}[Local optimal estimator]
  \label{thm:local-attain-seq}
  Fix \(W\succeq0\).
  Under \cref{ass:regular,ass:score-density-l1}, suppose
  \(I(\theta_0)\succ0\).
  Let \((\rho_n)\) satisfy the asymptotic regime in
  \cref{eq:unified-asymp-regime}.
  Then, for every \((R_n)\) and \(\widetilde\Theta_n\) satisfying
  \cref{eq:unified-local-param-set}, there is a
  sequence \((\widehat\theta_n^\star)\) of \(\rho_n\)-zCDP moment-map
  inverse estimators, each based on an \(\R^p\)-valued statistic of diameter
  at most one, such that
  \begin{equation}
    \label{eq:local-attain-seq}
    \limsup_{n\to\infty}
    \sigma_n^{-2}
    \sup_{\theta\in\widetilde\Theta_n}
    \E_\theta
    \normx{\widehat\theta_n^\star-\theta}_W^2
    \le
    \uCt(W).
  \end{equation}
\end{theorem}

The preceding risk bounds admit a distributional refinement: the
optimal sequence can be chosen to have a uniform Gaussian limit.

\begin{theorem}[Gaussian limit for attaining sequences]
  \label{thm:gauss-attain-limit}
  Under the assumptions of \cref{thm:local-attain-seq}, suppose
  \(W\succ0\) and take the unique minimizer
  \begin{math}
    H_{W,\beta}
    =
    \argmin_{H\in\caU_{\theta_0,\beta}}
    \caI_W(H).
  \end{math}
  The attaining sequence may be chosen so that
  \begin{equation}
    \label{eq:attain-seq-gauss-limit}
    \sigma_n^{-1}
    (\widehat\theta_n^\star-\theta)
    \rightsquigarrow
    N\xkm{0,H_{W,\beta}^{-1}}
    \qquad
    \text{uniformly over }\theta\in\widetilde\Theta_n.
  \end{equation}
\end{theorem}

This result parallels the asymptotic normality of efficient estimators in classical regular models, where the local limit has covariance \(I(\theta_0)^{-1}\) at the \(n^{-1/2}\) scale \citep[Chapter~8]{vaart1998_AsymptoticStatistics}.
In the sampling-dominated regime \(\beta=1\), we have \(\sigma_n^{-1}\sim\sqrt n\) and \(H_{W,1}=I(\theta_0)\), so \cref{eq:attain-seq-gauss-limit} recovers this classical limit.
For \(0\le\beta<1\), privacy changes both the effective scale and the attainable covariance: the optimizing information matrix \(H_{W,\beta}\) is selected from \(\caU_{\theta_0,\beta}\) and may depend on the loss \(W\).
For each fixed \(W\succ0\), the optimizing information matrix is unique: when \(\beta<1\), the region \(\caU_{\theta_0,\beta}\) is convex and \(H\mapsto\Tr(WH^{-1})\) is strictly convex on the positive definite cone, while at \(\beta=1\) the unique optimizer is \(I(\theta_0)\).
Different losses may therefore select different unique optimizing matrices, although statistics generating the same matrix need not be unique.

\subsection{Preliminary localization on compact parameter sets}
\label{subsec:preliminary-localization}

The local constructions above require the public center \(\theta_0\) to be
known.
In this subsection, we show that a private preliminary estimator can select
both the local center and the local statistic while using only vanishing
fractions of the sample and privacy budget.

The following lemma shows that, under the remaining regularity and nonsingularity conditions in \cref{thm:compact-minimax},
global identifiability supplies the bounded finite-dimensional moment identifier needed for preliminary localization.

\begin{lemma}[Bounded identifier]
  \label{lem:bounded-identifier}
  Let \(K\subset\Theta\) be nonempty and compact.
  Under \cref{ass:regular,ass:score-density-l1}, suppose
  \(I(\theta)\succ0\) for every \(\theta\in K\), and suppose the model is
  identifiable on \(K\).
  Then there are a finite integer \(m\) and a bounded measurable statistic
  \(T:\caX\to\R^m\), which may be chosen with \(\diam(T)\le1\), such that
  \[
    h(\theta)=\E_\theta T(X)
  \]
  is continuously differentiable on \(\Theta\) and one-to-one on \(K\).
\end{lemma}

Let \(K\subset\Theta\) be nonempty and compact, and let
\(T:\caX\to\R^m\) have finite diameter
\(\Delta_T=\diam(T)\), and suppose its mean map \(h\) is continuous and
one-to-one on \(K\).
Under the assumptions of \cref{thm:compact-minimax}, such a statistic is
supplied by \cref{lem:bounded-identifier}.
For \(n_1\) observations and privacy budget \(\rho_1\), define the
preliminary localization estimator \(\widetilde\theta\) by
\begin{equation}
  \label{eq:private-localization-estimator}
  \begin{aligned}
    Y_1
    &=
    \frac1{n_1}\sum_{i=1}^{n_1}T(X_i)+G_{1,n},
    \qquad
    G_{1,n}
    \sim
    N\xkm{0,\frac{\Delta_T^2I_m}{2n_1^2\rho_1}},\\
    \widetilde\theta
    &\in
    \argmin_{\vartheta\in K}
    \norm{Y_1-h(\vartheta)}_2.
  \end{aligned}
\end{equation}

The next lemma shows that this estimator attains the error probability needed
by the second stage while using only a small share of the available
resources.

\begin{lemma}[Private preliminary localization]
  \label{lem:private-localization}
  The estimator in \cref{eq:private-localization-estimator} is
  \(\rho_1\)-zCDP.
  Let \(a_n\to\infty\).
  If \(n_1/\log a_n\to\infty\) and
  \(n_1^2\rho_1/\log a_n\to\infty\), then, for every fixed \(r>0\),
  \[
    \sup_{\theta\in K}
    P_\theta\dk{
      \normx{\widetilde\theta-\theta}_2>r
    }
    =
    o(a_n^{-2}).
  \]
\end{lemma}

To pass from preliminary localization to compact attainment, take
\(a_n=\sigma_n^{-1}\).
Compactness yields finitely many local charts, each equipped with a
diameter-one statistic whose Gaussian release and projected local inverse
are nearly optimal for \(\uC_\theta(W)\) throughout the chart.
We spend a vanishing fraction of the sample and privacy budget on the
preliminary estimator, use it to select a valid chart, and apply the
corresponding Gaussian mechanism to the remaining observations.
Adaptive composition preserves the total privacy budget, while uniform
localization and the negligible first stage leave the leading risk unchanged
on the \(\sigma_n^2\)-scale.

\begin{theorem}[Uniform optimal estimator]
  \label{thm:compact-attain-seq}
  Under the assumptions of \cref{thm:compact-minimax}, there is a
  sequence \((\widehat\theta_n^\star)\) of \(\rho_n\)-zCDP estimators,
  eventually obtained by the two-stage Gaussian construction described
  above, such that
  \[
    \limsup_{n\to\infty}
    \sup_{\theta\in K}
    \frac{
      \sigma_n^{-2}
      \E_\theta
      \normx{\widehat\theta_n^\star-\theta}_W^2
    }{
      \uC_\theta(W)
    }
    \le1.
  \]
\end{theorem}
   \section{Lower Bounds over All zCDP Mechanisms}
\label{sec:lower-bounds-all-zcdp-mechanisms}

In this section, we establish optimality through minimax lower bounds over all zCDP procedures.
For a mechanism \(M\), let \(Q_\theta^M=M\circ P_\theta^{\otimes n}\) denote the output distribution, and let \(I_M(\theta)\) denote the Fisher information matrix of the output experiment.
Restricting attention to standard Borel outputs entails no loss, since composing any transcript with its final \(\R^p\)-valued estimator preserves zCDP by post-processing.

\subsection{Information-region contraction}
\label{subsec:info-region-contraction}

The main idea for the lower bound is to show that an arbitrary private mechanism cannot carry more local information than the information region in \cref{def:diameter-score}.
We have the following theorem.

\begin{theorem}[Information-region contraction under zCDP]
  \label{thm:diameter-contraction}
  Suppose \cref{ass:regular} holds.
  Let \(M\) be an arbitrary \(\rho\)-zCDP Markov kernel from \(\caX^n\) into a standard Borel output space \(\caZ\).
  Then there is a scalar \(\delta_\rho=1+o_\rho(1)\) as \(\rho\downarrow0\), depending only on \(\rho\), such that for every such \(M\) and every \(\theta\in\Theta\), there exists \(J_{M,\theta,\rho} \in\caJ_\theta\) with
  \begin{equation}
    \label{eq:fisher-matrix-contraction}
    I_M(\theta) = 2\rho n^2 \delta_\rho^2 J_{M,\theta,\rho}.
  \end{equation}
  The \(o_\rho(1)\) term is uniform over mechanisms and over parameter values for which the assumptions hold.
\end{theorem}

Taking the trace and using \cref{eq:info-region-fisher-bound} yields the following corollary.

\begin{corollary}
  \label{cor:fisher-trace-contraction}
  Under the same settings as in \cref{thm:diameter-contraction}, we have
  \begin{equation}
    \label{eq:fisher-trace-contraction}
    \Tr I_M(\theta)
    \le 2\rho n^2 \Gamma_\theta(1+o_\rho(1))
    \le \frac{1}{2}\rho n^2\Tr I(\theta)(1+o_\rho(1)).
  \end{equation}
\end{corollary}

\cref{thm:diameter-contraction} shows that the Fisher information of a \(\rho\)-zCDP mechanism is asymptotically restricted to \( \caJ_\theta \) after normalization.
In comparison to the rough upper bound in the private information contraction lemma in \citet{cai2026_MinimaxAdaptive},
\cref{thm:diameter-contraction} is more refined to capture the constant dependency.
In particular, the last inequality in \cref{cor:fisher-trace-contraction} recovers their private Fisher information contraction.
Moreover, the matrix representation is stronger than its trace consequence and retains the anisotropic geometry needed for general quadratic loss.

To prove \cref{thm:diameter-contraction}, the fingerprinting attack or score attack~\citep{cai2023_ScoreAttack} commonly used in the literature no longer works, as they are too coarse for our purpose.
We develop the following lemma to transfer \( \rho \)-zCDP condition to the boundedness of the likelihood ratio as \( L^2 \)-valued statistics.

\begin{lemma}[Barycentric likelihood ratio diameter]
  \label{lem:barycentric}
  Let \(\{R_t:t\in T\}\) be a probability kernel and let \(\lambda\) be a probability measure on \(T\).
  Suppose the likelihood ratios admit jointly measurable versions and \(D_\alpha(R_t\|R_s)\le \alpha\rho\) for all \(s,t\in T\) and \(\alpha>1\).
  If \(\bar{R}=\int R_t \lambda(dt)\) and \(r_t=dR_t/d\bar{R}\), then, uniformly over the family and mixing measure,
  \begin{equation}
     \sup_{s,t\in T} \norm{r_t-r_s}_{L^2(\bar{R})}^2 \le 2\rho(1+o_\rho(1)).
  \end{equation}
\end{lemma}

For each record \(i\), let \(Q_{i,x}\) be the output law obtained by fixing that record at \(x\) and averaging over the remaining observations.
Then \(Q_\theta^M=\int Q_{i,x} P_\theta(dx)\) is the barycenter of this conditional output family.
Joint convexity transfers the zCDP bounds for neighboring deterministic data sets to pairwise Rényi bounds for \(Q_{i,x}\) and \(Q_{i,x'}\).
Thus \cref{lem:barycentric}, applied with \(L_{i,x}=dQ_{i,x}/dQ_\theta^M\), shows that the Hilbert-valued statistic \( \Phi_i(x) = L_{i,x}/(\sqrt{2\rho}\delta_\rho)\) has diameter at most one in \(L^2(Q_\theta^M)\), so is their average \( \bar{\Phi}(x) \).
Then, a finite dimensional compression yields an at most \(p\)-dimensional statistic with no larger diameter and the same associated information matrix \(A_{\bar\Phi,\theta}^*A_{\bar\Phi,\theta}\),
so this matrix must be contained in \( \caJ_\theta \) by its definition.
On the other hand, we can establish identity between the Fisher information matrix \( I_M(\theta) \) and the information matrix of \( L_{i,x} \).
Plugging in the scaling concludes \cref{thm:diameter-contraction}.
The full proof is given in \cref{appendix-subsec:contraction}.

\subsection{Local minimax lower bound under privacy regime}
\label{subsec:local-minimax-lower-bound}

With the local information contraction, we can use a van Trees argument to obtain the local minimax lower bound.

\begin{theorem}[Minimax lower bound]
  \label{thm:matrix-local-lower}
  Fix \(\theta_0\in\Theta\) and \(W\succeq0\).
  Under \cref{ass:regular,ass:score-density-l1}, let \((\rho_n)\) satisfy the privacy regime in \cref{eq:privacy-dominated-regime}.
  Then, with \(\Theta_n\) defined in \cref{eq:local-param-set},
  \[
    \liminf_{n\to\infty}
    \tau_n^2
    \inf_{\widehat{\theta}:\rho_n \text{-zCDP}}
    \sup_{\theta\in\Theta_n}
    \E_\theta
    \normx{\widehat{\theta}-\theta}_W^2
    \ge
    \mfC_{\theta_0}(W).
  \]
\end{theorem}

The proof of \cref{thm:matrix-local-lower} is given in \cref{appendix-subsec:matrix-local-lower}.
Together with the local optimal estimator in \cref{thm:local-attain-seq}, \cref{thm:matrix-local-lower} completes the proof of the local efficiency theorem.
The same lower bound, applied inside compact parameter sets and combined with \cref{thm:compact-attain-seq}, proves the compact efficiency theorem in the privacy-dominated regime.

\subsection{Transition regime}
\label{subsec:trans-regime-lower-bound}

The privacy contraction above fails to give the sharp constant in the transition regime where the sampling fluctuation is not negligible.
A more refined analysis is needed to characterize the sampling and privacy contributions jointly for every mechanism.

To this end, we first consider the \(k\)-category model, for which the information can be computed more explicitly.
A key observation is that symmetrization reduces an arbitrary mechanism to a channel \(m\mapsto Q_m\) from the multinomial count vector, without changing either its output law under iid sampling or its zCDP guarantee.
The output information then satisfies an identity for missing information: the full multinomial information is reduced by a term determined by the conditional covariance of the counts given the mechanism output.
To control this loss sharply, we compare the channel at neighboring count vectors \(m\) and \(m+e_a-e_k\).
The Gram matrix of the resulting shift scores separates, to first order, into a sampling contribution \(V_q^{-1}/n\) from the multinomial count law and a privacy contribution \(2\rho_nJ_{n,q}\), where \(V_q\) is the covariance matrix of the first \(d = k-1\) category indicators and \(J_{n,q}\) is the Gram matrix of a configuration with pairwise distances at most one.
A matrix projection inequality converts this decomposition into an upper bound on the output Fisher information.
Consequently, the sampling and privacy contributions, though of the same order, can be decomposed separately through
\begin{equation}
  \label{eq:trans-joint-info-bound}
  \frac{1}{n}I_{M_n}(q)
  \preceq
  V_q^{-1}
  -V_q^{-1}
  \xk{V_q^{-1}+2\lambda_nJ_{n,q}}^{-1}
  V_q^{-1}
  +o(1)I_d,
  \qquad
  \lambda_n=n\rho_n.
\end{equation}

For the general regular model, we freeze the conditional distribution within each cell of a finite measurable partition and vary only the cell probabilities.
The \(L^1\)-smoothness assumption permits a partition whose cell score approximates the full score arbitrarily well.
The bound on a finite alphabet then transfers uniformly over all mechanisms: for every \(\epsilon>0\), a shrinking prior \(\pi_n\) supported on the local parameter set satisfies, for a finite constant \(C_{\theta_0,\lambda}\) independent of the mechanism,
\begin{equation}
  \label{eq:trans-avg-info-bound}
  \frac{1}{n}\int I_{M_n}(\theta)\,\pi_n(d\theta)
  \preceq
  L_{n,\epsilon}
  +\xk{C_{\theta_0,\lambda}(\epsilon+\epsilon^2)+o(1)}I_p,
  \qquad
  L_{n,\epsilon}\in\beta\caU_{\theta_0,\beta}.
\end{equation}
The prior can be chosen so that its own information is negligible on the root-\(n\) scale, namely \(J_{\pi_n}/n\to0\).
Applying the matrix van Trees inequality to the output experiment, then letting \(n\to\infty\) and \(\epsilon\downarrow0\), gives
\begin{equation}
  \label{eq:trans-lower-bound}
  \liminf_{n\to\infty}
  n\int
  \E_\theta
  \normx{\widehat\theta-\theta}_W^2
  \,\pi_n(d\theta)
  \ge
  \inf_{L\in\beta\caU_{\theta_0,\beta}}\caI_W(L)
  =
  \frac{1}{\beta}\uCt(W).
\end{equation}
Since the Bayes risk is bounded above by the local worst-case risk, this proves the lower-bound half of \cref{cor:trans-local-minimax} for arbitrary mechanisms.
The technical details are given in \cref{appendix-sec:trans-regime}.
   \section{Examples}
\label{sec:examples}

In this section, we give explicit expressions of the efficiency constants \( \uCt(W) \) for various parametric models.
For simplicity, we focus on local constructions with a public center \(\theta_0\),
while private localization is covered by \cref{thm:compact-minimax}.
For more examples and a summary, see \cref{appendix-sec:model-class-reductions,appendix-sec:concrete-model-constants}.

\subsection{One-dimensional regular families}
\label{subsec:1d-regular-models}

Consider the one-dimensional model with scalar target
\begin{equation}
  \label{eq:1d-model}
  X_1,\ldots,X_n \overset{\mathrm{iid}}{\sim} P_\theta,
  \qquad
  \theta\in\Theta\subset\R,
  \qquad
  \mu:\Theta\to\R.
\end{equation}
\begin{proposition}[One-dimensional constant]
  \label{prop:1d-matching}
  For the model and target in \cref{eq:1d-model}, suppose \(\mu\) is continuously differentiable, with \(\mu'(\theta_0)\ne0\).
  Under \cref{ass:regular,ass:score-density-l1}, suppose \(I(\theta_0)>0\).
  For \(\beta\in[0,1]\), define
  \begin{equation}
    \label{eq:1d-unified-info}
    \mathfrak{I}_{\theta_0}^{(\beta)}
    =
    \sup_{\substack{\phi:\caX\to\R\\\diam(\phi)\le1}}
    \frac{
      \xk{\E_{\theta_0}\zk{\phi(X)s_{\theta_0}(X)}}^2
    }{
      \beta\Var_{\theta_0}\xk{\phi(X)}+(1-\beta)/2
    }.
  \end{equation}
  Then, writing \(a_{\theta_0}=\E_{\theta_0} \abs{s_{\theta_0}(X)}/2\), the exact local minimax constant of \(\mu(\theta)\) is
  \begin{equation}
    \label{eq:1d-unified-constant}
    \uC_{\mu,\theta_0}
    =
    \frac{\xk{\mu'(\theta_0)}^2}{\mathfrak{I}_{\theta_0}^{(\beta)}},
    \qquad
    \mfC_{\mu,\theta_0}^{(0)}
    =
    \frac{\xk{\mu'(\theta_0)}^2}{2a_{\theta_0}^2},
    \quad
    \mfC_{\mu,\theta_0}^{(1)}
    =
    \frac{\xk{\mu'(\theta_0)}^2}{I(\theta_0)}.
  \end{equation}
\end{proposition}

The statistics for optimal estimators have simple forms in the privacy or sampling regime.
In the privacy regime \(\beta=0\), the optimal statistic is the score sign statistic
\begin{math}
  \phi_{\mathrm{priv}}(x)=\frac{1}{2}\sign(s_{\theta_0}(x)),
\end{math}
as in \cref{ex:1d-info-region}.
In the sampling regime \(\beta=1\), the optimal statistic is given by \( \phi_{\mathrm{stat}}(x) \propto s_{\theta_0}(x) \) with a proper rescaling and truncation approximation if the score is unbounded.

Moreover, we give explicit examples for the one-dimensional regular families.
See \Cref{tab:1d-examples}.
Specifically, the logistic location model is given by distribution function
\[
  F_\theta(x) = 1/(1+\exp\xkm{-(x-\theta)/s}),\theta \in \R,
\]
where the scale \(s>0\) is known and the target is \(\mu(\theta)=\theta\).
Denote \( r_\beta = (1-\sqrt{1-\beta})/(1+\sqrt{1-\beta}) \).
Both the Bernoulli and Laplace location models have binary scores: up to an affine transformation, \(s_{\theta_0}(X)\) is a binary statistic.
Consequently, the same statistic is simultaneously optimal in the privacy and sampling regimes.
By contrast, the logistic score takes a continuum of magnitudes, so its attaining statistic changes with \(\beta\) and its unified constant is not affine in \(\beta\).

\begin{table}[h!]
  \caption{Efficiency constants for selected one-dimensional models.}
  \label{tab:1d-examples}
  \centering
  \footnotesize
  \setlength{\tabcolsep}{3pt}
  \renewcommand{\arraystretch}{1.35}
  \begin{tabular}{@{}p{0.3\textwidth}p{0.25\textwidth}p{0.3\textwidth}@{}}
    \hline
    Model & Constant \( \mathfrak{C}^{(\beta)}(q) \) & Attaining statistic \\
    \hline
    \(\operatorname{Bernoulli}(q)\) &
    \(
    \beta q(1-q)+(1-\beta)/2
    \) &
    \(x\) \\
    {\raggedright Laplace location with scale \(b\)} &
    \(
    b^2(2-\beta)
    \) &
    \(
    \sign(x-\theta_0)/2
    \) \\
    {\raggedright Logistic location with scale \(s\)} &
    \(
    6\beta s^2/(r_\beta(3-r_\beta^2))
    \) &
    \(
    \sign(U)
    \min(\abs{U}/r_\beta,1)/2, \) \par \( U=2F_{\theta_0}(X)-1 \) \\
    \hline
  \end{tabular}
\end{table}

\subsection{Bounded mean estimation}
\label{subsec:bounded-mean-estimation}

We next give a global minimax consequence of the local theory.
Fix an integer \(p\ge1\) and \(\ell<r\), put \(\Delta=r-\ell\), and let
\(\mathcal{P}_{p,\ell,r}=\dk{P:\operatorname{supp}(P)\subset[\ell,r]^p}\).
The target is the mean
\(\mu(P)=\E_P X\in\R^p\), and the loss is squared Euclidean loss.

\begin{proposition}[Bounded mean]
  \label{prop:bounded-means}
  Suppose \(n^2\rho_n\to\infty\) and \(n\rho_n\to\lambda\in[0,\infty]\). Then
  \begin{equation}
    \label{eq:bounded-mean-global-risk}
    \inf_{\widehat\mu:\rho_n\text{-zCDP}}
    \sup_{P\in\mathcal{P}_{p,\ell,r}}
    \E_P
    \norm{\widehat\mu-\mu(P)}_2^2
    \sim
    \Delta^2\xk{
      \frac{p}{4n}
      +
      \frac{p^2}{2\tau_n^2}
    }.
  \end{equation}
  The bound is attained by releasing \(\bar X_n+G_n\) with independent \(G_n\sim N_p\xkm{0,p\Delta^2I_p/(2\tau_n^2)}\),
  where \( \bar X_n \) is the sample mean.
\end{proposition}

Thus, for distribution-free bounded mean estimation, the simple Gaussian release of the
sample mean is globally asymptotically minimax, and no other private procedure can improve its first-order risk.
The matching lower bound follows by restricting to the product Bernoulli
submodel on the vertices of the hypercube; see \cref{appendix-subsec:bounded-finite-support}.

\subsection{Gaussian mean}
\label{subsec:gauss-mean}

We consider the multidimensional Gaussian mean model \( N_p(\theta,\sigma^2 I_p), \)
where \(p\) is fixed and \(\sigma^2\) is known.
Let \(\chi_p=\norm{Z}_2\) for \(Z\sim N_p(0,I_p)\).
For \(0<\beta<1\), let \(b_{p,\beta}>0\) be the unique solution of
\begin{equation}
  \label{eq:gauss-unified-threshold}
  \E\xk{\frac{\chi_p}{b_{p,\beta}}-1}_+
  =
  \frac{2p(1-\beta)}{\beta}.
\end{equation}

\begin{proposition}[Gaussian mean]
  \label{prop:gauss-mean}
  For the Gaussian mean model, we have
  \begin{equation}
    \label{eq:gauss-unified-constant}
    \uC_{\theta}(I_p)
    =
    \frac{\beta p^2 \sigma^2}
    {\E\zk{\chi_p \min(\chi_p,b_{p,\beta})}},
    \qquad
    \mfC_{\theta}^{(0)}(I_p)
    =
    \frac{2\sigma^2 p^3}{(\E\chi_p)^2},
    \quad
    \mfC_{\theta}^{(1)}(I_p)
    =
    p\sigma^2.
  \end{equation}
  For \( \beta=0 \) and \( 0 < \beta < 1 \), the optimal statistic is given by the spatial sign and radial clipping respectively,
  \[
    \Phi_{\theta_0,\mathrm{priv}}(x) = \frac{1}{2}\frac{z}{\norm{z}},
    \quad
    \Phi_{\theta_0,\beta}(x)
    =
    \frac{1}{2b_{p,\beta}}
    \min\xk{1,\frac{b_{p,\beta}}{\norm{z}_2}}z,
    \quad
    z=\frac{x-\theta_0}{\sigma}.
  \]
  At \( \beta=1 \), it is asymptotically given by a truncation with diverging radius.
\end{proposition}

The Gaussian mean model is the canonical example with an unbounded score.
Under squared Euclidean loss, its location and rotational symmetries make the exact constant independent of the local center and reduce the multivariate optimization to the single clipping threshold \(b_{p,\beta}\).
The optimal statistic retains the score direction while clipping its radial magnitude: it is the spatial sign in the privacy regime, a radially clipped score in the transition regime, and has the linear score as its influence statistic in the sampling limit.
Thus, unlike the binary score examples, the statistic optimal in the privacy regime is not sampling efficient, and the unified constant is not an affine interpolation of the constants in the privacy and sampling regimes.

\subsection{Gaussian quantile}
\label{subsec:gauss-quantile}

For Gaussian quantile estimation, if the variance is known, then it reduces to estimating the mean.
However, when both the mean and variance are unknown, private estimation of the quantile is significantly different.
Let \(z_\alpha\) denote the \(\alpha\)-quantile of the standard normal distribution, and consider the estimation of the quantile \( q_\alpha(\mu,\sigma)=\mu+z_\alpha\sigma \).
Fix a local center \((\mu_0,\sigma_0)\) and take \(\Theta_n\) as a shrinking neighborhood.
Define
\begin{equation}
  \label{eq:gauss-quantile-constant}
  d_\alpha
  =
  \inf_{\lambda\in\R}
  \E\abs{Z-\lambda\xk{Z^2-1-z_\alpha Z}},
  \qquad Z\sim N(0,1).
\end{equation}

\begin{proposition}[Gaussian quantile with unknown scale]
  \label{prop:gauss-quantile-unknown-scale}
  Suppose that \(\alpha\in(0,1)\), \(\sigma_0>0\), \(n^2\rho_n\to\infty\), and \(n\rho_n\to0\).
  For the Gaussian quantile problem, the privacy constant is
  \[
    \lim_{n\to\infty}
    \tau_n^2
    \inf_{\widehat{q}_\alpha:\rho_n \text{-zCDP}}
    \sup_{(\mu,\sigma)\in\Theta_n}
    \E_{\mu,\sigma}
    \xk{\widehat{q}_\alpha-q_\alpha(\mu,\sigma)}^2
    =
    \frac{2\sigma_0^2}{d_\alpha^2},
  \]
  where \(d_\alpha\) is defined in \cref{eq:gauss-quantile-constant}.
  If \(\lambda_\alpha^*\) minimizes \cref{eq:gauss-quantile-constant}, the optimal statistic is generated by
  \[
    \phi_\alpha^*(x)
    =
    \frac{1}{2}\sign\xk{
      Z_0-\lambda_\alpha^*\xk{Z_0^2-1-z_\alpha Z_0}
    },
    \qquad
    Z_0=\frac{x-\mu_0}{\sigma_0}.
  \]
\end{proposition}

Here \(H_\alpha(Z)=Z^2-1-z_\alpha Z\) is the scaled score in the local nuisance direction that leaves \(q_\alpha\) unchanged.
The quantity \(d_\alpha\) is the \(L^1\) residual after the location score \(Z\) is adjusted for this nuisance direction, and the resulting sign statistic has zero covariance with \(H_\alpha(Z)\).
At \(\alpha=1/2\), symmetry gives \(\lambda_\alpha^*=0\), so the statistic reduces to \(\sign(Z)/2\) and the constant to \(\pi\sigma_0^2\), recovering the unknown-scale Gaussian mean problem.
In comparison, the non-private optimal estimator is given by the plug-in MLE \(\bar X+z_\alpha\widehat\sigma\).
This example further shows differences between optimal private and non-private estimation.

\subsection{Finite alphabet and categorical model}
\label{subsec:fin-alphabet-categorical}

Consider the finite alphabet model
\[
  \caX=\{x_1,\ldots,x_m\},\quad P_\theta(X=x_i)=\pi_i(\theta)>0,
\]
on a neighborhood of \(\theta_0\).
Under this model, the optimization for efficiency constant becomes finite dimensional and can be computed by a semidefinite programming.
Due to its complexity, we only present the simpler form in the privacy regime, and defer the general form to the appendices.
At \(\theta_0\), write \(\pi=(\pi_1,\ldots,\pi_m)^\top\), \( S \in \R^{m\times p} \) with \(S_{i\cdot}=s_{\theta_0}(x_i)^\top\), and
\(  B=\operatorname{diag}(\pi)S \).

\begin{proposition}[Finite alphabet SDP]
  \label{prop:fin-alphabet-sdp}
  If \(I(\theta_0)\succ0\) and \(W\succ0\), then
  \begin{equation}
    \label{eq:fin-alphabet-sdp}
    \begin{aligned}
      \mfC_{\theta_0}(W)
      ={}&
      \frac{1}{2}\min_{G\in\bbS_+^m,\,Z\in\bbS_+^p} \Tr Z \\
      \text{subject to }{}&
      G\pi=0,\quad
      G_{ii}+G_{jj}-2G_{ij} \le1,\quad
      \begin{pmatrix}
        B^\top G B&W^{1/2}\\
        W^{1/2}&Z
      \end{pmatrix}
      \succeq0.
    \end{aligned}
  \end{equation}
  The minimum of \cref{eq:fin-alphabet-sdp} is attained.
\end{proposition}

The optimization problem in \cref{prop:fin-alphabet-sdp} not only gives the exact efficiency constant, but also implies the construction of the optimal statistics.
In detail, for any optimal \( G \) solving \cref{eq:fin-alphabet-sdp} with factorization \(G=FF^\top\),
choose an orthonormal basis \(Q\) of \(\operatorname{ran}(F^\top B)\),
the optimal statistic is given by \(\Phi(x_i)=Q^\top F_{i\cdot}^\top\).
Conversely, any optimal statistic corresponds to an optimal \( G \).

The SDP gives a finite-dimensional form of the loss-specific incompatibility in \cref{prop:common-efficiency}.
Its feasible Gram region is fixed by the model, whereas the loss matrix \(W\) enters the objective and may select a different optimal \(G\).
If two losses have no common minimizing Gram matrix, no single statistic can be efficient for both.

As a special case, the categorical probability vector model further simplifies the computation of the exact constant.
Let \(\Delta_{k-1} = \dkx{q = (q_1,\dots,q_k) \in \R^k : q_j\ge0,\ \sum_{j=1}^k q_j=1} \) be the probability simplex of dimension \( k-1 \).
Let \(q\) denote the probability vector to be estimated, assumed to be an interior point of the simplex.
Use category \(k\) as the baseline and put \(d=k-1\).
The sampling covariance matrix of the first \(d\) category indicators is
\[
  V_q
  =
  \operatorname{diag}(q_1,\ldots,q_d)-q_{1:d}q_{1:d}^\top.
\]
Define the Gram region generated by diameter-one statistics
\[
  \mathcal{G}_k
  =
  \dkx{
    J\succeq0:
    J_{aa}\le1,\quad
    J_{aa}+J_{bb}-2J_{ab}\le1
    \text{ for all }a,b\in\dkx{1,\ldots,d}
  }.
\]

\begin{proposition}[Categorical model]
  \label{prop:categorical-prob}
  For \(W\succ0\),
  \begin{equation}
    \begin{aligned}
      \mfC_q(W)
      &=
      \frac{1}{2} \min_{J\in\mathcal{G}_k,~ J \succ 0 } \Tr(WJ^{-1}), \\
      \uC_q(W)
      &=
      \beta\Tr(WV_q)+(1-\beta)\mfC_q(W),
      \qquad 0\le\beta\le1.
    \end{aligned}
    \label{eq:categorical-unified-constant}
  \end{equation}
  Every statistic that is optimal in the privacy regime remains optimal for \(0\le\beta<1\).
  Under squared Euclidean loss on the full probability vector, \(W=I_d+\mathbf{1}_d\mathbf{1}_d^\top\), and
  \[
    \uC_q
    =
    \beta\xkx{1-\norm{q}_2^2}
    +(1-\beta)(k-1).
  \]
\end{proposition}

Here, \(\Tr(WV_q)\) is the sampling regime constant dictated by the Fisher information, whereas \(\mfC_q(W)\) is the privacy regime constant and depends only on \(W\) and \(k\).
The categorical model is exceptional because, for each fixed loss, a statistic optimal in the privacy regime remains optimal throughout \(0\le\beta<1\).
Thus, unlike in a general finite-alphabet model, the same bounded statistic jointly retains the sampling and privacy information needed in every non-sampling regime.

For squared Euclidean loss on the full probability vector, the scaled one-hot statistic \(\Phi_{\mathrm{hist}}(j)=e_j/\sqrt{2}\), up to translation to the simplex tangent space, has diameter one and attains the constant.
The term \(1-\norm{q}_2^2\) is the sampling contribution of the empirical histogram, whereas \(k-1\) is the privacy contribution from the \((k-1)\)-dimensional simplex tangent space.
The sampling contribution is largest at the uniform probability vector.

\subsection{Generalized linear regression}
\label{subsec:generalized-linear-regression}

We consider the canonical generalized linear model (GLM) with Gaussian design
\begin{equation}
  \label{eq:gauss-design-glm}
  X\sim N_p(0,I_p),
  \qquad
  p_\vartheta(y\mid x)
  =
  \exp\xk{
    \frac{y x^\top \vartheta-b(x^\top \vartheta)}{\kappa}
    +c(y,\kappa)
  },
\end{equation}
where \(\vartheta\in\R^p\) is the regression parameter, \(X\in\R^p\) is the covariate, and \(Y\) is the response.
The function \(b\) is the cumulant function, \(\kappa>0\) is a known dispersion parameter, and \(c(y,\kappa)\) is the normalizing term.

The results for GLM under Gaussian design resemble that of Gaussian mean estimation.
Let \(Y_0\) follow the response family at natural parameter zero and, independently, let \(\chi_p=\norm{Z}_2\) for \(Z\sim N_p(0,I_p)\).
\begin{equation}
  \label{eq:glm-absolute-score}
  A_0=\frac{Y_0-b'(0)}{\kappa},
  \qquad
  d_0=\E_0 \abs{A_0},
  \qquad
  R_0=\abs{A_0}\chi_p,
\end{equation}
At the null center, the score is \(s_0(X,Y_0)=A_0X\).
Suppose \cref{ass:regular,ass:score-density-l1} hold and \(b''(0)>0\).
For \(0<\beta<1\), let \(t_\beta>0\) be the unique solution of
\begin{math}
   \E\xk{{R_0}/{t_\beta}-1}_+
  =
  {2p(1-\beta)}/{\beta}.
\end{math}

\begin{proposition}[GLM with Gaussian design]
  \label{prop:gauss-design-glm}
  For the GLM \cref{eq:gauss-design-glm} at \(\vartheta_0=0\), the exact local constant under squared Euclidean loss is
  \begin{equation}
    \label{eq:glm-unified-constant}
    \uC_0(I_p)
    =
    \frac{\beta p^2}
    {\E\zk{R_0 \min(R_0,t_\beta)}},
    \qquad
    \mfC_0^{(0)}(I_p)
    =
    \frac{2p^3}{(\E\chi_p)^2d_0^2},
    \qquad
    \mfC_0^{(1)}(I_p)
    =
    \frac{p\kappa}{b''(0)}.
  \end{equation}
  For \(\beta=0\) and \(0<\beta<1\), the optimal statistic is given by the score direction and radial clipping, respectively,
  \[
    \Phi_{\mathrm{priv}}(x,y)
    =
    \frac{1}{2}\frac{s_0(x,y)}{\norm{s_0(x,y)}_2},
    \qquad
    \Phi_\beta(x,y)
    =
    \frac{1}{2}
    \frac{s_0(x,y)}{\norm{s_0(x,y)}_2}
    \min\xk{\frac{\norm{s_0(x,y)}_2}{t_\beta},1},
  \]
  with both statistics set to zero when \(s_0(x,y)=0\).
  At \(\beta=1\), an asymptotically efficient statistic is given by a score truncation with diverging radius.
\end{proposition}

\begin{table}[htpb]
  \caption{Privacy and sampling regime constants for selected GLMs with Gaussian design at the null center.}
  \label{tab:glm-regime-constants}
  \centering
  \footnotesize
  \setlength{\tabcolsep}{3pt}
  \renewcommand{\arraystretch}{1.25}
  \begin{tabular}{@{}p{0.25\textwidth}p{0.20\textwidth}p{0.28\textwidth}p{0.15\textwidth}@{}}
    \hline
    Model & \(R_0\) & Privacy & Sampling \\
    \hline
    {\raggedright Gaussian linear regression with noise variance \(\sigma^2\)\par} &
    \(\chi_1\chi_p/\sigma\) &
    \(\pi\sigma^2p^3(\E\chi_p)^{-2}\) &
    \(p\sigma^2\) \\
    {\raggedright Logistic regression\par} &
    \(\chi_p/2\) &
    \(8p^3(\E\chi_p)^{-2}\) &
    \(4p\) \\
    {\raggedright Poisson log-linear regression\par} &
    {\raggedright \(\abs{Y_0-1}\chi_p\), \(Y_0\sim\operatorname{Poisson}(1)\)\par} &
    \(2^{-1}e^2p^3(\E\chi_p)^{-2}\) &
    \(p\) \\
    \hline
  \end{tabular}
\end{table}

At the null center, Gaussian design makes the score spherically symmetric with radial magnitude \(R_0\).
As in Gaussian mean estimation, the multivariate calculation therefore reduces to the scalar threshold \(t_\beta\).
The same reduction applies to GLMs with spherically symmetric designs.
\Cref{tab:glm-regime-constants} lists the constants for several GLM examples.
   \section{Extensions}
\label{sec:extensions}

\subsection{Gaussian differential privacy}
\label{subsec:gauss-differential-privacy}

Our efficiency theory naturally extends to the Gaussian DP notion \citep{dong2022_GaussianDifferential}.
For probability measures \(P\) and \(Q\), define the trade-off function
\[
  T(P,Q)(a)=\inf_{\varphi:\,\E_P\varphi\le a} \xk{1-\E_Q\varphi},
  \qquad a\in[0,1],
\]
where the infimum is over measurable tests \(\varphi\) taking values in \([0,1]\).

\begin{definition}[\(\mu\)-GDP]
  \label{def:gdp}
  A mechanism \(M\) with a standard Borel output space is \(\mu\)-Gaussian differentially private (GDP) if
  \[
    T(M(S),M(S'))\ge G_\mu \coloneqq T(N(0,1),N(\mu,1))
  \]
  pointwise on \([0,1]\) for every adjacent pair \(S,S'\).
\end{definition}

We use the following standard comparison, Gaussian mechanism, and composition properties of GDP \citep[Theorems~1, 2, and~4 and Corollary~2]{dong2022_GaussianDifferential}.

\begin{lemma}
  \label{lem:gdp-transfer}
  The following statements hold.
  \begin{enumerate}[label=(\roman*)]
    \item Every \(\mu\)-GDP mechanism is \((\mu^2/2)\)-zCDP.
    \item The Gaussian mechanism in \cref{lem:gauss-composition} is
    \(\mu\)-GDP when \(\rho=\mu^2/2\), and the adaptive composition of
    \(\mu_j\)-GDP mechanisms, \(j=1,\ldots,k\), is
    \((\sum_{j=1}^k\mu_j^2)^{1/2}\)-GDP.
  \end{enumerate}
\end{lemma}

The zCDP theory transfers directly to GDP thanks to the connection between the two and the construction of the attaining estimator.
On one hand, \cref{lem:gdp-transfer}(i) makes the lower bound immediate by inclusion.
On the other hand, by \cref{lem:gdp-transfer}(ii), the moment-map inverse estimator and preliminary localization used in the upper bounds also satisfy GDP.
Combining the upper and lower part, we obtain the following efficiency corollary for GDP.

\begin{corollary}[Efficiency under GDP]
  \label{cor:gdp-unified-efficiency}
  Let \((\mu_n)\) be a positive sequence and set \(\rho_n=\mu_n^2/2\).
  In the notation of
  \cref{eq:unified-effective-info-scale,eq:unified-local-param-set},
  the following conclusions hold.
  \begin{enumerate}[label=(\roman*)]
    \item Under the assumptions of \cref{thm:local-minimax},
    \begin{equation}
      \lim_{n\to\infty}
      \sigma_n^{-2}
      \inf_{\widehat\theta:\mu_n\text{-GDP}}
      \sup_{\theta\in\widetilde\Theta_n}
      \E_\theta
      \normx{\widehat\theta-\theta}_W^2
      =
      \uCt(W).
    \end{equation}
    \item Under the assumptions of \cref{thm:compact-minimax},
    \begin{equation}
      \lim_{n\to\infty}
      \inf_{\widehat\theta:\mu_n\text{-GDP}}
      \sup_{\theta\in K}
      \frac{
        \sigma_n^{-2}
        \E_\theta
        \normx{\widehat\theta-\theta}_W^2
      }{
        \uC_\theta(W)
      }
      =1.
    \end{equation}
  \end{enumerate}
\end{corollary}

\subsection{Differentiable targets}
\label{subsec:differentiable-targets}

Although our results are stated for \(\theta\) itself, the same theory applies
to a differentiable target \(\psi:\Theta\to\R^m\) under quadratic loss
\( \normx{\hat{\psi} - \psi(\theta)}_W^2 \) for a quadratic loss matrix \(W\succ0\).
The optimal constant \( \mfC_{\psi,\theta}^{(\beta)}(W) \) is obtained via the local expansion of \( \psi \) and is given by
\begin{equation}
  \label{eq:target-efficiency-constant}
  \mfC_{\psi,\theta}^{(\beta)}(W)
  =
  \mfC_\theta^{(\beta)}(W_{\psi,\theta}),
  \qquad
  W_{\psi,\theta}
  =
  D\psi(\theta)^\top W D\psi(\theta).
\end{equation}
Detailed formal results are given in \cref{appendix-subsec:differentiable-target-efficiency}.
   \section{Conclusion}
\label{sec:conclusion}

This paper develops an exact asymptotic efficiency theory for estimation under central zCDP.
Privacy does not simply apply a scalar discount to ordinary Fisher information,
but introduces a set-valued information region recording the privacy constraint.
The same region supplies both sides of the theory: it constrains the Fisher information of every zCDP output experiment and generates matching moment-map inverse estimators based on nearly optimal statistics.
The set-valued information formulation extends across privacy regimes and shows how privacy and sampling interact through the same bounded statistic.
In particular, in the transition regime, the statistic's sampling covariance and privacy sensitivity yield a joint variational constant rather than a weighted average of the two endpoint constants, while the sampling regime limit recovers ordinary Fisher efficiency.

Our theory also supplies practically useful private estimators.
For each loss and privacy regime, a moment-map inverse estimator based on a nearly optimal statistic and its Gaussian release attains the optimal constant,
so no nonadditive, nonlinear, or otherwise unrestricted zCDP mechanism can improve the leading constant for these problems.
In the examples, this principle yields explicit constructions, showing how the information geometry translates into concrete private procedures.

Several directions now become natural.
An immediate question is to develop optimal private estimators for other parametric models
and numerical methods for optimizing the information region and its associated mechanisms.
A natural direction is to develop an exact efficiency theory under \((\varepsilon,\delta)\)-differential privacy, where the privacy constraint calls for a different geometric characterization of the attainable local experiments.
More broadly, semiparametric functionals and covariance or eigenspace problems offer settings in which nuisance structure and matrix geometry may produce genuinely new private efficiency phenomena.
Establishing efficiency theory for nonparametric estimation under privacy is also an interesting future problem.

We believe that the theory developed here provides a foundation for a general theory of statistical efficiency under differential privacy, in which privacy, statistical information, and decision loss are understood through a common information geometry.

  \paragraph{Appendices.} Complete proofs and additional model calculations are included in Sections S1--S8 below.

\clearpage
\begingroup
    \makeatletter
    \def\nocontentsline{%
      \let\@@addcontentsline\addcontentsline
      \def\addcontentsline##1##2##3{%
        \let\addcontentsline\@@addcontentsline
      }%
    }
    \makeatother
    \setcounter{tocdepth}{2}

  \endgroup
  \clearpage

  \appendix
\setcounter{table}{0}\setcounter{figure}{0}
\section*{Appendices}
  \renewcommand{\thesection}{S\arabic{section}}
  \numberwithin{equation}{section}
  \renewcommand{\thefigure}{S\arabic{figure}}
  \renewcommand{\thetable}{S\arabic{table}}

  \section{Information Geometry}
\label{appendix-sec:info-region-proofs}

This section collects the information-geometric results used throughout the
appendices.
We first establish the compactness, order structure, finite-dimensional representation, directional envelope, and local stability of the privacy information region.
We then study the geometry of the unified information region in the privacy and sampling regimes.
Throughout the appendices, we write \(J_{\Phi,\theta}\coloneqq A_{\Phi,\theta}^*A_{\Phi,\theta}\) for the information matrix generated by \(\Phi\).

\subsection{Proof of \Cref{prop:dc-info-prop}}
\label{appendix-subsec:compact}

The diameter identity in part~\textup{(i)} is immediate.
Since \(\E_\theta s_\theta(X)=0\), for every \(h\in\R^p\),
\[
  A_{\Phi+c,\theta} h
  =
  A_{\Phi,\theta} h
  +c\E_\theta \ang{s_\theta(X),h}
  =
  A_{\Phi,\theta} h.
\]

The zero statistic shows that the set is nonempty.
For convexity, it suffices to consider two generating matrices \(J_{\Phi_1,\theta}\) and \(J_{\Phi_2,\theta}\).
For \(0\le\lambda\le1\), define the direct-sum statistic
\[
  \Phi(x)
  =
  \xk{\sqrt\lambda\,\Phi_1(x),\sqrt{1-\lambda}\,\Phi_2(x)}.
\]
Then \(\diam(\Phi)\le1\) and
\[
  J_{\Phi,\theta}
  =
  \lambda J_{\Phi_1,\theta}
  +(1-\lambda)J_{\Phi_2,\theta}.
\]
Passing to the closure proves convexity.

Translate any generating statistic \(\Phi:\caX\to\R^d\) so that \(\norm{\Phi(x)}_2 \le1\).
For the standard basis \(e_1,\ldots,e_p\) of \(\R^p\),
\[
\norm{A_{\Phi,\theta}}_{\HS}^2 = \sum_{j=1}^p \norm{\E_\theta (\Phi(X)s_{\theta,j}(X))}_2^2 \le \E_\theta \norm{s_\theta(X)}_2^2.
\]
Thus the generating matrices, and hence their closure, are bounded.
The set is closed by definition, so it is compact in the finite-dimensional space \(\bbS^p\).

We next prove downward closure.
First suppose that \(J=A^*A\) is generated by an admissible statistic and that \(0\preceq\widetilde{J}\preceq J\).
There is a positive contraction \(D\) on \(\R^p\), supported on \(\ran(J)\), such that
\[
\widetilde{J}=J^{1/2} DJ^{1/2}.
\]
Let \(A=UJ^{1/2}\) be the polar decomposition of \(A\), and define the positive contraction \(T=UDU^*\) on the finite-dimensional codomain of \(A\), extended by zero outside \(\ran(U)\).
Then the statistic \(T^{1/2} \Phi\) has diameter at most one and
\[
J_{T^{1/2} \Phi,\theta}
=
A^*TA
=
\widetilde{J}.
\]
For a general \(J\in\caJ_\theta\), take generating matrices \(J_m \to J\) and write \(\widetilde{J}=J^{1/2} DJ^{1/2}\) as above.
The matrices
\[
\widetilde{J}_m=J_m^{1/2} DJ_m^{1/2}
\]
satisfy \(0\preceq\widetilde{J}_m \preceq J_m\) and converge to \(\widetilde{J}\).
The generating case and closedness of \(\caJ_\theta\) therefore prove part~\textup{(iii)}.

Finally, if \(\caJ_\theta\) contains a positive definite matrix, \cref{eq:info-region-fisher-bound} implies \(I(\theta)\succ0\).
Conversely, suppose \(I(\theta)\succ0\), and define the bounded \(\R^p\)-valued statistic
\[
\Phi_M(x)
=
\frac{1}{2M}s_\theta(x)\ind{\norm{s_\theta(x)}_2 \le M}.
\]
Its diameter is at most one, and
\[
A_{\Phi_M,\theta}
=
\frac{1}{2M}
\E_\theta \zk{
  s_\theta(X)s_\theta(X)^\top
  \ind{\norm{s_\theta(X)}_2 \le M}
}.
\]
The matrix inside the expectation converges to \(I(\theta)\) as \(M\to\infty\), so it is positive definite for all sufficiently large \(M\).
Thus \(J_{\Phi_M,\theta} \succ0\), proving part~\textup{(iv)}.

\subsection{Hilbert-to-finite-dimensional compression}
\label{appendix-subsec:hilbert-compression}

\begin{lemma}[Hilbert-to-finite-dimensional compression]
  \label{appendix-lem:hilbert-compression}
  Suppose \cref{ass:regular} holds, fix \(\theta\in\Theta\), and let \(\Phi:\caX\to\caH\) be a statistic taking values in a real Hilbert space \(\caH\), with finite second moment.
  Then there is a statistic \(\widetilde\Phi:\caX\to\R^p\) such that
  \[
    \diam(\widetilde\Phi)
    \le
    \diam(\Phi),
    \qquad
    J_{\widetilde\Phi,\theta}
    =
    J_{\Phi,\theta}.
  \]
  Consequently,
  \[
    \caJ_\theta
    =
    \cl\dk{
      J_{\Phi,\theta}:
      \Phi:\caX\to\R^p,\ \diam(\Phi)\le1
    }.
  \]
  The same region is obtained from arbitrary Hilbert-valued statistics:
  \[
    \caJ_\theta
    =
    \cl\dk{
      J_{\Phi,\theta}:\
      \substack{
        \caH\text{ a real Hilbert space},\ \Phi:\caX\to\caH,\\
        \E_\theta\norm{\Phi(X)}_{\caH}^2<\infty,\ \diam(\Phi)\le1
      }
    },
  \]
\end{lemma}

\begin{proof}
  Let \(\caK=\ran(A_{\Phi,\theta})\), so \(\dim(\caK)\le p\), and let \(\Pi_{\caK}:\caH\to\caK\) be the orthogonal projection.
  Choose a linear isometry \(U:\caK\to\R^p\), set \(T=U\Pi_{\caK}\), and define \(\widetilde\Phi=T\Phi\).
  Since \(T\) is a contraction,
  \[
    \diam(\widetilde\Phi)
    \le
    \diam(\Phi).
  \]
  Moreover,
  \[
    A_{\widetilde\Phi,\theta}
    =
    TA_{\Phi,\theta}
    =
    UA_{\Phi,\theta}.
  \]
  Because \(U\) is isometric on \(\caK\supseteq\ran(A_{\Phi,\theta})\),
  \[
    J_{\widetilde\Phi,\theta}
    =
    A_{\Phi,\theta}^*U^*UA_{\Phi,\theta}
    =
    A_{\Phi,\theta}^*A_{\Phi,\theta}
    =
    J_{\Phi,\theta}.
  \]
  Thus every Hilbert-valued generator has an \(\R^p\)-valued representative.
  The reverse inclusion follows because \(\R^p\) is itself a real Hilbert space, and the first equality follows from the finite-dimensional definition of \(\caJ_\theta\) together with the same compression applied to each generator before taking the closure.
\end{proof}

\subsection{Proof of \Cref{prop:directional-envelope-tv}}
\label{appendix-subsec:directional-info}

Fix \(h\in\R^p\) and put \(g_h(x)=\ang{s_\theta(x),h}\).
For every statistic \(\Phi\) with diameter at most one and every unit vector \(v\in\caH\), the scalar statistic \(\ang{\Phi,v}_{\caH}\) has diameter at most one.
After subtracting the midpoint of the infimum and supremum of its range and using \(\E_\theta g_h(X)=0\), we obtain
\[
  \abs{\E_\theta \zk{\ang{\Phi(X),v}_{\caH} g_h(X)}}
  \le
  \frac{1}{2}\E_\theta \abs{g_h(X)}.
\]
Taking the supremum over unit vectors \(v\) and then over admissible \(\Phi\) gives the upper bound in \cref{eq:directional-info-envelope}.
Equality is attained by the scalar statistic
\[
  \phi_h(x)
  =
  \frac{1}{2}\sign\zk{\ang{s_\theta(x),h}},
\]
where \(\sign(0)=0\).
For this statistic,
\[
  A_{\phi_h,\theta} h
  =
  \frac{1}{2}\E_\theta \abs{g_h(X)},
\]
which proves equality in \cref{eq:directional-info-envelope}.

By \cref{ass:score-density-l1},
\[
  p_{\theta+th}
  -
  p_\theta
  =
  t p_\theta g_h
  +r_t,
  \qquad
  \norm{r_t}_{L^1(\mu)}
  =
  o(|t|).
\]
The reverse triangle inequality gives
\[
  \abs{
    \norm{t p_\theta g_h+r_t}_{L^1(\mu)}
    -
    |t|\norm{p_\theta g_h}_{L^1(\mu)}
  }
  \le
  \norm{r_t}_{L^1(\mu)}
  =
  o(|t|).
\]
Consequently,
\[
  \TV(P_{\theta+th},P_\theta)
  =
  \frac{|t|}{2}\E_\theta \abs{g_h(X)}
  +o(|t|).
\]
Combining this expansion with \cref{eq:directional-info-envelope} proves \cref{eq:local-tv-info}.

\subsection{Local stability of the information region}
\label{appendix-subsec:stability}

\begin{lemma}[Local stability of the information region]
  \label{appendix-lem:score-density-stability}
  Under \cref{ass:score-density-l1}, \(\Gamma_\theta\) is locally bounded at \(\theta_0\).
  Moreover, with
  \[
    d(J,\caJ_{\theta_0})
    =
    \inf_{\widetilde{J}\in\caJ_{\theta_0}}
    \norm{J-\widetilde{J}}_{\HS},
  \]
  we have
  \[
    \sup_{\norm{\theta-\theta_0}_2 \le\eta}
    \sup_{J\in\caJ_\theta}
    d(J,\caJ_{\theta_0})
    \to0
    \qquad\text{as }\eta\downarrow0.
  \]
\end{lemma}

\begin{proof}
Let \(\Phi:\caX\to\R^d\) have diameter at most one.
After translation, we may assume that \(\norm{\Phi(x)}_2 \le1\).
Put
\[
\varepsilon(\theta) = \int\norm{\dot{p}_\theta(x)-\dot{p}_{\theta_0}(x)}_2 d\mu(x).
\]
For any unit vector \(h\in\R^p\),
\[
\norm{(A_{\Phi,\theta}-A_{\Phi,\theta_0})h}_2 \le \varepsilon(\theta),
\]
uniformly over such \(\Phi\).
Consequently
\[
\norm{A_{\Phi,\theta}-A_{\Phi,\theta_0}}_{\HS} \le \sqrt p \varepsilon(\theta),
\]
uniformly over diameter-one statistics.
The local boundedness follows from
\[
\norm{A_{\Phi,\theta}}_{\HS} \le \sqrt p\int\norm{\dot{p}_\theta(x)}_2 d\mu(x),
\]
again uniformly over translated diameter-one \(\Phi\).
In particular, the right-hand side is bounded for \(\theta\) near \(\theta_0\), so the corresponding operators have a common bound \(C\).
The identity
\[
A^*A-B^*B=A^*(A-B)+(A^*-B^*)B
\]
then yields
\[
\norm{ A_{\Phi,\theta}^*A_{\Phi,\theta} -A_{\Phi,\theta_0}^*A_{\Phi,\theta_0} }_{\HS} \le 2C\sqrt p \varepsilon(\theta).
\]
For every generating matrix in \(\caJ_\theta\), the matrix generated by the same statistic at \(\theta_0\) therefore lies within the distance of it and belongs to \(\caJ_{\theta_0}\).
Approximating an arbitrary element of \(\caJ_\theta\) by generating matrices and using closedness of \(\caJ_{\theta_0}\) gives
\[
\sup_{J\in\caJ_\theta} d(J,\caJ_{\theta_0}) \le 2C\sqrt p \varepsilon(\theta).
\]
Taking the supremum over \(\norm{\theta-\theta_0}_2 \le\eta\) proves the claimed stability, while the same common operator bound proves local boundedness of \(\Gamma_\theta\).
\end{proof}

\subsection{Unified information region}
\label{appendix-subsec:unified-info-geometry}

Fix \(\theta_0\in\Theta\).
For a bounded statistic \(\Phi:\caX\to\R^d\), write
\[
  A_{\Phi,0}
  =
  \E_{\theta_0}\{\Phi(X)s_{\theta_0}(X)^\top\},
  \qquad
  V_{\Phi,0}
  =
  \Var_{\theta_0}\{\Phi(X)\}.
\]
For \(\beta\in[0,1]\), let \(\caU_{\theta_0,\beta}\) and
\(\uCt(W)\) be the unified information region and
constant in \cref{def:unified-info}.

\begin{proposition}[Basic geometry of the unified information region]
  \label{appendix-prop:unified-region-geometry}
  Under \cref{ass:regular}, fix \(\theta_0\in\Theta\) and
  \(\beta\in[0,1]\).
  Then the following statements hold.
  \begin{enumerate}[label=(\roman*)]
    \item The unified information matrix is invariant under translations and
    linear isometries of the statistic.
    More precisely, for every \(c\in\R^d\) and every linear isometry
    \(Q:\R^d\to\R^m\),
    \[
      H_{\Phi+c,\theta_0,\beta}
      =H_{\Phi,\theta_0,\beta}
      =H_{Q\Phi,\theta_0,\beta},
      \qquad
      H_{\Phi,\theta_0,\beta}
      \coloneqq
      A_{\Phi,0}^\top
      \Sigma_{\Phi,\theta_0,\beta}^\dagger
      A_{\Phi,0}.
    \]
    \item Every finite-dimensional generator has a representative taking
    values in \(\R^p\), with no larger diameter.
    Consequently,
    \[
      \caU_{\theta_0,\beta}
      =
      \cl\dk{
        H_{\Phi,\theta_0,\beta}:
        \Phi:\caX\to\R^p,\ \diam(\Phi)\le1
      }.
    \]
    \item The region \(\caU_{\theta_0,\beta}\) is a nonempty compact subset
    of \(\bbS_+^p\), and
    \begin{equation}
      \label{appendix-eq:unified-fisher-envelope}
      0\preceq H
      \preceq
      \frac{1}{2-\beta}I(\theta_0),
      \qquad H\in\caU_{\theta_0,\beta}.
    \end{equation}
    \item If \(\beta<1\), then \(\caU_{\theta_0,\beta}\) is convex and
    downward closed in the Loewner order: if
    \(H\in\caU_{\theta_0,\beta}\) and
    \(0\preceq\widetilde H\preceq H\), then
    \(\widetilde H\in\caU_{\theta_0,\beta}\).
    \item The region \(\caU_{\theta_0,\beta}\) contains a positive-definite
    matrix if and only if \(I(\theta_0)\succ0\).
    In that case, for every \(W\succeq0\), the infimum defining
    \(\uCt(W)\) is attained.
  \end{enumerate}
\end{proposition}

\begin{proof}
  Translation leaves the diameter, covariance, and score--statistic
  covariance unchanged.
  If \(Q\) is an isometry, then
  \(A_{Q\Phi,0}=QA_{\Phi,0}\) and
  \(V_{Q\Phi,0}=QV_{\Phi,0}Q^\top\).
  The Moore--Penrose inverse on \(\ran(Q)\), together with
  \(Q^\top Q=I_d\), proves part~\textup{(i)}.

  We next prove the compression claim.
  Put
  \[
    S=\Sigma_{\Phi,\theta_0,\beta},
    \qquad
    \caK=\ran(S^\dagger A_{\Phi,0}),
  \]
  so that \(\dim(\caK)\le p\), and let \(P_{\caK}\) be the orthogonal
  projection onto \(\caK\).
  We have
  \(\ker(S)\subseteq\ker(A_{\Phi,0}^\top)\): this is immediate when
  \(\beta<1\), while at \(\beta=1\) it follows because a zero-variance
  linear combination of \(\Phi\) is constant almost surely.
  Hence, for every \(h\in\R^p\),
  \[
    (A_{\Phi,0}h)^\top S^\dagger A_{\Phi,0}h
    =
    \sup_{v\in\R^d}
    \dk{
      2v^\top A_{\Phi,0}h-v^\top Sv
    },
  \]
  and the supremum has a maximizer in \(\caK\).
  Restricting the variational formula to \(\caK\) therefore does not change
  its value.
  The statistic \(P_{\caK}\Phi\), viewed as \(\caK\)-valued, has unified
  covariance \(P_{\caK}SP_{\caK}\) on \(\caK\), so it generates the same
  information matrix and has no larger diameter.
  An isometric embedding of \(\caK\) into \(\R^p\) and part~\textup{(i)}
  prove part~\textup{(ii)}.

  Translate \(\Phi\) so that it is centered under \(P_{\theta_0}\).
  For every unit vector \(v\), the scalar statistic \(v^\top\Phi\) has range
  of length at most one, and hence
  \[
    v^\top V_{\Phi,0}v
    \le\frac14.
  \]
  Thus \(V_{\Phi,0}\preceq I_d/4\) and
  \[
    \Sigma_{\Phi,\theta_0,\beta}
    \succeq
    (2-\beta)V_{\Phi,0}.
  \]
  Moreover,
  \(\ker(V_{\Phi,0})\subseteq\ker(A_{\Phi,0}^\top)\), and the ordinary
  information inequality gives
  \[
    A_{\Phi,0}^\top V_{\Phi,0}^\dagger A_{\Phi,0}
    \preceq I(\theta_0).
  \]
  Combining the last two inequalities yields
  \[
    H_{\Phi,\theta_0,\beta}
    \preceq
    \frac{1}{2-\beta}
    A_{\Phi,0}^\top V_{\Phi,0}^\dagger A_{\Phi,0}
    \preceq
    \frac{1}{2-\beta}I(\theta_0).
  \]
  The zero statistic gives nonemptiness, and the definition gives closedness.
  The uniform bound therefore proves compactness and part~\textup{(iii)}.

  Suppose now that \(\beta<1\), and write
  \(c_\beta=(1-\beta)/2>0\).
  We first establish downward closure for a generator
  \(H=A^\top S^{-1}A\), where
  \(A=A_{\Phi,0}\) and \(S=\beta V_{\Phi,0}+c_\beta I_d\).
  If \(0\preceq\widetilde H\preceq H\), the Douglas factorization lemma
  gives a positive contraction \(D\) on \(\R^d\) such that, with
  \(B=S^{-1/2}A\),
  \[
    \widetilde H=B^\top DB.
  \]
  Put \(R=S^{-1/2}DS^{-1/2}\preceq S^{-1}\), and for
  \(\epsilon\in(0,1)\) define
  \[
    R_\epsilon=(1-\epsilon)R+\epsilon S^{-1},
    \qquad
    G_\epsilon
    =c_\beta
    \xk{R_\epsilon^{-1}-\beta V_{\Phi,0}}^{-1}.
  \]
  Then \(0\prec G_\epsilon\preceq I_d\).
  With \(T_\epsilon=G_\epsilon^{1/2}\), direct calculation gives
  \[
    T_\epsilon^\top
    \xk{
      \beta T_\epsilon V_{\Phi,0}T_\epsilon^\top
      +c_\beta I_d
    }^{-1}
    T_\epsilon
    =R_\epsilon.
  \]
  Since \(T_\epsilon\) is a contraction,
  \(\diam(T_\epsilon\Phi)\le1\), and its information matrix converges to
  \(A^\top RA=\widetilde H\).
  Now let \(H_m\to H\) be a sequence of generators and suppose
  \(0\preceq\widetilde H\preceq H\).
  Write
  \(\widetilde H=H^{1/2}DH^{1/2}\) for a positive contraction \(D\),
  supported on \(\ran(H)\), and put
  \(\widetilde H_m=H_m^{1/2}DH_m^{1/2}\).
  Then \(0\preceq\widetilde H_m\preceq H_m\) and
  \(\widetilde H_m\to\widetilde H\).
  The generating case and closedness therefore prove downward closure of
  the entire region.

  For convexity, let \(H_i\) be generated by \(\Phi_i\), \(i=1,2\), and
  fix \(t\in[0,1]\).
  The direct-sum statistic
  \[
    \Psi=\xk{\sqrt t\,\Phi_1,\sqrt{1-t}\,\Phi_2}
  \]
  has diameter at most one.
  For a fixed \(h\in\R^p\), apply the variational formula above to block
  vectors of the form
  \(\xk{\sqrt t\,v_1,\sqrt{1-t}\,v_2}\).
  With \(Z_i=v_i^\top\Phi_i(X)\), the inequality
  \[
    \Var\{tZ_1+(1-t)Z_2\}
    \le
    t\Var(Z_1)+(1-t)\Var(Z_2)
  \]
  and optimization over \(v_1,v_2\) show that
  \[
    h^\top H_{\Psi,\theta_0,\beta}h
    \ge
    t h^\top H_1h+(1-t)h^\top H_2h.
  \]
  Since this holds for every \(h\), the corresponding matrix inequality
  follows.
  Downward closure and approximation of arbitrary elements by generators
  prove convexity, completing part~\textup{(iv)}.

  If \(\caU_{\theta_0,\beta}\) contains a positive-definite matrix, the
  bound in \cref{appendix-eq:unified-fisher-envelope} implies
  \(I(\theta_0)\succ0\).
  Conversely, if \(I(\theta_0)\succ0\),
  \cref{prop:dc-info-prop}\textup{(iv)} and
  \cref{appendix-lem:hilbert-compression} give a diameter-one statistic with
  full-column-rank \(A_{\Phi,0}\), whose unified information matrix is
  positive definite.
  Finally, lower semicontinuity of \(\caI_W\) on the compact region gives
  attainment, proving part~\textup{(v)}.
\end{proof}

\begin{lemma}[Directional envelope]
  \label{appendix-lem:unified-directional-envelope}
  Under \cref{ass:regular}, fix \(\theta_0\in\Theta\),
  \(\beta\in[0,1]\), and \(h\in\R^p\).
  With \(g_h(X)=\ang{s_{\theta_0}(X),h}\),
  \begin{equation}
    \label{appendix-eq:unified-directional-envelope}
    \sup_{H\in\caU_{\theta_0,\beta}}h^\top Hh
    =
    \sup_{\substack{\phi:\caX\to\R\\\diam(\phi)\le1}}
    \frac{
      \xk{\E_{\theta_0}\{\phi(X)g_h(X)\}}^2
    }{
      \beta\Var_{\theta_0}\{\phi(X)\}+(1-\beta)/2
    },
  \end{equation}
  where a zero-over-zero ratio at \(\beta=1\) is interpreted as zero.
  In particular,
  \begin{align}
    \sup_{H\in\caU_{\theta_0,0}}h^\top Hh
    & =
    \frac12
    \xk{\E_{\theta_0}|g_h(X)|}^2,
    \label{appendix-eq:unified-directional-privacy-regime}
    \\
    \sup_{H\in\caU_{\theta_0,1}}h^\top Hh
    & =
    h^\top I(\theta_0)h.
    \label{appendix-eq:unified-directional-sampling-regime}
  \end{align}
\end{lemma}

\begin{proof}
  Fix a statistic \(\Phi:\caX\to\R^d\), put
  \(a=A_{\Phi,0}h\) and
  \(S=\Sigma_{\Phi,\theta_0,\beta}\).
  Since \(a\in\ran(S)\),
  \[
    a^\top S^\dagger a
    =
    \sup_{v\ne0}
    \frac{(v^\top a)^2}{v^\top Sv}.
  \]
  For every \(v\ne0\), the scalar statistic
  \(\phi=v^\top\Phi/\norm{v}_2\) has diameter at most one, and its ratio in
  \cref{appendix-eq:unified-directional-envelope} equals the corresponding
  Rayleigh quotient.
  The reverse inequality follows because scalar statistics are included in
  the definition of \(\caU_{\theta_0,\beta}\).

  At \(\beta=0\), translation invariance and \(\E_{\theta_0}g_h(X)=0\)
  give
  \[
    \sup_{\diam(\phi)\le1}
    \abs{\E_{\theta_0}\{\phi(X)g_h(X)\}}
    =
    \frac12\E_{\theta_0}|g_h(X)|,
  \]
  attained by \(\phi=\sign(g_h)/2\).
  At \(\beta=1\), Cauchy--Schwarz bounds the ratio by
  \(\E_{\theta_0}g_h(X)^2\), and bounded truncations of \(g_h\), followed
  by rescaling to diameter one, attain this bound in the limit.
  This proves the identities for the privacy and sampling regimes.
\end{proof}

\begin{lemma}[Monotonicity in the sampling weight]
  \label{appendix-lem:unified-beta-monotonicity}
  Under \cref{ass:regular}, fix \(\theta_0\in\Theta\).
  If \(0\le\beta_1\le\beta_2<1\), then
  \begin{equation}
    \label{appendix-eq:unified-region-nesting}
    \caU_{\theta_0,\beta_1}
    \subseteq
    \caU_{\theta_0,\beta_2}.
  \end{equation}
  Moreover, \(\beta\mapsto\caU_{\theta_0,\beta}\) is locally Lipschitz in
  Hausdorff distance on \([0,1)\).
  If \(I(\theta_0)\succ0\), then for every \(W\succeq0\),
  \begin{equation}
    \label{appendix-eq:unified-constant-monotonicity}
    \mfC_{\theta_0}^{(\beta_2)}(W)
    \le
    \mfC_{\theta_0}^{(\beta_1)}(W),
    \qquad
    0\le\beta_1\le\beta_2\le1,
  \end{equation}
  and
  \begin{equation}
    \label{appendix-eq:unified-constant-fisher-limit}
    \lim_{\beta\uparrow1}
    \uCt(W)
    =
    \mfC_{\theta_0}^{(1)}(W)
    =
    \Tr\{WI(\theta_0)^{-1}\}.
  \end{equation}
\end{lemma}

\begin{proof}
  As in the proof of
  \cref{appendix-prop:unified-region-geometry}\textup{(iii)}, translate every
  diameter-one statistic so that it is centered.
  Then \(V_{\Phi,0}\preceq I_d/4\), and hence
  \[
    \Sigma_{\Phi,\theta_0,\beta_1}
    -
    \Sigma_{\Phi,\theta_0,\beta_2}
    =
    (\beta_2-\beta_1)
    \xk{\frac12I_d-V_{\Phi,0}}
    \succeq0.
  \]
  Therefore
  \[
    H_{\Phi,\theta_0,\beta_1}
    \preceq
    H_{\Phi,\theta_0,\beta_2}.
  \]
  Since \(\caU_{\theta_0,\beta_2}\) is downward closed when
  \(\beta_2<1\), every generator at \(\beta_1\) belongs to the region at
  \(\beta_2\).
  Taking closures proves \cref{appendix-eq:unified-region-nesting}.

  On every interval \([0,\bar\beta]\) with \(\bar\beta<1\), the matrices
  \(\Sigma_{\Phi,\theta_0,\beta}\) are uniformly bounded below by
  \((1-\bar\beta)I_d/2\).
  The inverse identity, the bound \(V_{\Phi,0}\preceq I_d/4\), and the
  uniform bound on \(A_{\Phi,0}\) show that
  \[
    \norm{
      H_{\Phi,\theta_0,\beta_1}
      -H_{\Phi,\theta_0,\beta_2}
    }_{\mathrm{op}}
    \le
    C_{\bar\beta,\theta_0}|\beta_1-\beta_2|
  \]
  uniformly over all admissible statistics.
  Matching each generator with the same statistic at the other value of
  \(\beta\), and then taking closures, proves the Hausdorff claim.

  Region nesting and inverse monotonicity give
  \cref{appendix-eq:unified-constant-monotonicity} when \(\beta_2<1\).
  For \(\beta_2=1\), \cref{appendix-eq:unified-fisher-envelope} gives
  \[
    \uCt(W)
    \ge
    (2-\beta)\Tr\{WI(\theta_0)^{-1}\}
    \ge
    \mfC_{\theta_0}^{(1)}(W),
    \qquad \beta<1,
  \]
  where the last identity follows from
  \cref{appendix-lem:unified-info-identification}\textup{(ii)}.
  This extends monotonicity to \(\beta_2=1\) and gives the lower bound in
  \cref{appendix-eq:unified-constant-fisher-limit}.

  For the reverse bound, use the bounded truncated-score statistics
  \(\Phi_B\) constructed in the proof of
  \cref{appendix-lem:unified-info-identification}\textup{(ii)}.
  For every sufficiently large fixed \(B\), their information matrices are
  positive definite and continuous in \(\beta\) at one.
  Their values at \(\beta=1\) converge to \(I(\theta_0)\) as
  \(B\to\infty\).
  First letting \(\beta\uparrow1\) and then \(B\to\infty\) proves the
  matching upper bound.
\end{proof}

\begin{remark}[The sampling regime need not inherit the interior geometry]
  \label{appendix-rem:unified-sampling-regime-geometry}
  The restrictions \(\beta_2<1\) in
  \cref{appendix-eq:unified-region-nesting} and \(\beta<1\) in the convexity and
  downward-closure statements are essential.
  In a one-dimensional nondegenerate Bernoulli model, every nonconstant
  statistic generates the full Fisher information at \(\beta=1\), whereas a
  constant statistic generates zero.
  Hence
  \[
    \caU_{\theta_0,1}=\{0,I(\theta_0)\},
  \]
  which is neither convex nor downward closed.
  Thus the regions need not converge in Hausdorff distance to
  \(\caU_{\theta_0,1}\), even though their variational constants converge by
  \cref{appendix-eq:unified-constant-fisher-limit}.
\end{remark}

\begin{lemma}[Identification in the privacy and sampling regimes]
  \label{appendix-lem:unified-info-identification}
  Under \cref{ass:regular}, fix \(\theta_0\in\Theta\) with \(I(\theta_0)\succ0\).
  Then, for every \(W\succeq0\), the following statements hold.
  \begin{enumerate}[label=(\roman*)]
    \item At \(\beta=0\),
    \[
      \caU_{\theta_0,0}=2\caJ_{\theta_0},
      \qquad
      \mfC_{\theta_0}^{(0)}(W)=\mfC_{\theta_0}(W).
    \]
    \item At \(\beta=1\),
    \[
      \mfC_{\theta_0}^{(1)}(W)
      =
      \Tr\{WI(\theta_0)^{-1}\}.
    \]
  \end{enumerate}
\end{lemma}

\begin{proof}
  At \(\beta=0\), every matrix generating \(\caU_{\theta_0,0}\) is
  \[
    A_{\Phi,0}^\top
    \xk{\frac12I_d}^{-1}
    A_{\Phi,0}
    =
    2A_{\Phi,0}^\top A_{\Phi,0}.
  \]
  Since \(\caJ_{\theta_0}\) and \(\caU_{\theta_0,0}\) use the same finite-dimensional generating class, taking closures gives \(\caU_{\theta_0,0}=2\caJ_{\theta_0}\), and the first constant identity follows from
  \(\caI_W(2J)=\caI_W(J)/2\).

  It remains to identify the sampling regime.
  Translate \(\Phi\), if necessary, so that it is centered under \(P_{\theta_0}\).
  Then
  \(\ker(V_{\Phi,0})\subseteq\ker(A_{\Phi,0}^\top)\), and the ordinary information inequality gives
  \begin{equation}
    \label{appendix-eq:unified-sampling-regime-upper-bound}
    A_{\Phi,0}^\top
    V_{\Phi,0}^{\dagger}
    A_{\Phi,0}
    \preceq
    I(\theta_0).
  \end{equation}
  The inequality is preserved under closure, so every
  \(H\in\caU_{\theta_0,1}\) satisfies \(H\preceq I(\theta_0)\).

  To attain this upper bound, write \(s_0=s_{\theta_0}\) and, for \(B>0\), set
  \[
    U_B(x)=s_0(x)\ind{\norm{s_0(x)}_2\le B},
    \qquad
    \Phi_B(x)=\frac{U_B(x)}{2B}.
  \]
  Then \(\diam(\Phi_B)\le1\).
  With
  \[
    I_B
    =
    \E_{\theta_0}
    \zk{s_0(X)s_0(X)^\top
    \ind{\norm{s_0(X)}_2\le B}},
    \qquad
    V_B=\Var_{\theta_0}\{U_B(X)\},
  \]
  DQM and dominated convergence give
  \[
    I_B\longrightarrow I(\theta_0),
    \qquad
    V_B\longrightarrow I(\theta_0).
  \]
  Moreover,
  \[
    A_{\Phi_B,0}=\frac{I_B}{2B},
    \qquad
    V_{\Phi_B,0}=\frac{V_B}{4B^2}.
  \]
  Since \(I(\theta_0)\succ0\), both \(I_B\) and \(V_B\) are nonsingular for all sufficiently large \(B\), and
  \[
    A_{\Phi_B,0}^\top
    V_{\Phi_B,0}^{\dagger}
    A_{\Phi_B,0}
    =
    I_BV_B^{-1}I_B
    \longrightarrow
    I(\theta_0).
  \]
  Thus \(I(\theta_0)\in\caU_{\theta_0,1}\).
  Combining this fact with \cref{appendix-eq:unified-sampling-regime-upper-bound} and inverse monotonicity yields
  \[
    \inf_{H\in\caU_{\theta_0,1}}\caI_W(H)
    =
    \Tr\{WI(\theta_0)^{-1}\},
  \]
  proving part~(ii).
\end{proof}

\begin{lemma}[Continuity of unified constants]
  \label{appendix-lem:unified-constant-continuity}
  Fix \(\beta\in[0,1]\), \(W\succeq0\) with \(W\ne0\), and a nonempty compact set \(K\subset\Theta\).
  Under \cref{ass:regular,ass:score-density-l1}, suppose \(I(\theta)\succ0\) for every \(\theta\in K\).
  If \(\beta=1\), suppose additionally that \(\theta\mapsto I(\theta)\) is continuous on \(K\).
  Then \(\theta\mapsto\uC_\theta(W)\) is continuous and
  \[
    0<
    \min_{\theta\in K}\uC_\theta(W)
    \le
    \max_{\theta\in K}\uC_\theta(W)
    <\infty.
  \]
\end{lemma}

\begin{proof}
  Choose a compact neighborhood \(K^+\) of \(K\) contained in \(\Theta\).
  Translate every diameter-one statistic so that \(\norm{\Phi(x)}_2\le1\), without changing its mean derivative or covariance.
  For \(\theta,\vartheta\in K^+\), put
  \[
    d_0(\theta,\vartheta)
    =
    \int\abs{p_\theta-p_\vartheta}\,d\mu,
    \qquad
    d_1(\theta,\vartheta)
    =
    \int\norm{\dot p_\theta-\dot p_\vartheta}_2\,d\mu.
  \]
  Uniformly over these statistics,
  \[
    \norm{A_{\Phi,\theta}-A_{\Phi,\vartheta}}_{\mathrm{op}}
    \le d_1(\theta,\vartheta),
    \qquad
    \norm{V_{\Phi,\theta}-V_{\Phi,\vartheta}}_{\mathrm{op}}
    \le3d_0(\theta,\vartheta).
  \]
  Both moduli vanish uniformly as \(\vartheta\to\theta\) on \(K^+\).

  Suppose first that \(\beta<1\).
  Since \(\Sigma_{\Phi,\theta,\beta}\succeq(1-\beta)I/2\), the inverse identity for positive-definite matrices and the preceding bounds show that the generators
  \[
    A_{\Phi,\theta}^\top
    \Sigma_{\Phi,\theta,\beta}^{-1}
    A_{\Phi,\theta}
  \]
  vary uniformly continuously with \(\theta\).
  They are also uniformly bounded on \(K^+\).
  Hence the compact regions \(\caU_{\theta,\beta}\) vary continuously in Hausdorff distance.
  As in the standard compactness argument, lower semicontinuity of \(\caI_W\) gives lower semicontinuity of \(\uC_\theta(W)\).
  Uniform boundedness of the regions and \(W\ne0\) give a common positive lower bound, while a positive-definite generator exists because \(I(\theta)\succ0\).

  For the reverse semicontinuity, fix \(\theta\in K\).
  By \cref{appendix-lem:unified-near-optimal-stat}, for every \(\delta>0\) there is a diameter-one \(p\)-dimensional statistic with nonsingular \(A_{\Phi,\theta}\) whose inverse-map risk is at most \(\uC_\theta(W)+\delta\).
  The derivative remains nonsingular nearby, and the risk is continuous there.
  This proves upper semicontinuity, and hence continuity, when \(\beta<1\).
  Compactness of \(K\) gives the asserted finite upper bound.

  If \(\beta=1\), part~\textup{(ii)} of \cref{appendix-lem:unified-info-identification} gives
  \[
    \mfC_\theta^{(1)}(W)
    =
    \Tr\{WI(\theta)^{-1}\}.
  \]
  The assumed continuity and positive definiteness of \(I(\theta)\) imply continuity, strict positivity, and boundedness of this function on \(K\).
\end{proof}

\subsection{Proof of \Cref{prop:common-efficiency}}

Since \(\caI_{W_1+W_2}(H)=\caI_{W_1}(H)+\caI_{W_2}(H)\), taking infima over \(H\in\caU_{\theta_0,\beta}\) gives the inequality.
A common minimizer gives equality.
Conversely, if equality holds, choose a minimizer for \(W_1+W_2\), which exists by the compactness of \(\caU_{\theta_0,\beta}\) and lower semicontinuity.
It must minimize both summands and hence belongs to the intersection.

For the last assertion, let
\[
  \mathcal{R}_n^{(\beta)}(W;\widehat{\theta}_n)
  =
  \sigma_n^{-2}
  \sup_{\theta\in\widetilde\Theta_n}
  \E_\theta
  \norm{\widehat{\theta}_n-\theta}_W^2.
\]
By \cref{thm:local-minimax}, every sequence of \(\rho_n\)-zCDP estimators satisfies
\[
  \liminf_{n\to\infty}
  \mathcal{R}_n^{(\beta)}(W_1+W_2;\widehat{\theta}_n)
  \ge
  \uCt(W_1+W_2).
\]
On the other hand,
\[
  \mathcal{R}_n^{(\beta)}(W_1+W_2;\widehat{\theta}_n)
  \le
  \mathcal{R}_n^{(\beta)}(W_1;\widehat{\theta}_n)
  +
  \mathcal{R}_n^{(\beta)}(W_2;\widehat{\theta}_n).
\]
If the same sequence attained both separate constants, then
\[
  \limsup_{n\to\infty}
  \mathcal{R}_n^{(\beta)}(W_1+W_2;\widehat{\theta}_n)
  \le
  \uCt(W_1)
  +
  \uCt(W_2),
\]
contradicting the strict inequality implied by disjointness.
   \clearpage
  \section{Moment-Map Inverse Estimators with Gaussian Releases}
\label{appendix-sec:attain-proofs}

This section develops the common local estimator for every limiting balance between sampling and privacy noise.
A single fixed-statistic expansion keeps sampling and Gaussian privacy fluctuations on the common scale.
We next approximate the variational constant by bounded statistics with invertible mean derivatives and then obtain the upper bounds from one diagonal construction, provided the public center is fixed.

\subsection{Fixed-statistic expansion}
\label{appendix-subsec:fixed-stat-expansion}

\begin{lemma}[Fixed-statistic moment-map inverse estimator]
  \label{appendix-lem:unified-fixed-stat}
  Under \cref{ass:regular,ass:score-density-l1}, fix
  \(\theta_0\in\Theta\).
  Let \(\Phi:\caX\to\R^p\) have diameter at most one, and suppose its mean map
  \(g(\theta)=\E_\theta\Phi(X)\) is continuously differentiable near
  \(\theta_0\), with nonsingular derivative
  \(A_0=A_{\Phi,\theta_0}\).
  Let \((\rho_n)\) satisfy the asymptotic regime in
  \cref{eq:unified-asymp-regime}, and let
  \(\mathcal T_n\subset\Theta\) satisfy
  \[
    \sup_{\theta\in\mathcal T_n}
    \norm{\theta-\theta_0}_2
    \longrightarrow0.
  \]
  Release
  \[
    Y_n
    =
    \frac1n\sum_{i=1}^n\Phi(X_i)+G_n,
    \qquad
    G_n
    \sim
    N\left(0,\frac{I_p}{2\tau_n^2}\right),
  \]
  and let
  \(\widehat\theta_{\Phi,n}=F_0(Y_n)\), where
  \(F_0=\Psi_0\circ\Pi_{\mathcal B_0}\) is a projected local inverse of \(g\)
  around \(\theta_0\).
  Then \(\widehat\theta_{\Phi,n}\) is \(\rho_n\)-zCDP and
  \begin{equation}
    \label{appendix-eq:unified-fixed-stat-linearization}
    \sup_{\theta\in\mathcal T_n}
    \sigma_n^{-2}
    \E_\theta
    \norm{
      \widehat\theta_{\Phi,n}-\theta
      -
      A_0^{-1}\{Y_n-g(\theta)\}
    }_2^2
    \longrightarrow0.
  \end{equation}
  Consequently, for every \(W\succeq0\),
  \begin{equation}
    \label{appendix-eq:unified-fixed-stat-risk}
    \sup_{\theta\in\mathcal T_n}
    \left|
      \sigma_n^{-2}
      \E_\theta
      (\widehat\theta_{\Phi,n}-\theta)^\top
      W(\widehat\theta_{\Phi,n}-\theta)
      -
      \Tr\zk{
        W A_0^{-1}
        \Sigma_{\Phi,\theta_0,\beta}
        A_0^{-\top}
      }
    \right|
    \longrightarrow0.
  \end{equation}
  Moreover,
  \begin{equation}
    \label{appendix-eq:unified-fixed-stat-limit}
    \sup_{\theta\in\mathcal T_n}
    d_{\mathrm{BL}}\left(
      \caL_\theta\{
        \sigma_n^{-1}
        (\widehat\theta_{\Phi,n}-\theta)
      \},
      N\left(
        0,
        A_0^{-1}
        \Sigma_{\Phi,\theta_0,\beta}
        A_0^{-\top}
      \right)
    \right)
    \longrightarrow0.
  \end{equation}
\end{lemma}

\begin{proof}
  Translate \(\Phi\) by a fixed vector, if necessary, so that
  \(\norm{\Phi(x)}_2\le1\); this leaves its diameter, derivative, covariance,
  and the recentered estimator unchanged.
  The empirical query has replacement sensitivity at most \(1/n\), so the
  Gaussian mechanism gives \(\rho_n\)-zCDP.
  By the inverse function theorem, choose a closed image ball
  \(\mathcal B_0\) on which the local inverse \(\Psi_0\) is defined.
  The projected inverse \(F_0=\Psi_0\circ\Pi_{\mathcal B_0}\) is globally
  Lipschitz.

  Write
  \[
    Z_{n,\theta}
    =
    Y_n-g(\theta)
    =
    \frac1n\sum_{i=1}^n
    \{\Phi(X_i)-g(\theta)\}
    +
    G_n.
  \]
  The exact covariance identity is
  \[
    \Var_\theta(Z_{n,\theta})
    =
    \frac{V_{\Phi,\theta}}n
    +
    \frac{I_p}{2\tau_n^2}
    =
    \sigma_n^2
    \Sigma_{\Phi,\theta,\beta_n}.
  \]
  Boundedness of \(\Phi\) and Gaussian moment bounds give, uniformly on
  \(\mathcal T_n\),
  \[
    \E_\theta\norm{Z_{n,\theta}}_2^2
    =
    O(\sigma_n^2),
    \qquad
    \E_\theta\norm{Z_{n,\theta}}_2^4
    =
    O(\sigma_n^4).
  \]
  Since \(\mathcal T_n\) shrinks to \(\theta_0\), there is a fixed
  \(r>0\) such that \(g(\theta)\) lies at distance at least \(r\) from the
  boundary of \(\mathcal B_0\) for every \(\theta\in\mathcal T_n\) and all
  sufficiently large \(n\).
  Hence the projection is inactive on
  \(\{\norm{Z_{n,\theta}}_2\le r\}\), while the fourth-moment bound gives
  \[
    \sup_{\theta\in\mathcal T_n}
    P_\theta\{\norm{Z_{n,\theta}}_2>r\}
    =O(\sigma_n^4).
  \]
  On the first event, uniform differentiability of \(\Psi_0\), together
  with
  \(D\Psi_0\{g(\theta)\}=A_{\Phi,\theta}^{-1}\to A_0^{-1}\), gives the
  desired linearization.
  On the complement, the global Lipschitz property of \(F_0\) and the
  fourth-moment bound make the squared remainder \(o(\sigma_n^2)\)
  uniformly.
  Therefore
  \[
    \widehat\theta_{\Phi,n}-\theta
    =
    A_0^{-1}Z_{n,\theta}
    +
    o_{L^2}(\sigma_n)
  \]
  uniformly on \(\mathcal T_n\), proving
  \cref{appendix-eq:unified-fixed-stat-linearization}.

  Continuity of \(V_{\Phi,\theta}\) at \(\theta_0\), the covariance identity,
  and Cauchy--Schwarz give
  \cref{appendix-eq:unified-fixed-stat-risk}.
  Finally, write the empirical term as
  \[
    \sigma_n^{-1}
    \frac1n\sum_{i=1}^n
    \{\Phi(X_i)-g(\theta)\}
    =
    \sqrt{\beta_n}
    \frac1{\sqrt n}\sum_{i=1}^n
    \{\Phi(X_i)-g(\theta)\}.
  \]
  The variables are uniformly bounded, and
  \(P_\theta\to P_{\theta_0}\) in total variation uniformly for
  \(\theta\in\mathcal T_n\).
  Thus the Cram\'er--Wold Lindeberg argument applies uniformly and gives
  \[
    \sigma_n^{-1}
    \frac1n\sum_{i=1}^n
    \{\Phi(X_i)-g(\theta)\}
    \rightsquigarrow
    N(0,\beta V_{\Phi,\theta_0})
  \]
  uniformly on \(\mathcal T_n\).
  Independently,
  \[
    \sigma_n^{-1}G_n
    \sim
    N\left(
      0,
      \frac{1-\beta_n}{2}I_p
    \right).
  \]
  Combining the two limits with the uniform linearization proves
  \cref{appendix-eq:unified-fixed-stat-limit}.
\end{proof}

\subsection{Near-optimal invertible moment maps}
\label{appendix-subsec:near-optimal-moment-map}

\begin{lemma}[Approximation by invertible statistics]
  \label{appendix-lem:unified-near-optimal-stat}
  Under \cref{ass:regular,ass:score-density-l1}, suppose
  \(I(\theta_0)\succ0\).
  For every \(\beta\in[0,1]\), \(W\succeq0\), and \(\epsilon>0\), there is
  a diameter-one statistic \(\Phi:\caX\to\R^p\) whose mean map is
  continuously differentiable near \(\theta_0\), whose derivative
  \(A_{\Phi,\theta_0}\) is nonsingular, and such that
  \begin{equation}
    \label{appendix-eq:unified-near-optimal-stat}
    \Tr\zk{
      W A_{\Phi,\theta_0}^{-1}
      \Sigma_{\Phi,\theta_0,\beta}
      A_{\Phi,\theta_0}^{-\top}
    }
    \le
    \uCt(W)+\epsilon.
  \end{equation}
\end{lemma}

\begin{proof}
  Suppose first that \(\beta<1\).
  By \cref{appendix-prop:unified-region-geometry}\textup{(v)}, choose an
  optimizer \(H_*\in\caU_{\theta_0,\beta}\) and a matrix
  \(H_0\in\caU_{\theta_0,\beta}\) with \(H_0\succ0\).
  For \(\delta\in(0,1)\), convexity in
  \cref{appendix-prop:unified-region-geometry}\textup{(iv)} gives
  \[
    H_\delta=(1-\delta)H_*+\delta H_0
    \in\caU_{\theta_0,\beta},
    \qquad H_\delta\succ0.
  \]
  Since \(H_\delta\succeq(1-\delta)H_*\), inverse monotonicity and lower
  semicontinuity yield
  \[
    \caI_W(H_*)
    \le
    \liminf_{\delta\downarrow0}\caI_W(H_\delta)
    \le
    \limsup_{\delta\downarrow0}\caI_W(H_\delta)
    \le
    \lim_{\delta\downarrow0}
    \frac{\caI_W(H_*)}{1-\delta}
    =
    \caI_W(H_*).
  \]
  Choose \(\delta\) so that
  \(\caI_W(H_\delta)\le
  \uCt(W)+\epsilon/2\).

  By the \(p\)-dimensional representation in
  \cref{appendix-prop:unified-region-geometry}\textup{(ii)}, there are
  diameter-one statistics \(\Phi_m:\caX\to\R^p\) whose generated matrices
  \[
    H_m
    =
    A_{\Phi_m,\theta_0}^\top
    \Sigma_{\Phi_m,\theta_0,\beta}^{-1}
    A_{\Phi_m,\theta_0}
  \]
  converge to \(H_\delta\).
  Since \(H_\delta\succ0\), for all sufficiently large \(m\),
  \(H_m\succ0\), and hence the square matrix
  \(A_{\Phi_m,\theta_0}\) is nonsingular.
  For such \(m\),
  \[
    A_{\Phi_m,\theta_0}^{-1}
    \Sigma_{\Phi_m,\theta_0,\beta}
    A_{\Phi_m,\theta_0}^{-\top}
    =H_m^{-1}.
  \]
  Continuity of the inverse on \(\bbS_{++}^p\) gives
  \[
    \Tr(WH_m^{-1})
    \longrightarrow
    \caI_W(H_\delta).
  \]
  Translate \(\Phi_m\), if necessary, so that it is bounded.
  The score-density assumption makes its mean map continuously
  differentiable near \(\theta_0\).
  Taking \(m\) sufficiently large proves
  \cref{appendix-eq:unified-near-optimal-stat} for every \(\beta<1\).

  It remains to treat \(\beta=1\).
  Use the truncated score statistic \(\Phi_B\) from the proof of
  \cref{appendix-lem:unified-info-identification}\textup{(ii)}.
  With the notation there,
  \[
    A_{\Phi_B,\theta_0}=\frac{I_B}{2B},
    \qquad
    V_{\Phi_B,\theta_0}=\frac{V_B}{4B^2},
  \]
  and hence its inverse-map covariance is
  \(I_B^{-1}V_BI_B^{-1}\).
  This converges to \(I(\theta_0)^{-1}\), while part~\textup{(ii)} of the
  same lemma identifies the limiting trace with
  \(\mfC_{\theta_0}^{(1)}(W)\).
  For all sufficiently large \(B\), \(I_B\) is nonsingular; moreover,
  boundedness of \(\Phi_B\) and the score-density assumption make its mean
  map continuously differentiable near \(\theta_0\).
  Taking \(B\) sufficiently large completes the proof.
\end{proof}

\subsection{Local attaining sequences}
\label{appendix-subsec:unified-local-upper}

\begin{lemma}[Local upper bound]
  \label{appendix-lem:unified-local-upper}
  Under the assumptions of \cref{thm:local-minimax}, there exists a
  sequence \((\widehat\theta_n^\star)\) of \(\rho_n\)-zCDP moment-map inverse
  estimators such that
  \begin{equation}
    \label{appendix-eq:unified-local-upper}
    \limsup_{n\to\infty}
    \sigma_n^{-2}
    \sup_{\theta\in\widetilde\Theta_n}
    \E_\theta
    (\widehat\theta_n^\star-\theta)^\top
    W(\widehat\theta_n^\star-\theta)
    \le
    \uCt(W).
  \end{equation}
\end{lemma}

\begin{proof}
  For each \(k\), apply
  \cref{appendix-lem:unified-near-optimal-stat} with
  \(\epsilon_k\downarrow0\), and use the corresponding estimator from
  \cref{appendix-lem:unified-fixed-stat} with
  \(\mathcal T_n=\widetilde\Theta_n\).
  For fixed \(k\), its normalized maximal risk converges to at most
  \(\uCt(W)+\epsilon_k\).
  Choose increasing deterministic thresholds \(N_k\) beyond which this risk
  is at most \(\uCt(W)+2\epsilon_k\), and use the
  \(k\)th estimator when \(N_k\le n<N_{k+1}\).
  This slow diagonal sequence remains \(\rho_n\)-zCDP and proves the claim.
\end{proof}

\begin{lemma}[Gaussian diagonalization]
  \label{appendix-lem:unified-gauss-diagonalization}
  Under the assumptions of \cref{thm:local-attain-seq}, suppose \(W\succ0\), and fix \(H_{W,\beta}\in\argmin_{H\in\caU_{\theta_0,\beta}}\caI_W(H)\).
  There is a sequence \((\widehat\theta_n^\star)\) of \(\rho_n\)-zCDP moment-map inverse estimators such that
  \[
    \limsup_{n\to\infty}
    \sigma_n^{-2}
    \sup_{\theta\in\widetilde\Theta_n}
    \E_\theta
    (\widehat\theta_n^\star-\theta)^\top
    W(\widehat\theta_n^\star-\theta)
    \le
    \uCt(W)
  \]
  and
  \[
    \sup_{\theta\in\widetilde\Theta_n}
    d_{\mathrm{BL}}\left(
      \caL_\theta\{\sigma_n^{-1}(\widehat\theta_n^\star-\theta)\},
      N\xk{0,H_{W,\beta}^{-1}}
    \right)
    \longrightarrow0.
  \]
\end{lemma}

\begin{proof}
  Since \(W\succ0\), the optimizer \(H_{W,\beta}\) is positive definite.
  By \cref{appendix-prop:unified-region-geometry}\textup{(ii)}, there are diameter-one statistics \(\Phi_k:\caX\to\R^p\) whose information matrices
  \[
    H_k
    =
    A_{\Phi_k,\theta_0}^\top
    \Sigma_{\Phi_k,\theta_0,\beta}^{\dagger}
    A_{\Phi_k,\theta_0}
  \]
  converge to \(H_{W,\beta}\).
  Translate them so that they are bounded.
  After discarding finitely many terms, \(H_k\succ0\), so the square derivatives \(A_{\Phi_k,\theta_0}\) are nonsingular.
  \[
    A_{\Phi_k,\theta_0}^{-1}
    \Sigma_{\Phi_k,\theta_0,\beta}
    A_{\Phi_k,\theta_0}^{-\top}
    =
    H_k^{-1}.
  \]
  The fixed-statistic limit in \cref{appendix-eq:unified-fixed-stat-limit} therefore applies.

  Choose deterministic blocks slowly enough that both the bounded-Lipschitz error and the scaled-risk error for the \(k\)th statistic are at most \(\epsilon_k\downarrow0\).
  \[
    N\xk{0,H_k^{-1}}
    \rightsquigarrow
    N\xk{0,H_{W,\beta}^{-1}},
    \qquad
    \Tr(WH_k^{-1})
    \longrightarrow
    \uCt(W),
  \]
  The preceding limits show that the block sequence is both attaining and uniformly asymptotically normal, proving the lemma.
\end{proof}
   \clearpage
  \section{Upper Bounds over Compact Parameter Set}
\label{appendix-sec:compact-param-upper}

This section extends the previous local estimator to compact parameter sets.
We first construct a bounded identifying statistic for private preliminary localization.
We then combine finite near-optimal charts with a two-stage Gaussian release to obtain the uniform upper bound.

\subsection{Preliminary localization}
\label{appendix-subsec:preliminary-localization}

\subsubsection{Bounded identification on a compact parameter set}
\label{appendix-subsubsec:bounded-identifier}

\begin{proof}[Proof of \cref{lem:bounded-identifier}]
We first construct finitely many bounded statistics that separate nearby parameter values.
Fix \(\theta\in K\).
By dominated convergence, the matrices
\[
  B_{\theta,M}
  =
  \E_\theta \left[
    s_\theta(X)s_\theta(X)^\top
    \ind{\norm{s_\theta(X)}_2 \le M}
  \right]
\]
converge to \(I(\theta)\) as \(M\to\infty\).
Because \(I(\theta)\succ0\), the matrix \(B_{\theta,M}\) is positive definite for all sufficiently large \(M\).
For such an \(M\), define
\[
  \Phi_\theta(x)
  =
  \frac{1}{2M}
  s_\theta(x)
  \ind{\norm{s_\theta(x)}_2 \le M}.
\]
This is a bounded measurable \(\R^p\)-valued statistic with diameter at most one.
Its mean map
\[
  g_\theta(\vartheta)
  =
  \E_\vartheta \Phi_\theta(X)
\]
is continuously differentiable on \(\Theta\) by \cref{ass:score-density-l1}, and
\[
  Dg_\theta(\theta)
  =
  \E_\theta \left[
    \Phi_\theta(X)s_\theta(X)^\top
  \right]
  =
  \frac{1}{2M}B_{\theta,M}
\]
is nonsingular.
The inverse function theorem therefore gives an open neighborhood \(U_\theta \subset\Theta\) on which \(g_\theta\) is one-to-one.

Compactness of \(K\) gives a finite subcover
\[
  K\subset\bigcup_{j=1}^J U_j,
\]
where \(U_j=U_{\theta_j}\), and we write \(\Phi_j=\Phi_{\theta_j}\).
The cover \(\{U_j \cap K:1\le j\le J\}\) has a Lebesgue number \(\delta>0\).
Consequently, whenever \(\theta,\vartheta\in K\) satisfy
\(\norm{\theta-\vartheta}_2<\delta\), both points belong to some common \(U_j\), and hence
\[
  \theta\ne\vartheta
  \quad\Longrightarrow\quad
  \E_\theta \Phi_j(X)\ne\E_\vartheta \Phi_j(X).
\]
Thus the finite collection \(\Phi_1,\ldots,\Phi_J\) separates all distinct nearby parameter values.

It remains to separate parameter pairs at distance at least \(\delta\).
The set
\[
  F_\delta
  =
  \left\{
    (\theta,\vartheta)\in K\times K:
    \norm{\theta-\vartheta}_2 \ge\delta
  \right\}
\]
is compact.
For every \((\theta,\vartheta)\in F_\delta\), identifiability gives
\(P_\theta \ne P_\vartheta\), so there is a measurable set
\(A_{\theta,\vartheta} \subset\caX\) such that
\[
  P_\theta(A_{\theta,\vartheta})
  \ne
  P_\vartheta(A_{\theta,\vartheta}).
\]
For every fixed measurable \(A\), the map
\[
  \eta\longmapsto P_\eta(A)
  =
  \int\ind{x\in A}p_\eta(x)d\mu(x)
\]
is continuous by the \(L^1(\mu)\) continuity in
\cref{ass:score-density-l1}.
Hence the strict inequality for \(A_{\theta,\vartheta}\) persists on an open neighborhood of \((\theta,\vartheta)\) in \(K\times K\).
Compactness of \(F_\delta\) gives finitely many measurable sets
\(A_1,\ldots,A_L\) whose indicators separate every pair in \(F_\delta\).

Concatenate the local statistics and separating indicators:
\[
  T_0(x)
  =
  \left(
    \Phi_1(x),\ldots,\Phi_J(x),
    \ind{x\in A_1},\ldots,\ind{x\in A_L}
  \right)
  \in\R^{Jp+L}.
\]
This statistic is bounded and measurable.
Its mean map is one-to-one on \(K\): a pair at distance below \(\delta\) is separated by one of the \(\Phi_j\), while a pair in \(F_\delta\) is separated by one of the indicators.
Moreover,
\[
  \diam(T_0)^2
  \le
  J+L.
\]
Thus \(T=T_0/\sqrt{J+L}\) has diameter at most one and retains the same identifying property.
Finally, boundedness of \(T\) and \cref{ass:score-density-l1} show that
\[
  h(\eta)=\E_\eta T(X),
  \qquad
  Dh(\eta)v
  =
  \int T(x)\ang{\dot{p}_\eta(x),v}d\mu(x),
\]
with derivative continuous in \(\eta\).
\end{proof}

\subsubsection{Proof of \Cref{lem:private-localization}}
\label{appendix-subsubsec:private-localization}

The argmin correspondence in \cref{eq:private-localization-estimator} has nonempty compact values.
The measurable maximum theorem gives a Borel measurable selection, because its criterion is jointly measurable and continuous in \(\vartheta\), and \(K\) is compact.
Fix such a selection for \(\widetilde\theta\).
Because \(T\) has finite diameter, a fixed translation makes it bounded; this shifts \(Y_1\) and \(h\) by the same amount and leaves the estimator in \cref{eq:private-localization-estimator} unchanged.
The query has replacement sensitivity \(\Delta_T/n_1\), so \cref{lem:gauss-composition} gives \(\rho_1\)-zCDP.

Fix \(r>0\).
If
\[
  \dk{
    (\theta,\vartheta)\in K\times K:
    \norm{\theta-\vartheta}_2\ge r
  }
  =\varnothing,
\]
then the error probability is zero.
Otherwise, continuity and injectivity of \(h\) on the compact set \(K\) imply
\[
a_r
=
\inf_{\substack{\theta,\vartheta\in K\\
                  \norm{\theta-\vartheta}_2 \ge r}}
\norm{h(\theta)-h(\vartheta)}_2
>0.
\]
If \(\norm{Y_1-h(\theta)}_2<a_r/3\), the criterion at every \(\vartheta\) with \(\norm{\vartheta-\theta}_2 \ge r\) exceeds \(2a_r/3\), while its value at \(\theta\) is below \(a_r/3\).
Hence \(\norm{\widetilde{\theta}-\theta}_2<r\).

Each coordinate of \(T\) has range at most \(\Delta_T\).
Fixed-dimensional Hoeffding and Gaussian tail bounds therefore give constants \(c,C>0\), independent of \(\theta\in K\), such that
\[
\sup_{\theta\in K}
P_\theta\{\norm{Y_1-h(\theta)}_2 \ge a_r/3\}
\le
C\exp(-cn_1)+C\exp(-cn_1^2 \rho_1).
\]
The two growth conditions in the statement make the right-hand side \(o(a_n^{-2})\), proving the claim.

\subsection{Finite near-optimal charts}
\label{appendix-subsec:unified-compact-charts}

\begin{lemma}[Finite near-optimal charts]
  \label{appendix-lem:unified-compact-charts}
  Fix \(\beta\in[0,1]\), \(W\succeq0\) with \(W\ne0\), and a nonempty compact set \(K\subset\Theta\).
  Under \cref{ass:regular,ass:score-density-l1}, suppose \(I(\theta)\succ0\) for every \(\theta\in K\).
  If \(\beta=1\), suppose additionally that \(\theta\mapsto I(\theta)\) is continuous on \(K\).
  For every \(\epsilon>0\), there are open sets \(V_j\) with
  \(\overline V_j\subset\Theta\) and \(K\subset\bigcup_{j=1}^J V_j\),
  diameter-one statistics \(\Phi_j:\caX\to\R^p\), and closed Euclidean balls
  \(\mathcal B_j\).
  Writing
  \[
    g_j(\theta)=\E_\theta\Phi_j(X),
    \qquad
    A_{\Phi_j,\theta}=Dg_j(\theta),
  \]
  each \(g_j\) has a \(C^1\) inverse \(\Psi_j\) on a neighborhood of
  \(\mathcal B_j\),
  \[
    \Psi_j\{g_j(\theta)\}=\theta
    \quad\text{for }\theta\in\overline V_j,
    \qquad
    g_j(\overline V_j)\subset\operatorname{int}(\mathcal B_j),
  \]
  and
  \begin{equation}
    \label{appendix-eq:unified-compact-chart-bound}
    \sup_{\theta\in\overline V_j\cap K}
    \frac{
      \Tr\zk{
        W A_{\Phi_j,\theta}^{-1}
        \Sigma_{\Phi_j,\theta,\beta}
        A_{\Phi_j,\theta}^{-\top}
      }
    }{
      \uC_\theta(W)
    }
    \le1+\epsilon.
  \end{equation}
  In addition,
  \(F_j=\Psi_j\circ\Pi_{\mathcal B_j}\) is globally Lipschitz and has compact image.
\end{lemma}

\begin{proof}
  By \cref{appendix-lem:unified-constant-continuity}, the normalizing constant is continuous and bounded away from zero on \(K\).
  For each \(\theta\in K\), choose a diameter-one \(p\)-dimensional statistic whose derivative is nonsingular and whose risk at \(\theta\) is at most \((1+\epsilon/2)\uC_\theta(W)\).
  Its derivative and covariance are continuous nearby.
  At \(\beta=1\), nonsingularity of the derivative also ensures positive definiteness of the covariance: if \(v^\top V_{\Phi,\theta}v=0\), then \(v^\top\Phi\) is constant \(P_\theta\)-almost surely, which forces \(v^\top A_{\Phi,\theta}=0\).
  Hence the risk ratio is continuous and remains at most \(1+\epsilon\) on a sufficiently small neighborhood.

  The inverse function theorem supplies a \(C^1\) inverse \(\Psi_\theta\) on
  an open neighborhood of \(g(\theta)\).
  Choose \(0<r_1<r_2\) so that \(\Psi_\theta\) is defined on a neighborhood of
  \(\overline B(g(\theta),r_2)\), and so that
  \(\Psi_\theta\{\overline B(g(\theta),r_1)\}\) lies in the neighborhood on
  which the ratio bound holds.
  Set
  \[
    V_\theta
    =
    \Psi_\theta\{B(g(\theta),r_1)\},
    \qquad
    \mathcal B_\theta
    =
    \overline B(g(\theta),r_2).
  \]
  Then \(\overline V_\theta\subset\Theta\),
  \(g(\overline V_\theta)\subset\operatorname{int}(\mathcal B_\theta)\), and
  \(\Psi_\theta\{g(\vartheta)\}=\vartheta\) on \(\overline V_\theta\).
  Boundedness of \(D\Psi_\theta\) on the larger ball makes
  \(\Psi_\theta\circ\Pi_{\mathcal B_\theta}\) globally Lipschitz, and its
  image is compact.
  A finite subcover of \(K\) completes the proof.
\end{proof}

\begin{lemma}[Uniform fixed-chart expansion]
  \label{appendix-lem:unified-fixed-chart-expansion}
  Let \((V_j,\Phi_j,\mathcal B_j,F_j)\), \(1\le j\le J\), be a finite chart family as in \cref{appendix-lem:unified-compact-charts}, and put \(K_j=\overline V_j\cap K\).
  Let \(m_n\to\infty\) and \(\varrho_n>0\) satisfy \(m_n^2\varrho_n\to\infty\), and set
  \[
    \widetilde\sigma_n^2
    =
    \frac1{m_n}+\frac1{m_n^2\varrho_n}.
  \]
  For each \(j\), release
  \[
    Y_{j,n}
    =
    \frac1{m_n}\sum_{i=1}^{m_n}\Phi_j(X_i)+G_{j,n},
    \qquad
    G_{j,n}\sim N\left(0,\frac{I_p}{2m_n^2\varrho_n}\right),
  \]
  and define \(\widehat\theta_{j,n}=F_j(Y_{j,n})\).
  Then
  \[
    \max_{1\le j\le J}
    \sup_{\theta\in K_j}
    \widetilde\sigma_n^{-2}
    \E_\theta
    \norm{
      \widehat\theta_{j,n}-\theta
      -
      A_{\Phi_j,\theta}^{-1}
      \{Y_{j,n}-g_j(\theta)\}
    }_2^2
    \longrightarrow0.
  \]
\end{lemma}

\begin{proof}
  Write \(Z_{j,n,\theta}=Y_{j,n}-g_j(\theta)\).
  Since the chart family is finite and each statistic has diameter at most one, the empirical summands are uniformly bounded.
  Gaussian moment bounds therefore give
  \[
    \max_{1\le j\le J}
    \sup_{\theta\in K_j}
    \E_\theta\norm{Z_{j,n,\theta}}_2^4
    =
    O(\widetilde\sigma_n^4).
  \]
  The compact sets \(g_j(K_j)\) lie in \(\operatorname{int}(\mathcal B_j)\).
  Thus there is a common \(r>0\) such that every such set has distance at least \(r\) from the boundary of its corresponding ball.
  On \(\{\norm{Z_{j,n,\theta}}_2\le r\}\), projection is inactive.
  Uniform continuity of the derivatives of the finitely many inverse maps on the relevant compact neighborhoods gives the stated linearization on this event.

  On its complement, the fourth-moment bound gives probability \(O(\widetilde\sigma_n^4)\) uniformly over \(j\) and \(\theta\in K_j\).
  The compact images of the \(F_j\), together with the same moment bound, make the complementary squared remainder \(o(\widetilde\sigma_n^2)\).
  Combining the two events proves the lemma.
\end{proof}

\subsection{Two-stage Gaussian estimator}
\label{appendix-subsec:unified-compact-upper}

\begin{lemma}[Compact upper bound]
  \label{appendix-lem:unified-compact-upper}
  Fix \(W\succeq0\) with \(W\ne0\), and let \(K\subset\Theta\) be nonempty
  and compact.
  Under \cref{ass:regular,ass:score-density-l1}, suppose
  \(I(\theta)\succ0\) for every \(\theta\in K\), and suppose the model is
  identifiable on \(K\).
  Let \((\rho_n)\) satisfy the asymptotic regime in
  \cref{eq:unified-asymp-regime}.
  If \(\beta=1\), suppose additionally that
  \(\theta\mapsto I(\theta)\) is continuous on \(K\).
  Then there exists a sequence \((\widehat\theta_n^\star)\) of
  \(\rho_n\)-zCDP estimators,
  eventually of the two-stage Gaussian form below, such that
  \[
    \limsup_{n\to\infty}
    \sup_{\theta\in K}
    \frac{
      \sigma_n^{-2}
      \E_\theta
      (\widehat\theta_n^\star-\theta)^\top
      W(\widehat\theta_n^\star-\theta)
    }{
      \uC_\theta(W)
    }
    \le1.
  \]
\end{lemma}

\begin{proof}
  \noindent\textit{Resource split.}
  Put
  \[
    q_n=\sigma_n^{-2},
    \qquad
    \eta_n=q_n^{-1/4},
    \qquad
    n_1=\lfloor\eta_n n\rfloor,
    \qquad
    \rho_1=\eta_n\rho_n,
  \]
  and let \(n_2=n-n_1\) and \(\rho_2=\rho_n-\rho_1\).
  Since \(n^2\rho_n\to\infty\), we have \(q_n\to\infty\).
  Moreover,
  \[
    q_n\le n,
    \qquad
    q_n\le n^2\rho_n,
  \]
  and therefore
  \[
    n_1\sim\eta_n n,
    \qquad
    \eta_n n\ge q_n^{3/4},
    \qquad
    n_1^2\rho_1
    \sim
    \eta_n^3n^2\rho_n,
    \qquad
    \eta_n^3n^2\rho_n\ge q_n^{1/4}.
  \]
  Both quantities dominate
  \(\log(\sigma_n^{-1})=\frac12\log q_n\).
  Choose the bounded identifying statistic from
  \cref{lem:bounded-identifier}, and apply
  \cref{lem:private-localization} with target sequence
  \(\sigma_n^{-1}\).
  For every fixed \(r>0\),
  \begin{equation}
    \label{appendix-eq:unified-compact-localization}
    \sup_{\theta\in K}
    P_\theta\dk{
      \norm{\widetilde\theta-\theta}_2>r
    }
    =
    o(\sigma_n^2).
  \end{equation}

  The split spends a vanishing fraction of both resources:
  \[
    \frac{n_1}{n}\to0,
    \qquad
    \frac{\rho_1}{\rho_n}\to0.
  \]
  Define the second-stage scale and weight by
  \[
    \sigma_{2,n}^2
    =
    \frac1{n_2}+\frac1{n_2^2\rho_2},
    \qquad
    \beta_{2,n}
    =
    \frac{n_2\rho_2}{1+n_2\rho_2}.
  \]
  Then
  \[
    \frac{\sigma_{2,n}^2}{\sigma_n^2}\to1,
    \qquad
    \beta_{2,n}\to\beta.
  \]

  \noindent\textit{Chart selection and release.}
  Fix \(\epsilon>0\) and take the finite charts from
  \cref{appendix-lem:unified-compact-charts}.
  By the Lebesgue number lemma, there is \(r>0\) such that, for every
  \(t\in K\), at least one chart satisfies
  \[
    B(t,2r)\cap K\subset V_j.
  \]
  Let \(j(t)\) be the smallest such index.
  This selector is measurable because
  \[
    \dk{
      t\in K:
      B(t,2r)\cap K\subset V_j
    }
    =
    \dk{
      t\in K:
      \operatorname{dist}(t,K\setminus V_j)\ge2r
    }.
  \]
  After releasing the preliminary estimator, set
  \begin{equation}
    \label{appendix-eq:compact-two-stage-estimator}
    \begin{aligned}
      \widehat j
      &=
      j(\widetilde\theta),\\
      Y_2
      &=
      \frac1{n_2}\sum_{i=n_1+1}^n
      \Phi_{\widehat j}(X_i)
      +
      G_{2,n},
      \qquad
      G_{2,n}
      \sim
      N\left(
        0,
        \frac{I_p}{2n_2^2\rho_2}
      \right),\\
      \widehat\theta_{\epsilon,n}
      &=
      F_{\widehat j}(Y_2)
      =
      \Psi_{\widehat j}
      \left(
        \Pi_{\mathcal B_{\widehat j}}Y_2
      \right).
    \end{aligned}
  \end{equation}
  The first- and second-stage Gaussian noises are independent of the data and
  of each other.
  Conditional on every first-stage output, the selected second-stage query
  has sensitivity at most \(1/n_2\), and hence is \(\rho_2\)-zCDP.
  Adaptive composition gives total budget \(\rho_n\).

  \noindent\textit{Risk bound.}
  Let
  \[
    \mathcal L_n
    =
    \dk{
      \norm{\widetilde\theta-\theta}_2\le r
    }.
  \]
  On \(\mathcal L_n\), the selected chart contains the true parameter.
  Conditional on the first-stage release, the selected chart is fixed and
  the second-stage sample is independent.
  For each \(j\), let \(Y_{j,n}\) and \(\widehat\theta_{j,n}\) denote the second-stage release and projected inverse obtained by using \(\Phi_j\) and \(F_j\) in place of \(\Phi_{\widehat j}\) and \(F_{\widehat j}\).
  By \cref{appendix-lem:unified-fixed-chart-expansion}, applied with \(m_n=n_2\) and \(\varrho_n=\rho_2\),
  \[
    \max_{1\le j\le J}
    \sup_{\theta\in\overline V_j\cap K}
    \sigma_{2,n}^{-2}
    \E_\theta
    \norm{
      \widehat\theta_{j,n}-\theta
      -
      A_{\Phi_j,\theta}^{-1}
      \{Y_{j,n}-g_j(\theta)\}
    }_2^2
    \longrightarrow0.
  \]
  \[
    \sigma_{2,n}^{-2}
    \xk{
      \frac{V_{\Phi_j,\theta}}{n_2}
      +
      \frac{I_p}{2n_2^2\rho_2}
    }
    =
    \Sigma_{\Phi_j,\theta,\beta_{2,n}}
  \]
  The covariance identity and \(\beta_{2,n}\to\beta\) show, using the finite chart family, that
  \[
    \E_\theta\zk{
      (\widehat\theta_{\epsilon,n}-\theta)^\top
      W(\widehat\theta_{\epsilon,n}-\theta)
      \ind{\mathcal L_n}
    }
    \le
    (1+2\epsilon)
    \sigma_n^2
    \uC_\theta(W)
    +
    o(\sigma_n^2)
  \]
  uniformly on \(K\).
  The union of the compact images of the \(F_j\) is bounded.
  Thus \cref{appendix-eq:unified-compact-localization} makes the complementary contribution \(o(\sigma_n^2)\).
  Consequently,
  \[
    \limsup_{n\to\infty}
    \sup_{\theta\in K}
    \frac{
      \sigma_n^{-2}
      \E_\theta
      (\widehat\theta_{\epsilon,n}-\theta)^\top
      W(\widehat\theta_{\epsilon,n}-\theta)
    }{
      \uC_\theta(W)
    }
    \le1+2\epsilon.
  \]

  \noindent\textit{Diagonalization.}
  Finally, take \(\epsilon_k\downarrow0\), choose increasing deterministic thresholds \(N_k\) beyond which the \(\epsilon_k\)-construction has normalized risk at most \(1+3\epsilon_k\), and use that construction for \(N_k\le n<N_{k+1}\).
  Use any fixed point of \(K\) for the finitely many earlier indices.
  This diagonal sequence remains \(\rho_n\)-zCDP and proves the lemma.
\end{proof}
   \clearpage
  \section{zCDP Contraction and Lower Bounds in Privacy Regime}
\label{appendix-sec:lower-bound-proofs}

This section provides the technical arguments for the information region contraction and proves the lower bounds under the privacy regime.
We first establish the barycentric likelihood ratio bound in \cref{lem:barycentric};
then use it to prove the information contraction \cref{thm:diameter-contraction};
finally establish the lower bound \Cref{thm:matrix-local-lower} using the van Trees inequality.

\subsection{DQM under a Markov kernel}
\label{appendix-subsec:dqm}

\begin{lemma}[DQM under a Markov kernel]
  \label{appendix-lem:markov-dqm}
  Under \cref{ass:regular}, fix \(\theta\in\Theta\), and let \(M\) be any parameter-independent Markov kernel from \(\caX^n\) to a standard Borel output space \(\caZ\).
  Then \(Q_\theta^M=M\circ P_\theta^{\otimes n}\) is differentiable in quadratic mean at \(\theta\), with score
  \begin{equation}
    \label{appendix-eq:output-cond-score}
    s_M(Z;\theta)
    =
    \E_\theta \left[
      \sum_{i=1}^n s_\theta(X_i)
      \,\middle|\,
      Z
    \right].
  \end{equation}
  Consequently,
  \[
    I_M(\theta)
    =
    \E_\theta \left[
      s_M(Z;\theta)s_M(Z;\theta)^\top
    \right].
  \]
\end{lemma}

The product-score statement is the standard product rule for DQM models; see \citet[Chapter~7]{vaart1998_AsymptoticStatistics}.
Preservation of DQM under a parameter-independent Markov kernel and the conditional-score identity follow from the information-loss theorem of \citet[Section~7, Proposition~4]{lecam1988_PreservationLocal}, applied to the product experiment after adjoining the kernel randomization.

\subsection{Barycentric estimate for a small zCDP budget}
\label{appendix-subsec:barycentric}

We begin with the privacy loss moment bound that controls the passage from a pairwise reference law to a mixture reference law.

\begin{lemma}[Privacy loss moments for a small zCDP budget]
  \label{appendix-lem:privacy-loss-moments}
Suppose that \(P,Q\) satisfy
  \[
D_\alpha(P\|Q)\le\alpha\rho, \qquad D_\alpha(Q\|P)\le\alpha\rho, \qquad \alpha>1.
  \]
Let \(L=\log(dP/dQ)\).
For each fixed integer \(m\ge1\), as \(\rho\downarrow0\),
  \[
\E_Q \abs L^m=O_m(\rho^{m/2}), \qquad \E_Q \bigl(e^{\abs L}-1\bigr)^m=O_m(\rho^{m/2}).
  \]
The bounds are uniform over all pairs satisfying the two-sided R\'enyi inequalities.
\end{lemma}

\begin{proof}
The privacy-loss moment-generating-function identity and its subgaussian tail consequence are standard; see \citet[equation~(2) and the subsequent tail bound]{bun2016_ConcentratedDifferential}.
Applying them to both \((P,Q)\) and \((Q,P)\), there are absolute constants \(C_0,c_0,c_1>0\) such that
\[
Q(\abs L>t) \le C_0 \exp(-c_1 t^2/\rho), \qquad t\ge c_0 \rho.
\]
The layer-cake formula, split at \(c_0 \rho\), now yields
\[
  \begin{aligned}
\E_Q \abs L^m & =
  m\int_0^\infty t^{m-1} Q(\abs L>t) dt \\
& \le (c_0 \rho)^m + mC_0 \int_{c_0 \rho}^\infty t^{m-1} e^{-c_1 t^2/\rho} dt =O_m(\rho^{m/2}).
  \end{aligned}
\]
Finally, \(e^{\abs L}-1\le\abs L e^{\abs L}\).
Cauchy--Schwarz and the same two-sided moment-generating-function bounds imply
\[
\E_Q \bigl(e^{\abs L}-1\bigr)^m \le (\E_Q \abs L^{2m})^{1/2} (\E_Q e^{2m\abs L})^{1/2} =O_m(\rho^{m/2}),
\]
because \(m\) is fixed and \(\E_Q e^{2m\abs L}=O_m(1)\) uniformly as \(\rho\downarrow0\).
\end{proof}

\begin{proof}[Proof of \cref{lem:barycentric}]
The finite two-sided R\'enyi bounds make the laws \(R_t\) mutually absolutely continuous.
They are also dominated by \(\bar{R}\): if \(\bar{R}(A)=0\), then \(R_u(A)=0\) for \(\lambda\)-almost every \(u\); choosing one such \(u\) and using \(R_t \ll R_u\) gives \(R_t(A)=0\).

Fix \(s,t\).
Relative to \(R_s\), write
\[
Y=\frac{dR_t}{dR_s}=e^L, \qquad Y_u=\frac{dR_u}{dR_s}=e^{L_u}, \qquad Z=\frac{d\bar{R}}{dR_s}=\int Y_u \lambda(du).
\]
Then
\[
\norm{r_t-r_s}_{L^2(\bar{R})}^2 = \E_{R_s} \frac{(Y-1)^2}{Z}.
\]
Separating the pairwise term from the change of reference measure gives
\[
\E_{R_s} \frac{(Y-1)^2}{Z} \le \E_{R_s}(Y-1)^2 + \E_{R_s} \left[ (Y-1)^2\left(\frac{1}{Z}-1\right)_+ \right].
\]
The first term is the \(\chi^2\)-divergence of \(R_t\) from \(R_s\), and therefore
\[
\E_{R_s}(Y-1)^2 = e^{D_2(R_t\|R_s)}-1 \le e^{2\rho}-1.
\]

For the second term, Jensen's inequality gives
\[
Z = \int e^{L_u} \lambda(du) \ge \exp\left(\int L_u \lambda(du)\right) =e^{\bar{L}}, \qquad \bar{L}=\int L_u \lambda(du).
\]
It follows that
\[
\left(\frac{1}{Z}-1\right)_+ \le e^{\abs{\bar{L}}}-1.
\]
The function \(x\mapsto(e^{\abs x}-1)^2\) is convex.
Hence \cref{appendix-lem:privacy-loss-moments}, Jensen's inequality, and Tonelli's theorem give
\[
  \begin{aligned}
\E_{R_s}(e^{\abs{\bar{L}}}-1)^2 & \le \int \E_{R_s}(e^{\abs{L_u}}-1)^2\lambda(du)
  =O(\rho), \\
\E_{R_s}(Y-1)^4 & =O(\rho^2).
  \end{aligned}
\]
Cauchy--Schwarz therefore bounds the denominator error by
\[
\E_{R_s} \left[(Y-1)^2(e^{\abs{\bar{L}}}-1)\right] =O(\rho^{3/2})=o(\rho).
\]
Combining the bounds yields
\[
\norm{r_t-r_s}_{L^2(\bar{R})}^2 \le e^{2\rho}-1+O(\rho^{3/2}) = (e^{2\rho}-1)(1+o_\rho(1)).
\]
All constants come from the uniform bounds in \cref{appendix-lem:privacy-loss-moments}; the remainder is therefore uniform in \(s,t\), the family, and the mixing measure \(\lambda\).
Since \(e^{2\rho}-1=2\rho\{1+o_\rho(1)\}\), this proves \cref{lem:barycentric}.
\end{proof}

\subsection{Operator representation of the conditional score}
\label{appendix-subsec:cond}

\begin{lemma}[Operator representation of the conditional score]
  \label{appendix-lem:cond-score-operator}
  Under \cref{ass:regular}, fix \(\theta\in\Theta\), and let \(Q=Q_\theta^M=M\circ P_\theta^{\otimes n}\).
  For each \(i\in[n]\) and \(x\in\caX\), let \(Q_{i,x}\) be the output law obtained by fixing the \(i\)th record at \(x\) and drawing the remaining records from \(P_\theta\).
  Assume that, for each \(i\in[n]\), \(Q_{i,x}\ll Q\) for \(P_\theta\)-almost every \(x\), and write \(L_{i,x}=dQ_{i,x}/dQ\) for such \(x\).
  Suppose that \(\Lambda_i(x)=L_{i,x}\) is a measurable \(L^2(Q)\)-valued statistic and
  \[
    \E_\theta \left[
      \norm{\Lambda_i(X)}_{L^2(Q)}^2
      \norm{s_\theta(X)}_2^2
    \right]
    <\infty.
  \]
  Then
  \[
    B_M=\sum_{i=1}^n A_{\Lambda_i,\theta},
  \]
  where \(B_M:\R^p \to L^2_0(Q)\) is defined by \(B_M a=a^\top s_M(\cdot;\theta)\).
  Consequently,
  \[
    B_M^*B_M=I_M(\theta),
    \qquad
    \norm{B_M}_{\HS}^2=\Tr I_M(\theta).
  \]
\end{lemma}

\begin{proof}
The integrability assumption and Cauchy--Schwarz make each \(A_{\Lambda_i,\theta}\) a well-defined \(L^2(Q)\)-valued operator.
Moreover,
\[
\int A_{\Lambda_i,\theta} a dQ = \E_\theta \left[ a^\top s_\theta(X)\int L_{i,X} dQ \right] = \E_\theta (a^\top s_\theta(X)) =0,
\]
so its range is contained in \(L^2_0(Q)\).

Let \(g\in L^2(Q)\) and \(a\in\R^p\).
The joint law of \((X_i,Z)\) is \(P_\theta(dx)Q_{i,x}(dz)\), and therefore
\[
  \begin{aligned}
\langle g,A_{\Lambda_i,\theta} a\rangle_{L^2(Q)} & = \E_\theta \left[ a^\top s_\theta(X) \int g(z)L_{i,X}(z) Q(dz)
  \right] \\
& = \E_\theta \left[ a^\top s_\theta(X) \int g(z) Q_{i,X}(dz)
  \right] \\
& = \E_\theta (g(Z)a^\top s_\theta(X_i)).
  \end{aligned}
\]
Summing this identity over \(i\), and using \cref{appendix-lem:markov-dqm,appendix-eq:output-cond-score}, gives
\[
  \begin{aligned}
\left\langle g,\sum_{i=1}^n A_{\Lambda_i,\theta} a\right\rangle_{L^2(Q)} & = \E_\theta \left[ g(Z)a^\top \sum_{i=1}^n s_\theta(X_i)
  \right] \\
& = \E_Q (g(Z)a^\top s_M(Z;\theta)) = \langle g,B_M a \rangle_{L^2(Q)}.
  \end{aligned}
\]
Since \(g\) is arbitrary, this proves \(B_M=\sum_i A_{\Lambda_i,\theta}\).
Finally, for \(a,b\in\R^p\),
\[
\langle B_M a,B_M b \rangle_{L^2(Q)} = \E_Q ((a^\top s_M)(b^\top s_M)) = a^\top I_M(\theta)b.
\]
Thus \(B_M^*B_M=I_M(\theta)\), and taking traces gives the Hilbert--Schmidt identity.
\end{proof}

\subsection{Proof of \Cref{thm:diameter-contraction}}
\label{appendix-subsec:contraction}

Let \(Q=M\circ P_\theta^{\otimes n}\).
For each \(i\) and \(x\), define \(Q_{i,x}\) by fixing the \(i\)th record at \(x\) and averaging the remaining records under \(P_\theta\); thus
\[
Q = \int Q_{i,x} P_\theta(dx).
\]

\textit{Record-level R\'enyi bound.}
If \(x,x'\in\caX\), then \(Q_{i,x}\) and \(Q_{i,x'}\) are mixtures of output laws on neighboring deterministic databases, with the same mixing measure for the other \(n-1\) records.
For
\[
H_\alpha(P,Q) = \exp((\alpha-1)D_\alpha(P\|Q)),
\]
let \(Q_{x,w}\) and \(Q_{x',w}\) be the output laws when the remaining records are fixed at \(w\).
The function \(H_\alpha\) is the \(f\)-divergence associated with \(u\mapsto u^\alpha\), hence is jointly convex.
If \(\nu_i=P_\theta^{\otimes(n-1)}\), then
\[
  \begin{aligned}
H_\alpha(Q_{i,x},Q_{i,x'}) &= H_\alpha \left( \int Q_{x,w} \nu_i(dw), \int Q_{x',w} \nu_i(dw)
  \right)\\
&\le \int H_\alpha(Q_{x,w},Q_{x',w})\nu_i(dw) \le e^{\alpha(\alpha-1)\rho}.
  \end{aligned}
\]
The same argument with \(x,x'\) interchanged gives
\[
D_\alpha(Q_{i,x}\|Q_{i,x'})\le\alpha\rho, \qquad D_\alpha(Q_{i,x'}\|Q_{i,x})\le\alpha\rho .
\]
In particular, these conditional output laws are mutually absolutely continuous.
The same null set argument as in \cref{lem:barycentric} shows that each \(Q_{i,x}\) is absolutely continuous with respect to the barycenter \(Q\), so \(L_{i,x}=dQ_{i,x}/dQ\) is well defined.

\textit{Square integrability and record diameter.}
The pairwise bounds also control the second moments relative to the barycenter \(Q=\int Q_{i,u} P_\theta(du)\).
Choose a common dominating measure \(\nu\) and write \(q_x,q_u,q\) for the densities of \(Q_{i,x},Q_{i,u},Q\), respectively.
Pointwise convexity of \(v\mapsto v^{-1}\) gives
\[
  \begin{aligned}
\int L_{i,x}^2 dQ &=
  \int\frac{q_x^2}{q} d\nu\\
&\le \int  \int\frac{q_x^2}{q_u} d\nu P_\theta(du) = \int H_2(Q_{i,x},Q_{i,u})P_\theta(du) \le e^{2\rho}.
  \end{aligned}
\]
Thus \(x\mapsto L_{i,x}\) is an \(L^2(Q)\)-valued statistic, and the integrability condition in \cref{appendix-lem:cond-score-operator} follows from
\[
\sup_x \norm{L_{i,x}}_{L^2(Q)}^2 \le e^{2\rho}, \qquad \E_\theta \norm{s_\theta(X)}_2^2<\infty.
\]
Applying \cref{lem:barycentric} to the family \(\{Q_{i,x}:x\in\caX\}\), whose barycenter is \(Q\), gives
\[
\norm{L_{i,x}-L_{i,x'}}_{L^2(Q)}^2 \le (e^{2\rho}-1)(1+o_\rho(1))
\]
uniformly in \(i,x,x'\).
Choose the deterministic remainder in the barycentric lemma uniformly over all families and mixing measures, and set
\[
\delta_\rho^2 = \frac{(e^{2\rho}-1)(1+o_\rho(1))}{2\rho}.
\]
Then \(\delta_\rho=1+o_\rho(1)\), depends only on \(\rho\), and
\[
\diam\left( x\mapsto\frac{L_{i,x}}{\sqrt{2\rho}\delta_\rho} \right) \le1
\]
simultaneously for every \(i\), mechanism, and parameter value under consideration.

\textit{Operator assembly.}
Set
\[
\Phi_i(x)=\frac{L_{i,x}}{\sqrt{2\rho}\delta_\rho}\in L^2(Q), \qquad \bar{\Phi}(x)=\frac{1}{n}\sum_{i=1}^n \Phi_i(x).
\]
The average \(\bar{\Phi}\) also has diameter at most one.
Indeed, for all \(x,x'\),
\[
\norm{\bar{\Phi}(x)-\bar{\Phi}(x')}_{L^2(Q)} \le \frac{1}{n}\sum_{i=1}^n \norm{\Phi_i(x)-\Phi_i(x')}_{L^2(Q)} \le1.
\]
By \cref{appendix-lem:cond-score-operator},
\[
B_M = \sum_{i=1}^n A_{\Lambda_i,\theta} = n\sqrt{2\rho}\delta_\rho A_{\bar{\Phi},\theta}.
\]
Thus
\[
I_M(\theta) = B_M^*B_M = 2\rho n^2 \delta_\rho^2 A_{\bar{\Phi},\theta}^*A_{\bar{\Phi},\theta}.
\]
By \cref{appendix-lem:hilbert-compression}, there is an \(\R^p\)-valued statistic \(\widetilde\Phi\) with diameter at most one and
\[
  J_{\widetilde\Phi,\theta}
  =
  A_{\bar{\Phi},\theta}^*A_{\bar{\Phi},\theta}.
\]
Taking \(J_{M,\theta,\rho}=J_{\widetilde\Phi,\theta}\in\caJ_\theta\) gives \cref{eq:fisher-matrix-contraction}, and the trace bound follows from
\[
\Tr I_M(\theta) = 2\rho n^2 \delta_\rho^2 \Tr J_{M,\theta,\rho} \le 2\rho n^2 \Gamma_\theta(1+o_\rho(1)).
\]
All remainders are inherited from the uniform barycentric estimate, which proves the asserted uniformity.

\subsection{Matrix van Trees inequality}
\label{appendix-subsec:van-trees}

We use the following exact matrix form of the multivariate van Trees inequality.

\begin{lemma}[Matrix van Trees inequality]
  \label{appendix-lem:matrix-van-trees}
Let \(\pi\) be a continuously differentiable prior density with compact support contained in an open set \(\caV\subset\R^p\).
Suppose its boundary term vanishes under coordinatewise integration by parts and its Fisher matrix
  \[
J_\pi = \int \frac{\nabla\pi(\theta)\nabla\pi(\theta)^\top}{\pi(\theta)} d\theta
  \]
is finite.
Let \(Q_\theta\), \(\theta\in\caV\), be a dominated experiment that is DQM on the support of \(\pi\), with a jointly measurable score \(s_Z(z;\theta)\), and suppose
  \[
\int\Tr I_Z(\theta)\pi(\theta) d\theta<\infty.
  \]
Assume the density derivatives are integrable so that differentiation and coordinatewise integration by parts may be interchanged.
Define
  \[
H = J_\pi+ \int I_Z(\theta)\pi(\theta) d\theta,
  \]
and suppose \(H\succ0\).
Then every estimator \(T(Z)\in\R^p\) satisfying
\(\E_\pi \E_\theta \norm{T(Z)-\theta}_2^2<\infty\) obeys
  \begin{equation}
    \label{appendix-eq:matrix-van-trees}
\E_\pi \E_\theta (T(Z)-\theta)(T(Z)-\theta)^\top \succeq H^{-1}.
  \end{equation}
Moreover, for every estimator \(T\) and every \(W\succeq0\), without requiring a finite
unweighted second moment,
  \[
\E_\pi \E_\theta \norm{T(Z)-\theta}_W^2 \ge \Tr(WH^{-1}).
  \]
Here the left-hand side is interpreted as an extended nonnegative expectation.
\end{lemma}

Both conclusions are direct consequences of the multivariate van Trees inequality of \citet[Theorem~1, pp.~64--65]{gill1995_ApplicationsVan}, applied to linear functionals of \(\theta\) and optimized over its constant matrix weight; the weighted form follows by spectral decomposition of \(W\).

\subsection{Proof of \Cref{prop:strong-privacy-lower}}
\label{appendix-subsec:strong-privacy-lower}

Choose a continuously differentiable density \(\pi\) whose compact support is contained in \(\Theta^\circ\), whose boundary term vanishes, and whose Fisher matrix \(J_\pi\) is finite and positive definite.
By \cref{appendix-lem:score-density-stability},
\[
  \overline{\Gamma}_\pi
  =
  \sup_{\theta\in\operatorname{supp}(\pi)}\Gamma_\theta
  <\infty.
\]
For any \(\rho_n\)-zCDP mechanism and estimator, \cref{appendix-lem:matrix-van-trees}, the matrix Cauchy--Schwarz inequality, and \cref{thm:diameter-contraction} give
\[
  \begin{aligned}
  \sup_{\theta\in \Theta}
  \E_\theta\norm{\widehat{\theta}-\theta}_W^2
  &\ge
  \Tr\left[
    W\left\{J_\pi+\int I_M(\theta)\pi(\theta)\,d\theta\right\}^{-1}
  \right] \\
  &\ge
  \frac{\{\Tr(W^{1/2})\}^2}{
    \Tr(J_\pi)
    +2n^2\rho_n\delta_{\rho_n}^2\overline{\Gamma}_\pi
  }.
  \end{aligned}
\]
Here the regularity conditions in \cref{appendix-lem:matrix-van-trees} follow exactly as in the proof of \cref{thm:matrix-local-lower} below.
Since \(W\ne0\), \(n^2\rho_n=O(1)\), and \(\delta_{\rho_n}\to1\), the last expression is bounded below by a positive constant independent of the mechanism and estimator.
Taking the infimum proves the result.

\subsection{Proof of \Cref{thm:matrix-local-lower}}
\label{appendix-subsec:matrix-local-lower}

We prove the bound for an arbitrary sequence of \(\rho_n\)-zCDP mechanisms and estimators.

\textit{Prior construction.}
Choose a smooth product density \(\varphi\) supported on \([-1,1]^p\), with vanishing boundary term and finite positive definite Fisher matrix
\[
J_\varphi = \int \frac{\nabla\varphi(v)\nabla\varphi(v)^\top}{\varphi(v)} dv.
\]
Set
\[
r_n=\frac{R_n}{\sqrt p}, \qquad \varphi_{r_n}(u)=r_n^{-p} \varphi(u/r_n),
\]
and define the prior by
\[
U\sim\varphi_{r_n}, \qquad \theta=\theta_0+\frac{U}{\tau_n}.
\]
Because \(\norm{U}_\infty \le r_n\) implies \(\norm{U}_2 \le\sqrt p r_n=R_n\), this prior, denoted by \(\pi_n\), is supported on \(\Theta_n\).
The support is eventually contained in \(\Theta\) because \(\Theta\) is open and \(R_n/\tau_n \to0\).

On the \(\theta\)-scale, the prior score is
\[
\nabla_\theta \log\pi_n(\theta) = \frac{\tau_n}{r_n} \nabla\log\varphi\left( \frac{\tau_n(\theta-\theta_0)}{r_n} \right).
\]
Changing variables gives
\begin{equation}
  \label{appendix-eq:lower-prior-info}
J_{\pi_n} = \tau_n^2 r_n^{-2} J_\varphi = p\tau_n^2 R_n^{-2} J_\varphi, \qquad \frac{J_{\pi_n}}{\tau_n^2}\to0.
\end{equation}
In particular, \(J_{\pi_n} \succ0\).

\textit{Average privatized information.}
Fix a sequence of mechanisms \(M_n\), put
\[
\delta_n=\delta_{\rho_n}=1+o(1), \qquad s_n=2n^2 \rho_n \delta_n^2,
\]
and use \cref{appendix-lem:markov-dqm} to define the output Fisher information for every \(\theta\) in the prior support.
By \cref{thm:diameter-contraction},
\[
J_{n,\theta} = \frac{I_{M_n}(\theta)}{s_n} \in \caJ_\theta.
\]
The jointly measurable density and score versions supplied by the dominated setup on standard Borel spaces make \(\theta\mapsto J_{n,\theta}\) measurable.
Define
\[
K_n = \int J_{n,\theta} \pi_n(d\theta).
\]

Let \(\eta_n=R_n/\tau_n=o(1)\).
By \cref{appendix-lem:score-density-stability},
\[
\Delta_n = \sup_{\theta\in\operatorname{supp}(\pi_n)} \sup_{J\in\caJ_\theta} d(J,\caJ_{\theta_0}) \to0.
\]
Let \(\Pi_0\) be Euclidean projection onto the nonempty closed convex set \(\caJ_{\theta_0}\), and write
\[
\widetilde{J}_{n,\theta}=\Pi_0(J_{n,\theta}), \qquad \widetilde{K}_n = \int\widetilde{J}_{n,\theta} \pi_n(d\theta).
\]
The projection is continuous, hence measurable.
Convexity and closedness of \(\caJ_{\theta_0}\) imply \(\widetilde{K}_n \in\caJ_{\theta_0}\), and
\[
\norm{K_n-\widetilde{K}_n}_{\HS} \le \Delta_n=o(1).
\]
Local boundedness of \(\Gamma_\theta\) gives a uniform trace, and hence norm, bound for \(K_n\).
It follows that every subsequential limit \(K\) of \(K_n\) belongs to \(\caJ_{\theta_0}\): along the same subsequence, \(\widetilde{K}_n-K_n \to0\), and \(\caJ_{\theta_0}\) is closed.

\textit{Matrix van Trees and the limit.}
Continuous differentiability in \(L^1\), as assumed in \cref{ass:score-density-l1}, is preserved by the product experiment and by the Markov kernel \(M_n\).
Indeed, the product density has derivative
\[
\dot{p}_\theta^{(n)}(x_{1:n}) = p_\theta^{\otimes n}(x_{1:n}) \sum_{i=1}^n s_\theta(x_i),
\]
and the product rule is valid in \(L^1(\mu^{\otimes n})\).
The kernel maps finite signed input measures to output measures with no larger total variation norm, so applying it to the \(L^1\) remainder preserves continuous differentiability and pushes this derivative forward to the output score density.
Together with \cref{appendix-lem:markov-dqm}, it verifies the score and integration-by-parts conditions in \cref{appendix-lem:matrix-van-trees} on the support of \(\pi_n\).
The integrated output information is finite because the contraction theorem and local boundedness of \(\Gamma_\theta\) give
\[
\int\Tr I_{M_n}(\theta)\pi_n(d\theta) \le s_n \sup_{\theta\in\operatorname{supp}(\pi_n)}\Gamma_\theta<\infty.
\]
Applying \cref{appendix-lem:matrix-van-trees} therefore yields, for every estimator,
\[
  \begin{aligned}
&\E_{\pi_n} \E_\theta
  \norm{\widehat{\theta}-\theta}_W^2 \\
&\quad\ge \Tr\left[ W\xk{ J_{\pi_n}+ \int I_{M_n}(\theta)\pi_n(d\theta) }^{-1}
  \right] \\
&\quad= \frac{1}{s_n} \Tr\left[ W\left(K_n+\frac{J_{\pi_n}}{s_n}\right)^{-1} \right].
  \end{aligned}
\]
The inverse exists because \(J_{\pi_n} \succ0\).
By \cref{appendix-eq:lower-prior-info},
\[
\frac{J_{\pi_n}}{s_n} = \frac{p}{2\delta_n^2}R_n^{-2} J_\varphi \to0, \qquad \frac{\tau_n^2}{s_n} = \frac{1}{2\delta_n^2} \to\frac{1}{2}.
\]

Take a subsequence along which the liminf of the scaled Bayes risks is attained, and then a further subsequence for which \(K_n \to K\).
The preceding projection argument gives \(K\in\caJ_{\theta_0}\).
Since \(\caI_W\) is lower semicontinuous and agrees with \(A\mapsto\Tr(WA^{-1})\) on positive definite matrices,
\[
  \begin{aligned}
\liminf_n \Tr\left[ W\left(K_n+\frac{J_{\pi_n}}{s_n}\right)^{-1} \right] &\ge
  \caI_W(K) \\
&\ge \inf_{J\in\caJ_{\theta_0}}\caI_W(J).
  \end{aligned}
\]
Consequently,
\[
\liminf_{n\to\infty} \tau_n^2 \E_{\pi_n} \E_\theta \norm{\widehat{\theta}-\theta}_W^2 \ge \frac{1}{2} \inf_{J\in\caJ_{\theta_0}}\caI_W(J) = \mfC_{\theta_0}(W).
\]
The Bayes risk under a prior supported on \(\Theta_n\) is no larger than the supremum risk over \(\Theta_n\).
Because the sequence of mechanisms and estimators was arbitrary, the local minimax lower bound follows.

\begin{corollary}[Trace-relaxed Euclidean lower bound]
  \label{appendix-cor:euclidean-trace-lower}
Under the assumptions of \cref{thm:matrix-local-lower}, if \(0<\Gamma_{\theta_0}<\infty\), then
  \[
\liminf_{n\to\infty}
\tau_n^2
\inf_{\widehat{\theta}:\rho_n \text{-zCDP}}
\sup_{\theta\in\Theta_n}
\E_\theta \norm{\widehat{\theta}-\theta}_2^2
\ge
\frac{p^2}{2\Gamma_{\theta_0}}.
  \]
\end{corollary}

\begin{proof}
For every \(J\in\caJ_{\theta_0}\), if \(J\) is singular, then
\(\caI_{I_p}(J)=\infty\).
Otherwise, the Cauchy--Schwarz inequality for the eigenvalues of \(J\) gives
  \[
\caI_{I_p}(J)
=
\Tr(J^{-1})
\ge
\frac{p^2}{\Tr J}
\ge
\frac{p^2}{\Gamma_{\theta_0}}.
  \]
Thus, by \cref{def:diameter-score-constant},
  \[
\mfC_{\theta_0}(I_p)
=
\frac{1}{2}\inf_{J\in\caJ_{\theta_0}}\caI_{I_p}(J)
\ge
\frac{p^2}{2\Gamma_{\theta_0}}.
  \]
Applying \cref{thm:matrix-local-lower} with \(W=I_p\) proves the result.
\end{proof}

\subsection{Extensions}
\label{appendix-subsec:outer}

\begin{remark}[Local outer information without stability]
  \label{appendix-rem:local-outer-bound}
If \cref{ass:score-density-l1} is not imposed, the same proof gives a lower bound with
  \[
\caJ_{\theta_0}^{\mathrm{loc}} = \bigcap_{\eta>0} \overline{\conv} \left( \bigcup_{\norm{\theta-\theta_0}_2 \le\eta}\caJ_\theta \right)
  \]
in place of \(\caJ_{\theta_0}\), provided \(\Gamma_\theta\) is locally bounded.
The corresponding constant is
  \[
\mfC_{\theta_0}^{\mathrm{loc}}(W) = \frac{1}{2} \inf_{J\in\caJ_{\theta_0}^{\mathrm{loc}}}\caI_W(J).
  \]
This version separates the contraction argument from continuity of the feasible information region.
\end{remark}
   \clearpage
  \section{Lower Bounds in the Transition Regime}
\label{appendix-sec:trans-regime}

This section proves the lower bound over arbitrary mechanisms in the transition regime, where sampling fluctuation and privacy noise have the same order.
Fix \(\theta_0 \in\Theta\), suppose \cref{ass:regular,ass:score-density-l1} hold, and assume \(I(\theta_0)\succ0\).
Let
\begin{equation}
  \label{appendix-eq:trans-regime}
  \rho_n \to0,
  \qquad
  \lambda_n \coloneqq n\rho_n \to\lambda\in(0,\infty),
  \qquad
  \tau_n=n\sqrt{\rho_n}.
\end{equation}
For \(R_n \to\infty\) with \(R_n/\tau_n \to0\), put
\[
  \Theta_n
  =
  \dk{
    \theta_0+\frac{u}{\tau_n}:
    \norm{u}_2 \le R_n
  }.
\]
Put \(\beta=\lambda/(1+\lambda)\).

\begin{theorem}[Local lower bound in the transition regime]
  \label{appendix-thm:trans-local-lower}
  For every \(W\succeq0\),
  \begin{equation}
    \label{appendix-eq:trans-local-lower}
    \liminf_{n\to\infty}
    n
    \inf_{\widehat{\theta}:\rho_n \text{-zCDP}}
    \sup_{\theta\in\Theta_n}
    \E_\theta
    (\widehat{\theta}-\theta)^\top
    W(\widehat{\theta}-\theta)
    \ge
    \frac{\uCt(W)}{\beta}.
  \end{equation}
\end{theorem}

\noindent\emph{Proof idea and roadmap.}
The proof has three stages.
\begin{itemize}[leftmargin=*]
  \item \emph{Finite alphabet.}
  Establish the sharp lower bound for the finite alphabet model by reducing an arbitrary mechanism to a channel from multinomial counts and proving the corresponding information contraction.
  \item \emph{Transfer to the general model.}
  Approximate the original score by conditional scores on a fixed finite partition and pull the finite alphabet information bound back to the original sample space.
  \item \emph{Van Trees closure.}
  Average the transferred information bound under a shrinking prior and apply the matrix van Trees inequality to obtain the local minimax constant.
\end{itemize}

\smallskip
\noindent\emph{Finite alphabet.}
The first stage, developed in \cref{appendix-subsec:trans-fin-alphabet}, retains the sampling and privacy contributions within a single information bound.
By \cref{appendix-lem:trans-symmetrization}, symmetrization reduces an arbitrary mechanism to a channel from the multinomial count vector without changing its iid output law or its zCDP guarantee.
For this count channel, \cref{appendix-lem:trans-count-residual-info,appendix-eq:trans-output-info-residual} express the output information as the full multinomial information minus a term determined by the residual count covariance.
To bound this covariance from below, we introduce the vector \(D\) of joint scores associated with count shifts.
The residual identity in \cref{appendix-lem:trans-count-shift-identity,appendix-eq:trans-cross-identity} supplies the asymptotically standardized cross moment between the count residual and \(D\), and a matrix projection inequality reduces the problem to controlling \(\E_q(DD^\top)\).
The exact decomposition in \cref{appendix-eq:trans-joint-score-gram-split} separates the contribution of the multinomial count law from that of the privacy channel.
The two terms are evaluated in \cref{appendix-eq:trans-prior-gram,appendix-eq:trans-privacy-gram}, producing respectively \(V_q^{-1}/n\) and \(2\rho_nJ_{n,q}\) to first order.
Combining these ingredients gives the contraction on a finite alphabet in \cref{appendix-thm:trans-categorical-contraction}; \cref{appendix-cor:trans-fin-alphabet-gauss} then rewrites its right-hand side as the Gaussian information generated by a bounded categorical statistic.

\smallskip
\noindent\emph{Transfer to the general model.}
The second stage is carried out in \cref{appendix-subsec:trans-partition-transfer}.
For each local parameter value, we freeze the conditional distribution within every cell of a fixed finite partition and vary only the cell probabilities.
At the frozen parameter, the resulting experiment on a finite alphabet has exactly the same output marginal as the original mechanism, while its score is the cellwise conditional expectation of the original score.
Using the \(L^1\) smoothness assumption together with the barycentric bound for likelihood ratios, \cref{appendix-lem:trans-partition-score-approx} chooses a partition, independent of the mechanism, for which the two normalized output information matrices differ by at most the error in \cref{appendix-eq:trans-partition-info-approx}.
The Gaussian representation on a finite alphabet can then be pulled back to the original sample space by making the statistic constant on each cell, placing the resulting information matrix in \(\beta\caU_{\theta_0,\beta}\).

\smallskip
\noindent\emph{Van Trees closure.}
Finally, \cref{appendix-subsec:trans-local-lower-proof} chooses a shrinking prior whose information is negligible by \cref{appendix-eq:trans-prior-info-vanishes}.
Averaging the transferred information bound gives \cref{appendix-eq:trans-averaged-info-bound}.
The matrix van Trees inequality in \cref{appendix-lem:matrix-van-trees}, followed by compactness of the unified information region and passage to the limits \(n\to\infty\) and \(\ep\downarrow0\), yields the constant in \cref{appendix-eq:trans-local-lower}.

\subsection{Sharp lower bound on a finite alphabet}
\label{appendix-subsec:trans-fin-alphabet}

We next prove an information bound, uniform over mechanisms, that retains both sources of fluctuation.
Consider the \(k\)-category model, use category \(k\) as baseline, and put \(d=k-1\).
For \(x\in\{1,\ldots,k\}\), define the indicator vector relative to this baseline,
\[
  Y(x)
  =
  \xk{\mathbf{1}\{x=1\},\ldots,\mathbf{1}\{x=d\}}^\top.
\]
Parameterize the categorical law by \(q_{1:d}=(q_1,\ldots,q_d)^\top\), with \(q_k=1-\sum_{a=1}^d q_a\).
If \(X\sim q\), then
\begin{equation}
  \label{appendix-eq:trans-categorical-indicator-moments}
  \E_q \xk{Y(X)}=q_{1:d},
  \qquad
  V_q
  :=
  \Var_q \xk{Y(X)}
  =
  \operatorname{diag}(q_1,\ldots,q_d)
  -q_{1:d} q_{1:d}^\top.
\end{equation}
For iid observations \(X_1,\ldots,X_n\), the first \(d\) category counts satisfy
\[
  M_{1:d}=\sum_{i=1}^n Y(X_i),
  \qquad
  \Var_q(M_{1:d})=nV_q.
\]
When \(q\) lies in the simplex interior, the single-observation score in these coordinates is
\[
  S_q(X)
  =
  V_q^{-1} \xk{Y(X)-q_{1:d}},
\]
and hence its Fisher information is
\[
  \Var_q \xk{S_q(X)}
  =
  V_q^{-1} V_q V_q^{-1}
  =
  V_q^{-1}.
\]
Because the multinomial count vector is sufficient for \(q\), the complete count experiment has information \(nV_q^{-1}\), or \(V_q^{-1}\) after normalization by \(n\).
We also have
\begin{equation}
  \label{appendix-eq:trans-categorical-Vq-inverse}
  V_q^{-1}
  =
  \operatorname{diag}\xk{q_1^{-1},\ldots,q_d^{-1}}
  +q_k^{-1}\mathbf{1}\mathbf{1}^\top,
\end{equation}

\begin{lemma}[Symmetrization on a finite alphabet]
  \label{appendix-lem:trans-symmetrization}
  Let \(M_n:\caX^n \to\caZ_n\) be a \(\rho_n\)-zCDP mechanism with \(\caX=\{1,\ldots,k\}\).
  Define its symmetrization by
  \[
    \overline{M}_n(dz\mid x_{1:n})
    =
    \frac{1}{n!}
    \sum_{\sigma\in\mathfrak{S}_n}
    M_n(dz\mid x_{\sigma(1)},\ldots,x_{\sigma(n)}).
  \]
  Then \(\overline{M}_n\) is also \(\rho_n\)-zCDP.
  Under every iid distribution \(q^{\otimes n}\), it has the same output law as \(M_n\).
  Moreover, \(\overline{M}_n\) factors through the count vector \(m(x)=(m_1(x),\ldots,m_k(x))\): there is a channel \(m\mapsto Q_m\) such that \(\overline{M}_n(\,\cdot\mid x)=Q_{m(x)}\).

  \begin{proof}
    If \(x,x'\) are neighboring data sets, then \((x_{\sigma(1)},\ldots,x_{\sigma(n)})\) and \((x'_{\sigma(1)},\ldots,x'_{\sigma(n)})\) are neighboring for every \(\sigma\in\mathfrak{S}_n\).
    Hence, for every \(\alpha>1\), the corresponding output laws satisfy the zCDP bounds in both directions.
    Joint convexity of \(\exp\xk{(\alpha-1)D_\alpha(\,\cdot\,\|\,\cdot\,)}\) then gives the same bound for the two equally weighted mixtures defining \(\overline{M}_n\), so \(\overline{M}_n\) is \(\rho_n\)-zCDP.

    If \(X\sim q^{\otimes n}\), then \(X_\sigma\) has the same law as \(X\) for every \(\sigma\), and therefore \(M_n(X_\sigma)\) and \(M_n(X)\) have the same mixture law.
    Finally, two sequences with the same count vector are permutations of one another, so the averaged kernel \(\overline{M}_n(\,\cdot\mid x)\) depends on \(x\) only through \(m(x)\).
  \end{proof}
\end{lemma}

For the remaining arguments on a finite alphabet, we use the following common notation for the count channel.
Fix \(q\) in the simplex interior.
Given a channel \(m\mapsto Q_m\), let \(\pi_q\) denote the \(\operatorname{Multinomial}(n,q)\) count law, and let
\[
  M\sim\pi_q,
  \qquad
  Z\mid M=m\sim Q_m.
\]
Put
\begin{equation}
  \label{appendix-eq:trans-count-residual-notation}
  T(Z)=\E_q(M_{1:d} \mid Z),
  \qquad
  R=T-M_{1:d}.
\end{equation}

\begin{lemma}[Output information from count residuals]
  \label{appendix-lem:trans-count-residual-info}
  Let \(M_n\) be a mechanism on the \(k\)-category model whose kernel factors through the count vector as \(m\mapsto Q_m\), and adopt the notation for the count channel in \cref{appendix-eq:trans-count-residual-notation} at a point \(q\) in the simplex interior.
  Then
  \begin{equation}
    \label{appendix-eq:trans-output-info-residual}
    \frac{1}{n}I_{M_n}(q)
    =
    V_q^{-1}
    -V_q^{-1}
    \frac{\E_q(RR^\top)}{n}
    V_q^{-1}.
  \end{equation}

  \begin{proof}
    The score for the full multinomial experiment is
    \[
      S_M
      =
      V_q^{-1}(M_{1:d}-nq_{1:d}),
      \qquad
      \Var_q(S_M)=nV_q^{-1}.
    \]
    By \cref{appendix-lem:markov-dqm,appendix-eq:output-cond-score}, the output score is
    \(\E_q(S_M \mid Z)\).
    Total covariance therefore gives
    \[
      I_{M_n}(q)
      =
      \Var_q \zk{\E_q(S_M \mid Z)}
      =
      \Var_q(S_M) -\E_q \zk{\Var_q(S_M \mid Z)}.
    \]
    Since \(R=T-M_{1:d}\) has conditional mean zero given \(Z\),
    \[
      \E_q(RR^\top)
      =
      \E_q[\Var_q(M_{1:d} \mid Z)].
    \]
    Moreover,
    \[
      S_M-\E_q(S_M \mid Z)
      =
      V_q^{-1}(M_{1:d}-T(Z))
      =
      -V_q^{-1} R,
    \]
    so
    \[
      \E_q \zk{\Var_q(S_M \mid Z)}
      =
      V_q^{-1} \E_q(RR^\top)V_q^{-1}.
    \]
    Hence
    \[
      I_{M_n}(q)
      =
      nV_q^{-1}
      -V_q^{-1} \E_q(RR^\top)V_q^{-1}.
    \]
    Dividing by \(n\) proves the claim.
  \end{proof}
\end{lemma}

For the arguments involving count shifts below, write \(\delta_a=e_a-e_k\), \(a=1,\ldots,d\).
Whenever the neighboring output laws are mutually absolutely continuous, define, for \(m_k \ge1\),
\begin{equation}
  \label{appendix-eq:trans-count-shift-likelihood-ratios}
  r_a(m)
  =
  \frac{\pi_q(m+\delta_a)}{\pi_q(m)}
  =
  \frac{q_a m_k}{q_k(m_a+1)},
  \qquad
  L_a(m,z)
  =
  \frac{dQ_{m+\delta_a}}{dQ_m}(z).
\end{equation}
Set \(D_a=r_a L_a-1\).
When \(m_k=0\), set \(r_a=L_a=1\), and hence \(D_a=0\), as an algebraic boundary convention.
Write \(D=(D_1,\ldots,D_d)^\top\).

\begin{lemma}[Residual identity for count shifts]
  \label{appendix-lem:trans-count-shift-identity}
  Under the notation for the count channel in \cref{appendix-eq:trans-count-residual-notation,appendix-eq:trans-count-shift-likelihood-ratios}, suppose the likelihood ratios \(L_a\) for the shifts exist.
  Then, for \(i,a=1,\ldots,d\),
  \begin{equation}
    \label{appendix-eq:trans-count-shift-exact}
    \E_q(R_i D_a)
    =
    \ind{i=a}\xk{1-\Pr_q \dk{M_a=0}}
    +\E_q[R_i \ind{M_k=0}]
    -\E_q[R_i \ind{M_a=0}].
  \end{equation}
  Consequently, uniformly for \(q\) in a fixed compact subset of the simplex interior,
  \begin{equation}
    \label{appendix-eq:trans-cross-identity}
    C_{n,q} \coloneqq\E_q(RD^\top)=I_d+o(1).
  \end{equation}

  \begin{proof}
    Since \(\E_q(R\mid Z)=0\), we have \(\E_q R=0\).
    By the boundary convention and a change of variables \(m'=m+\delta_a\) on the count lattice,
    \[
      \begin{aligned}
        \E_q(R_i D_a)
        ={}&
        \sum_{m:m_k \ge1}
        \pi_q(m)
        \int \xk{T_i(z)-m_i}\xk{r_a(m)L_a(m,z)-1}Q_m(dz)
        \\
        ={}&
        \sum_{m':m'_a\ge1}
        \pi_q(m')
        \int \xk{T_i(z)-m'_i+\ind{i=a}}Q_{m'}(dz)
        -\E_q[R_i \ind{M_k \ge1}]
        \\
        ={}&
        \E_q[R_i \ind{M_a \ge1}]
        +\ind{i=a}\Pr_q(M_a \ge1)
        +\E_q[R_i \ind{M_k=0}].
      \end{aligned}
    \]
    Using \(\E_q R_i=0\) in the last identity proves \cref{appendix-eq:trans-count-shift-exact}.
    Let \(K\) be a fixed compact subset of the simplex interior and put
    \[
      \eta
      =
      \inf_{q\in K} \min_{1\le j\le k} q_j
      >0.
    \]
    Since \(M_j \sim\operatorname{Binomial}(n,q_j)\),
    \[
      \sup_{q\in K} \Pr_q(M_j=0)
      =
      \sup_{q\in K}(1-q_j)^n
      \le
      e^{-\eta n}.
    \]
    Together with \(\abs{R_i}\le n\), the exact identity therefore gives
    \[
      \sup_{q\in K}
      \abs{\E_q(R_i D_a)-\ind{i=a}}
      \le
      (2n+1)e^{-\eta n}
      =
      o(1).
    \]
    Hence every entry of \(\E_q(RD^\top)-I_d\) is \(o(1)\) uniformly, proving \cref{appendix-eq:trans-cross-identity}.
  \end{proof}
\end{lemma}

For the contraction bound, let \(\mathcal{G}_k\) be the configuration Gram region consisting of matrices
\begin{equation}
  \label{appendix-eq:trans-config-region}
  J_{ab}
  =
  \ang{\phi_a-\phi_k,\phi_b-\phi_k}_{\caH},
  \qquad
  \max_{a,b} \norm{\phi_a-\phi_b}_{\caH} \le1,
\end{equation}
over all real Hilbert spaces and configurations \((\phi_1,\ldots,\phi_k)\).
After translating \(\phi_k\) to zero, the same region has the finite-dimensional characterization
\[
  \mathcal{G}_k
  =
  \dk{
    J\succeq0:
    J_{aa} \le1,\quad
    J_{aa}+J_{bb}-2J_{ab} \le1
    \text{ for all }a,b
  }.
\]
Thus \(\mathcal{G}_k\) is convex and compact, and every one of its elements has a realization in \(\R^d\).

\begin{theorem}[Transition contraction on a finite alphabet]
  \label{appendix-thm:trans-categorical-contraction}
  Let \(M_n\) be any \(\rho_n\)-zCDP mechanism applied to \(n\) observations from the \(k\)-category model.
  Uniformly for \(q\) in a fixed compact subset of the simplex interior, there exists \(J_{n,q} \in\mathcal{G}_k\) such that
  \begin{equation}
    \label{appendix-eq:trans-categorical-info-bound}
    \frac{1}{n} I_{M_n}(q)
    \preceq
    V_q^{-1}
    -V_q^{-1}
    (V_q^{-1}+2\lambda_n J_{n,q})^{-1}
    V_q^{-1}
    +o(1)I_d.
  \end{equation}
  The remainder is uniform over all such mechanisms.
\end{theorem}

\begin{proof}
  By \cref{appendix-lem:trans-symmetrization}, we may replace \(M_n\) by its symmetrization.
  This leaves its output law, and hence \(I_{M_n}(q)\), unchanged, preserves \(\rho_n\)-zCDP, and gives a channel \(m\mapsto Q_m\) from the multinomial count vector \(M=(M_1,\ldots,M_k)\).
  For every admissible count shift, zCDP bounds the R\'enyi divergences between \(Q_m\) and \(Q_{m+\delta_a}\) in both directions.
  These laws are therefore mutually absolutely continuous, so the shift likelihood ratios above exist.
  We adopt the notation for the count channel in \cref{appendix-eq:trans-count-residual-notation,appendix-eq:trans-count-shift-likelihood-ratios}.

  We first reduce the desired information bound to a residual covariance bound.
  By \cref{appendix-lem:trans-count-residual-info}, it is enough to show that, for some \(J_{n,q} \in\mathcal{G}_k\),
  \begin{equation}
    \label{appendix-eq:trans-residual-covariance-target}
    \frac{1}{n}\E_q(RR^\top)
    \succeq
    (V_q^{-1}+2\lambda_n J_{n,q})^{-1}
    -o(1)I_d.
  \end{equation}

  The residual identity for count shifts supplies the cross moment needed to project the residual onto \(D\).
  By \cref{appendix-lem:trans-count-shift-identity}, \(C_{n,q}=I_d+o(1)\) uniformly on interior compact subsets.
  In particular, \(C_{n,q}\) is nonsingular for all sufficiently large \(n\).
  Put \(G_{n,q}=\E_q(DD^\top)\).
  If \(G_{n,q}\) were singular, there would be a nonzero \(v\) such that \(v^\top D=0\) almost surely, which would imply \(C_{n,q} v=\E_q \xk{R(D^\top v)}=0\), a contradiction.
  Thus \(G_{n,q}\) is also nonsingular for all sufficiently large \(n\).

  Taking \(B=C_{n,q} G_{n,q}^{-1}\), we obtain
  \[
    0
    \preceq
    \E_q \zk{(R-BD)(R-BD)^\top}
    =
    \E_q(RR^\top)
    -
    C_{n,q} G_{n,q}^{-1} C_{n,q}^\top.
  \]
  Therefore,
  \[
    \E_q(RR^\top)
    \succeq
    C_{n,q}
    G_{n,q}^{-1}
    C_{n,q}^\top
    =
    C_{n,q} \zk{\E_q(DD^\top)}^{-1} C_{n,q}^\top.
  \]
  It therefore remains to show that, for some \(J_{n,q} \in\mathcal{G}_k\),
  \begin{equation}
    \label{appendix-eq:trans-total-joint-score-gram}
    \E_q(DD^\top)
    =
    \frac{1}{n}(
    V_q^{-1}+2\lambda_n J_{n,q}
    )
    +o(n^{-1}).
  \end{equation}
  Indeed, \(q\) ranges over a fixed interior compact subset, \(\lambda_n\) is bounded, and \(\mathcal{G}_k\) is compact, so the matrices \(V_q^{-1}+2\lambda_n J_{n,q}\) are uniformly positive definite.
  Hence \cref{appendix-eq:trans-total-joint-score-gram} implies
  \[
    \frac{1}{n}G_{n,q}^{-1}
    =
    (V_q^{-1}+2\lambda_n J_{n,q})^{-1}
    +o(1)I_d,
  \]
  and the preceding projection inequality together with \(C_{n,q}=I_d+o(1)\) yields \cref{appendix-eq:trans-residual-covariance-target}.

  We now establish \cref{appendix-eq:trans-total-joint-score-gram} by separating the contributions of the count law and the privacy channel to the joint shift score.
  For \(a=1,\ldots,d\), set
  \begin{equation}
    \label{appendix-eq:trans-count-shift-score-notation}
    U_a=L_a-1,
    \qquad
    D_a
    =
    r_a(1+U_a)-1
    =
    (r_a-1)+r_a U_a.
  \end{equation}
  Here \(r_a-1\) is the shift score of the multinomial count law, whereas \(U_a\) is the shift score of the output channel.
  Since \(L_a=dQ_{m+\delta_a}/dQ_m\), we have \(\E_{Q_m} U_a=0\).
  Conditioning on \(M\) therefore makes both mixed terms vanish, giving the exact decomposition
  \begin{equation}
    \label{appendix-eq:trans-joint-score-gram-split}
    \E_q(DD^\top)
    =
    J^{\mathrm{prior}}_{n,q}
    +H^{\mathrm{priv}}_{n,q}.
  \end{equation}
  Writing \(r=(r_1,\ldots,r_d)^\top\), the two terms are
  \[
    J^{\mathrm{prior}}_{n,q}
    =
    \E_q[
    (r(M)-\mathbf{1})(r(M)-\mathbf{1})^\top
    \ind{M_k \ge1}
    ],
  \]
  and
  \[
    [H^{\mathrm{priv}}_{n,q}]_{ab}
    =
    \E_q[
    r_a(M)r_b(M)
    \E_{Q_M}[U_a(M,Z)U_b(M,Z)]
    \ind{M_k \ge1}
    ].
  \]

  \smallskip
  \noindent\emph{Prior term.}
  Let \(K\) be the fixed compact subset of the simplex interior and put
  \[
    \eta
    =
    \inf_{q\in K}\min_{1\le a\le k}q_a
    >0.
  \]
  Write
  \[
    h=m_{1:d}-nq_{1:d},
    \qquad
    h_k=m_k-nq_k=-\sum_{b=1}^d h_b,
  \]
  and define
  \[
    \Delta(m)=\norm{m-nq}_2,
    \qquad
    A_{n,q}
    =
    \dk{m:\Delta(m)\le n^{2/3}}.
  \]
  For all sufficiently large \(n\), every \(m\in A_{n,q}\) satisfies
  \[
    m_k
    \ge
    nq_k-n^{2/3}
    \ge
    n\eta-n^{2/3}
    \ge1.
  \]
  Therefore,
  \begin{equation}
    \label{appendix-eq:trans-prior-central-tail-split}
    \begin{aligned}
      J^{\mathrm{prior}}_{n,q}
      ={}&
      \E_q\zk{
        (r(M)-\mathbf{1})(r(M)-\mathbf{1})^\top
        \ind{M\in A_{n,q}}
      }
      \\
      &+
      \E_q\zk{
        (r(M)-\mathbf{1})(r(M)-\mathbf{1})^\top
        \ind{M\notin A_{n,q}}
        \ind{M_k\ge1}
      }\\
      \eqqcolon{} & I_1 + I_2
    \end{aligned}
  \end{equation}

  Introduce the linear term,
  \[
    \ell_{n,q}=(\ell_{n,q,1},\ldots,\ell_{n,q,d})^\top,\quad
    \ell_{n,q,a}(m)
    =
    \frac{h_k}{nq_k}-\frac{h_a}{nq_a},\quad a=1,\ldots,d.
  \]
  Direct algebra gives the exact decomposition
  \[
    r_a(m)-1
    =
    \ell_{n,q,a}(m)
    +b_{n,q,a}(m),\quad
    b_{n,q,a}(m) \coloneqq
    -\frac{
      1+(h_a+1)\ell_{n,q,a}(m)
    }{m_a+1}.
  \]

  Moreover, on \(A_{n,q}\),
  \[
    \max\xk{\abs{h_a},\abs{h_k}}
    \le
    \Delta(m),
    \qquad
    \abs{\ell_{n,q,a}(m)}
    \le
    C\frac{\Delta(m)}{n},
  \]
  and
  \[
    m_a+1
    \ge
    nq_a-n^{2/3}+1
    \ge
    cn.
  \]
  Consequently,
  \[
    \abs{b_{n,q,a}(m)}
    \le
    \frac{C}{n}
    \xk{
      1+\xk{\Delta(m)+1}\frac{\Delta(m)}{n}
    }
    \le
    Cn^{-2/3}.
  \]
  The last inequality follows by substituting \(\Delta(m)\le n^{2/3}\).
  Since \(d\) is fixed, the same bound holds for \(\norm{b_{n,q}(m)}_2\), uniformly over \(q\in K\).
  Therefore, we conclude that
  \[
    \sup_{q\in K}
    \sup_{m\in A_{n,q}}
    \norm{
      r(m)-\mathbf{1}-\ell_{n,q}(m)
    }_2
    \le
    Cn^{-2/3}.
  \]
  Consequently, Cauchy--Schwarz shows that, uniformly over \(q\in K\),
  \[
    I_1 = \E_q \left[
           \ell_{n,q}(M)\ell_{n,q}(M)^\top
           \ind{M\in A_{n,q}}
    \right]
    +o(n^{-1}).
  \]
  On the other hand, direct computation shows that
  \[
    \E_q\zk{
      \ell_{n,q}(M)\ell_{n,q}(M)^\top
    }
    =
    \frac{1}{n}V_q^{-1},
  \]
  so we have
  \begin{equation}
    \label{appendix-eq:trans-prior-central-tail-split-1}
    I_1 =  \frac{1}{n}V_q^{-1} - I_{12} + o(n^{-1}),\quad I_{12} = \E_q \zk{\ell_{n,q}(M)\ell_{n,q}(M)^\top \ind{M\notin A_{n,q}}}.
  \end{equation}

  It remains to control the complement of the central event.
  Hoeffding's inequality and a union bound give constants \(c,C>0\), depending only on \(K\) and \(k\), such that
  \[
    \sup_{q\in K}
    \Pr_q\xk{M\notin A_{n,q}}
    \le
    C\exp\xk{-c n^{1/3}}.
  \]
  The definition of \(r_a\) gives, whenever \(M_k\ge1\),
  \[
    0
    \le
    r_a(M)
    \le
    Cn.
  \]
  Hence,
  \[
    \begin{aligned}
      \norm{I_2}_{\mathrm{op}}
      &\le
      \E_q\zk{
        \norm{r(M)-\mathbf{1}}_2^2
        \ind{M\notin A_{n,q}}
        \ind{M_k\ge1}
      }
      \\
      &\le
      Cn^2\exp\xk{-c n^{1/3}}
      =
      o(n^{-1}).
    \end{aligned}
  \]
  Similarly, \(\ell_{n,q}(M)\) is uniformly bounded on the whole count lattice, and hence
  \[
    \begin{aligned}
      \norm{I_{12}}_{\mathrm{op}}
      &\le
      \E_q\zk{
        \norm{\ell_{n,q}(M)}_2^2
        \ind{M\notin A_{n,q}}
      }
      \\
      &\le
      C\exp\xk{-c n^{1/3}}
      =
      o(n^{-1}).
    \end{aligned}
  \]
  Plugging these back into \cref{appendix-eq:trans-prior-central-tail-split,appendix-eq:trans-prior-central-tail-split-1} gives
  \begin{equation}
    \label{appendix-eq:trans-prior-gram}
    J^{\mathrm{prior}}_{n,q}
    =
    \frac{1}{n}V_q^{-1}+o(n^{-1}).
  \end{equation}

  \smallskip
  \noindent\emph{Privacy term.}
  Write \(m=l+e_k\), where \(l\) has total count \(n-1\), and consider the star family \(\dk{Q_{l+e_a}:a=1,\ldots,k}\).
  Any two members correspond to adjacent databases.
  Put
  \[
    \begin{aligned}
      U(l+e_k,z)
      &={}
      \xk{U_1(l+e_k,z),\ldots,U_d(l+e_k,z)}^\top,
      \\
      G_l
      &={}
      \E_{Q_{l+e_k}}\zk{
        U(l+e_k,Z)U(l+e_k,Z)^\top
      }.
    \end{aligned}
  \]
  For the application of the barycentric lemma, extend the notation by setting \(U_k(l+e_k,\cdot)=0\).
  Apply \cref{lem:barycentric} to \(R_a=Q_{l+e_a}\), \(a=1,\ldots,k\), with mixing measure \(\lambda_{\mathrm{mix}}=\delta_k\).
  Then \(\bar{R}=Q_{l+e_k}\); by \cref{appendix-eq:trans-count-shift-score-notation}, the corresponding centered likelihood ratios are \(U_a(l+e_k,\cdot)\) for \(a=1,\ldots,d\), while the \(k\)th one is zero by convention.
  By the lemma, there is a deterministic \(\delta_{\rho_n} \to1\), independent of the star family and the mechanism, such that the normalized vectors
  \[
    \phi_{l,a}
    =
    \frac{1}{\sqrt{2\rho_n}\,\delta_{\rho_n}}
    U_a(l+e_k,\cdot),
    \qquad a=1,\ldots,k,
  \]
  satisfy
  \[
    \phi_{l,k}=0,
    \qquad
    \max_{1\le a,b\le k}
    \norm{\phi_{l,a}-\phi_{l,b}}_{L^2(Q_{l+e_k})}
    \le1.
  \]
  Let \(J_l\) be the Gram matrix of \(\phi_{l,1},\ldots,\phi_{l,d}\) in \(L^2(Q_{l+e_k})\).
  By \cref{appendix-eq:trans-config-region}, the preceding bound on pairwise distances gives \(J_l\in\mathcal{G}_k\).
  By the definition of \(G_l\),
  \[
    G_l
    =
    2\rho_n\delta_{\rho_n}^2J_l.
  \]

  Let \(\pi_{n-1,q}\) denote the \(\operatorname{Multinomial}(n-1,q)\) law and put
  \[
    c_a(l)=\frac{nq_a}{l_a+1},
    \qquad a=1,\ldots,k,
    \qquad
    w_{ab}(l)=\frac{c_a(l)c_b(l)}{c_k(l)},
    \qquad
    a,b=1,\ldots,d.
  \]
  Direct calculation gives
  \[
    \pi_q(l+e_k)
    =
    c_k(l)\pi_{n-1,q}(l),
    \qquad
    r_a(l+e_k)
    =
    \frac{c_a(l)}{c_k(l)}.
  \]
  Reindexing the definition of \(H^{\mathrm{priv}}_{n,q}\) by \(m=l+e_k\) and using the preceding identities gives
  \[
    \begin{aligned}
      [H^{\mathrm{priv}}_{n,q}]_{ab}
      &={}
      \sum_l
      \pi_q(l+e_k)
      r_a(l+e_k)r_b(l+e_k)
      [G_l]_{ab}
      \\
      &={}
      2\rho_n\delta_{\rho_n}^2
      \sum_l
      \pi_{n-1,q}(l)
      w_{ab}(l)[J_l]_{ab}.
    \end{aligned}
  \]
  Define
  \[
    \overline{J}_{n,q}
    =
    \sum_l \pi_{n-1,q}(l)J_l
    \in\mathcal{G}_k.
  \]
  Splitting \(w_{ab}(l)=1+\xk{w_{ab}(l)-1}\) gives the exact decomposition
  \begin{equation}
    \label{appendix-eq:trans-privacy-weighted-gram}
    H^{\mathrm{priv}}_{n,q}
    =
    2\rho_n\delta_{\rho_n}^2\xk{\overline{J}_{n,q}+H_{12}},
    \qquad
    [H_{12}]_{ab}
    =
    \sum_l
    \pi_{n-1,q}(l)
    \xk{w_{ab}(l)-1}[J_l]_{ab}.
  \end{equation}

  To control the remainder, define
  \[
    B_{n,q}
    =
    \dk{
      l:
      \max_{1\le a\le k}
      \abs{l_a-(n-1)q_a}
      \le n^{2/3}
    }.
  \]
  On \(B_{n,q}\), direct calculation gives
  \[
    c_a(l)-1
    =
    -\frac{l_a-(n-1)q_a+1-q_a}{l_a+1},
    \qquad
    \sup_{q\in K}
    \sup_{l\in B_{n,q}}
    \max_{1\le a,b\le d}
    \abs{w_{ab}(l)-1}
    \le
    Cn^{-1/3}.
  \]
  Hoeffding's inequality and a union bound give
  \[
    \sup_{q\in K}
    \sum_{l\notin B_{n,q}}
    \pi_{n-1,q}(l)
    \le
    C\exp\xk{-cn^{1/3}}.
  \]
  On the whole count lattice,
  \[
    0<c_a(l)\le Cn,
    \qquad
    c_k(l)^{-1}\le C,
    \qquad
    0\le w_{ab}(l)\le Cn^2.
  \]
  Since \(J_l\in\mathcal{G}_k\) implies \(\abs{[J_l]_{ab}}\le1\) and \(d\) is fixed, the preceding central and tail bounds yield
  \[
    \norm{H_{12}}_{\mathrm{op}}
    \le
    C\xk{
      n^{-1/3}
      +n^2\exp\xk{-cn^{1/3}}
    }
    =
    o(1).
  \]
  Set \(J_{n,q}=\overline{J}_{n,q}\).
  Plugging this bound into \cref{appendix-eq:trans-privacy-weighted-gram} and using \(\delta_{\rho_n} \to1\) give
  \begin{equation}
    \label{appendix-eq:trans-privacy-gram}
    H^{\mathrm{priv}}_{n,q}
    =
    2\rho_n(J_{n,q}+o(1)),
    \qquad
    J_{n,q} \in\mathcal{G}_k.
  \end{equation}
  The privacy remainder is uniform over both the fixed compact subset and all mechanisms by the uniform barycentric bound.
  The prior remainder is uniform over the fixed compact subset.
  Since \(\lambda_n=n\rho_n \to\lambda\in(0,\infty)\), the privacy remainder \(o(\rho_n)\) is also \(o(n^{-1})\).
  Combining \cref{appendix-eq:trans-joint-score-gram-split,appendix-eq:trans-prior-gram,appendix-eq:trans-privacy-gram} proves \cref{appendix-eq:trans-total-joint-score-gram} and completes the proof.
\end{proof}

Define the information map for a finite alphabet
\begin{equation}
  \label{appendix-eq:trans-categorical-map}
  \caT_{\lambda,q}(J)
  =
  V_q^{-1}
  -V_q^{-1}
  (V_q^{-1}+2\lambda J)^{-1}
  V_q^{-1}.
\end{equation}
Let a differentiable submodel on a finite alphabet have probability vector
\[
  q(\theta)=(q_1(\theta),\ldots,q_k(\theta)),
\]
and let
\[
  D_\theta
  =
  \frac{\partial q_{1:d}(\theta)}{\partial\theta^\top}.
\]
Define
\begin{equation}
  \label{appendix-eq:trans-finite-submodel-map}
  F_{\theta,\lambda}(J)
  =
  D_\theta^\top
  \caT_{\lambda,q(\theta)}(J)
  D_\theta.
\end{equation}

\begin{corollary}[Gaussian representation on a finite alphabet]
  \label{appendix-cor:trans-fin-alphabet-gauss}
  Suppose \(q(\theta)\) ranges over a fixed compact subset of the simplex interior and \(D_\theta\) is uniformly bounded.
  For every \(\rho_n\)-zCDP mechanism, there exists \(J_{n,\theta}\in\mathcal{G}_k\) such that, uniformly over this parameter set and all such mechanisms,
  \begin{equation}
    \label{appendix-eq:trans-finite-submodel-bound}
    \frac{1}{n}I_{M_n}(\theta)
    \preceq
    F_{\theta,\lambda_n}(J_{n,\theta})+o(1)I_p.
  \end{equation}
  For fixed \((\theta,\lambda)\), the map \(J\mapsto F_{\theta,\lambda}(J)\) is concave in the Loewner order.
  Moreover, for every \(J\in\mathcal{G}_k\), there is a statistic \(\Phi\), constant on the alphabet and with diameter at most one, such that
  \[
    F_{\theta,\lambda}(J)
    =
    A_{\Phi,\theta}^\top
    \xk{V_{\Phi,\theta}+\frac{1}{2\lambda}I}^{-1}
    A_{\Phi,\theta}.
  \]
\end{corollary}

\begin{proof}
  The information bound follows from \cref{appendix-thm:trans-categorical-contraction} and the chain rule.
  Matrix inversion is operator convex, so \(J\mapsto\caT_{\lambda,q}(J)\) is concave in the Loewner order; congruence by \(D_\theta\) gives the same property for \(F_{\theta,\lambda}\).
  For \(J\in\mathcal{G}_k\), choose a generating configuration \(\phi_1,\ldots,\phi_k\), translate it so that \(\phi_k=0\), and define \(U:\R^d\to\caH\) by \(Ue_a=\phi_a\).
  Then \(J=U^*U\).
  The categorical statistic \(\Phi(a)=\phi_a\) has diameter at most one, covariance \(UV_qU^*\), and derivative \(U\) in the saturated coordinates.
  The Woodbury identity gives
  \begin{equation}
    \label{appendix-eq:trans-categorical-gauss-identity}
    \caT_{\lambda,q}(J)
    =
    U^*
    \xk{UV_qU^*+\frac{1}{2\lambda}I}^{-1}
    U,
  \end{equation}
  first for nonsingular \(J\) and then by continuity.
  In the submodel, \(V_{\Phi,\theta}=UV_{q(\theta)}U^*\) and \(A_{\Phi,\theta}=UD_\theta\), so congruence by \(D_\theta\) proves the Gaussian representation.
\end{proof}

\subsection{Transfer from finite partitions to the general model}
\label{appendix-subsec:trans-partition-transfer}

The transfer has two steps.
First, freezing the conditional law within each cell produces a zCDP experiment on a finite alphabet with the same output law at the frozen parameter.
We then choose a fixed partition whose cell score uniformly approximates the original score.

Fix \(\theta\) and a finite measurable partition \(\mathcal{P}=\dk{C_1,\ldots,C_k}\) with \(q_j(\theta)=P_\theta(C_j)>0\).
Let \(R_{\theta,j}=P_\theta(\,\cdot\mid C_j)\).
For an arbitrary mechanism \(M_n:\caX^n \to\caZ_n\), define the frozen cell channel
\begin{equation}
  \label{appendix-eq:trans-frozen-cell-channel}
  N_{n,\theta,\mathcal{P}}(dz\mid c_{1:n})
  =
  \int
  M_n(dz\mid x_{1:n})
  \prod_{i=1}^n R_{\theta,c_i}(dx_i).
\end{equation}
This channel remains \(\rho_n\)-zCDP.
Indeed, if \(c_{1:n}\) and \(c'_{1:n}\) differ in one coordinate, couple the observations identically in every unchanged coordinate and use any coupling in the remaining coordinate.
Every resulting pair of databases is adjacent, so joint convexity of the R\'enyi Hellinger integral preserves the same zCDP bound after mixing.
Freeze the conditional distributions \(R_{\theta,j}\), but vary the cell probabilities, by setting
\[
  Q^{(\theta,\mathcal{P})}_{\vartheta;n}
  =
  \sum_{c_{1:n}}
  q(\vartheta)^{\otimes n}(c_{1:n})
  N_{n,\theta,\mathcal{P}}(\,\cdot\mid c_{1:n}).
\]
At \(\vartheta=\theta\), this experiment has the same output marginal as \(M_n \circ P_\theta^{\otimes n}\), but its score keeps only the cellwise conditional expectation of the original score.

\begin{lemma}[Uniform score approximation by a frozen partition]
  \label{appendix-lem:trans-partition-score-approx}
  For every \(\ep>0\), there is a fixed finite partition \(\mathcal{P}\), all of whose cells have positive \(P_{\theta_0}\)-probability.
  The partition may depend on \(\ep\), but not on \(n\), \(\theta\), or the mechanism.
  There is a constant \(C_{\theta_0,\lambda}\ge1\) such that, uniformly over \(\theta\in\Theta_n\) and all \(\rho_n\)-zCDP mechanisms,
  \begin{equation}
    \label{appendix-eq:trans-partition-info-approx}
    \norm{
      n^{-1} I_{M_n}(\theta)
      -n^{-1} I_{Q^{(\theta,\mathcal{P})}_{\vartheta;n}}(\theta)
    }_{\mathrm{op}}
    \le
    C_{\theta_0,\lambda}(\ep+\ep^2)+o(1).
  \end{equation}
\end{lemma}

\begin{proof}
  \noindent\emph{Comparison of score operators.}
  Write \(Q_\theta=M_n \circ P_\theta^{\otimes n}\) for the common output marginal at the frozen center.
  The cell score is
  \[
    s_{\theta,\mathcal{P}}(j)
    =
    \frac{\dot{q}_j(\theta)}{q_j(\theta)}
    =
    \E_\theta[s_\theta(X)\mid C_j].
  \]
  On \(C_j\), define the signed density residual
  \[
    a_{\theta,j}(x)
    =
    \dot{p}_\theta(x)
    -p_\theta(x)s_{\theta,\mathcal{P}}(j),
    \qquad
    \int_{C_j} a_{\theta,j}\,d\mu=0,
  \]
  and put
  \[
    \mathcal{R}_{\mathcal{P}}(\theta)
    =
    \sum_j \int_{C_j} \norm{a_{\theta,j}(x)}_2 d \mu
    =
    \E_\theta \norm{
      s_\theta-\E_\theta(s_\theta \mid\mathcal{P})
    }_2.
  \]
  If \(B_{M,\theta}\) and \(B_{\mathcal{P},\theta}\) are the two output score operators, the representation in \cref{appendix-lem:cond-score-operator} and the zero integral above give
  \[
    (B_{M,\theta}-B_{\mathcal{P},\theta})h
    =
    \sum_{i,j} \int_{C_j}
    h^\top a_{\theta,j}(x)
    (L_{i,x}-\overline{L}_{i,j})
    d\mu(x),
  \]
  where \(L_{i,x}\) is the conditional output likelihood ratio and \(\overline{L}_{i,j}=\E_\theta(L_{i,X} \mid C_j)\).
  For \(x,x'\in C_j\), the uniform form of \cref{lem:barycentric} yields
  \[
    \norm{L_{i,x}-L_{i,x'}}_{L^2(Q_\theta)}
    \le
    \sqrt{2\rho_n}\,\delta_{\rho_n}.
  \]
  Averaging in \(x'\) conditionally on \(C_j\), and then using the triangle inequality in the preceding score representation, gives
  \begin{equation}
    \label{appendix-eq:trans-score-operator-approx}
    \frac{\norm{B_{M,\theta}-B_{\mathcal{P},\theta}}_{\mathrm{op}}}{\sqrt n}
    \le
    \sqrt{2n\rho_n}\,\delta_{\rho_n} \mathcal{R}_{\mathcal{P}}(\theta),
    \qquad
    \delta_{\rho_n} \to1.
  \end{equation}

  \smallskip
  \noindent\emph{Partition choice and information comparison.}
  Simple functions are dense in \(L^1(P_{\theta_0};\R^p)\).
  If a \(\mathcal{P}\)-measurable simple function \(g\) satisfies \(\E_{\theta_0} \norm{s_{\theta_0}-g}_2<\ep/2\), conditional expectation is an \(L^1\) contraction and hence
  \[
    \mathcal{R}_{\mathcal{P}}(\theta_0)
    \le
    2\E_{\theta_0} \norm{s_{\theta_0}-g}_2
    <\ep.
  \]
  Thus a partition can be chosen with \(\mathcal{R}_{\mathcal{P}}(\theta_0)<\ep\); null cells may be merged into a cell of positive probability.
  For this fixed partition,
  \[
    \mathcal{R}_{\mathcal{P}}(\theta)
    =
    \sum_j
    \norm{
      \dot{p}_\theta \mathbf{1}_{C_j}
      -p_\theta \mathbf{1}_{C_j}
      \frac{\dot{q}_j(\theta)}{q_j(\theta)}
    }_{L^1(\mu)}.
  \]
  All \(q_j(\theta)\) remain bounded away from zero near \(\theta_0\), and \cref{ass:score-density-l1} makes every term continuous.
  Hence
  \[
    \sup_{\theta\in\Theta_n} \mathcal{R}_{\mathcal{P}}(\theta)
    \le2\ep
  \]
  for all sufficiently large \(n\).

  Finally, the norms of the normalized score operators are locally bounded uniformly over mechanisms.
  For the original experiment this follows from \cref{thm:diameter-contraction}, local boundedness of the diameter information region, and \(n\rho_n=O(1)\).
  The frozen score operator obeys the same bound because
  \[
    \frac{\norm{B_{\mathcal{P},\theta}}_{\mathrm{op}}}{\sqrt n}
    \le
    \frac{\norm{B_{M,\theta}}_{\mathrm{op}}}{\sqrt n}
    +
    \frac{\norm{B_{M,\theta}-B_{\mathcal{P},\theta}}_{\mathrm{op}}}{\sqrt n}.
  \]
  Since \(I=B^*B\), the identity \(B^*B-C^*C=B^*(B-C)+(B^*-C^*)C\) gives
  \[
    \begin{aligned}
      &\frac{1}{n}
      \norm{
        I_{M_n}(\theta)
        -I_{Q^{(\theta,\mathcal{P})}_{\vartheta;n}}(\theta)
      }_{\mathrm{op}}
      \\
      &\qquad\le
      \xk{
        \frac{\norm{B_{M,\theta}}_{\mathrm{op}}}{\sqrt n}
        +
        \frac{\norm{B_{\mathcal{P},\theta}}_{\mathrm{op}}}{\sqrt n}
      }
      \frac{\norm{B_{M,\theta}-B_{\mathcal{P},\theta}}_{\mathrm{op}}}{\sqrt n}
      \\
      &\qquad\le
      C_{\theta_0,\lambda}(\ep+\ep^2)+o(1),
    \end{aligned}
  \]
  where the last bound uses \cref{appendix-eq:trans-regime,appendix-eq:trans-score-operator-approx} and the uniform estimate \(\sup_{\theta\in\Theta_n}\mathcal{R}_{\mathcal{P}}(\theta)\le2\ep\).
  This proves \cref{appendix-eq:trans-partition-info-approx}.
\end{proof}

\begin{lemma}[Uniform averaged information bound]
  \label{appendix-lem:trans-averaged-information}
  Under the assumptions and notation of this section, fix a compact neighborhood
  \(K\subset\Theta\) of \(\theta_0\) containing \(\Theta_n\) for all sufficiently large \(n\).
  There is a constant \(C_{\theta_0,\lambda}\ge1\), depending only on the model on
  \(K\) and on \(\lambda\), with the following property.
  For every fixed \(\ep>0\), there is a deterministic sequence
  \(r_{n,\ep}\ge0\) tending to zero such that, for all sufficiently large \(n\),
  every \(\rho_n\)-zCDP mechanism \(M_n\) and every probability measure \(\pi_n\)
  supported on \(\Theta_n\) admit \(L_{n,\ep}\in\beta\caU_{\theta_0,\beta}\) satisfying
  \begin{equation}
    \label{appendix-eq:trans-averaged-info-bound}
    \frac1n\int I_{M_n}(\theta)\pi_n(d\theta)
    \preceq L_{n,\ep}
    +\bigl[C_{\theta_0,\lambda}(\ep+\ep^2)+r_{n,\ep}\bigr]I_p.
  \end{equation}
  The constant is independent of \(\ep\), the partition, the mechanism, and the prior.
  The remainder and the threshold for \(n\) may depend on the fixed partition
  chosen for \(\ep\), but are uniform over mechanisms and priors.
\end{lemma}

\begin{proof}
  Fix \(\ep>0\) and take the partition from \cref{appendix-lem:trans-partition-score-approx}.
  Its cell probabilities are bounded away from zero on a neighborhood of
  \(\theta_0\), and its derivative \(D_\theta\) is continuous there.
  Thus, for all sufficiently large \(n\), \cref{appendix-cor:trans-fin-alphabet-gauss}
  applies uniformly over \(\Theta_n\) to the frozen cell channels, including
  their dependence on the frozen parameter.
  Compactness of \(\mathcal G_k\), continuity of \(q\) and \(D\), and
  \(\lambda_n\to\lambda>0\) also give
  \[
    \sup_{\theta\in\Theta_n,\,J\in\mathcal G_k}
    \norm{F_{\theta,\lambda_n}(J)-F_{\theta_0,\lambda}(J)}_{\mathrm{op}}
    \longrightarrow0.
  \]
  Combining these two bounds with \cref{appendix-eq:trans-partition-info-approx},
  there is a deterministic \(r_{n,\ep}\to0\), uniform over mechanisms, such that
  for each \(\theta\in\Theta_n\) some \(J\in\mathcal G_k\) satisfies
  \[
    n^{-1}I_{M_n}(\theta)
    \preceq F_{\theta_0,\lambda}(J)+e_{n,\ep}I_p,
    \qquad
    e_{n,\ep}=C_{\theta_0,\lambda}(\ep+\ep^2)+r_{n,\ep}.
  \]
  The constant in the partition error is independent of the partition:
  \cref{appendix-eq:trans-score-operator-approx} and the local contraction bound control
  the score-operator norms using only the model on \(K\) and an eventual bound
  on \(\lambda_n\).
  The factors involving cell probabilities enter only the fixed-partition
  remainders and the threshold for \(n\).

  To average this inequality without choosing \(J\) measurably, put
  \[
    \mathcal D
    =\{B\in\bbS^p:B\preceq F_{\theta_0,\lambda}(J)
      \text{ for some }J\in\mathcal G_k\}.
  \]
  This set is closed by compactness of \(\mathcal G_k\) and continuity of \(F\),
  and convex by its Loewner concavity.
  The output Fisher matrix has a measurable version, obtained from measurable
  likelihood-ratio score versions on the standard Borel output space.
  Its positive semidefiniteness and the preceding compact information bound
  make its entries integrable under \(\pi_n\).
  Since \(n^{-1}I_{M_n}(\theta)-e_{n,\ep}I_p\in\mathcal D\), its integral
  belongs to the same closed convex set.
  Consequently, some \(\overline{J}_{n,\ep}\in\mathcal G_k\) satisfies
  \[
    \frac1n\int I_{M_n}(\theta)\pi_n(d\theta)
    \preceq F_{\theta_0,\lambda}(\overline{J}_{n,\ep})+e_{n,\ep}I_p.
  \]
  By \cref{appendix-cor:trans-fin-alphabet-gauss}, choose a categorical statistic \(\psi_n\) of diameter at most one that realizes \(F_{\theta_0,\lambda}(\overline{J}_{n,\ep})\), and pull it back to \(\caX\) by setting \(\Phi_n(x)=\psi_n(j)\) for \(x\in C_j\).
  Since
  \[
    P_\theta\circ\Phi_n^{-1}
    =
    \sum_{j=1}^k q_j(\theta)\delta_{\psi_n(j)},
  \]
  the pullback has the same expectation map as \(\psi_n\) under the categorical law \(q(\theta)\), and hence preserves the derivative and covariance at \(\theta_0\):
  \[
    A_{\Phi_n,0}=A_{\psi_n,\theta_0},
    \qquad
    V_{\Phi_n,0}=V_{\psi_n,\theta_0}.
  \]
  It also satisfies \(\diam(\Phi_n)\le1\).
  Using \(\beta=\lambda/(1+\lambda)\) and
  \(
    \Sigma_{\Phi_n,\theta_0,\beta}
    =
    \beta\xk{V_{\Phi_n,0}+(2\lambda)^{-1}I}
  \), we obtain
  \[
    F_{\theta_0,\lambda}(\overline{J}_{n,\ep})
    =
    \beta A_{\Phi_n,0}^\top
    \Sigma_{\Phi_n,\theta_0,\beta}^{-1}
    A_{\Phi_n,0}
    \in
    \beta\caU_{\theta_0,\beta}.
  \]
  Taking \(L_{n,\ep}=F_{\theta_0,\lambda}(\overline{J}_{n,\ep})\)
  proves \cref{appendix-eq:trans-averaged-info-bound}.
  All limits above are taken with the partition, and hence \(\ep\), fixed;
  \(\ep\) is sent to zero only after \(n\to\infty\).
\end{proof}

\subsection{Proof of \Cref{appendix-thm:trans-local-lower}}
\label{appendix-subsec:trans-local-lower-proof}

  \noindent\emph{Shrinking prior.}
  Choose a continuously differentiable density \(\varphi\), compactly supported in the interior of the unit ball, with finite positive definite Fisher matrix and vanishing boundary term.
  Let \(b_n \to\infty\) satisfy \(b_n \le R_n/2\), and define the shrinking prior
  \[
    \pi_n(\theta)
    =
    \xk{\frac{\tau_n}{b_n}}^p
    \varphi\xk{
      \frac{\tau_n(\theta-\theta_0)}{b_n}
    }.
  \]
  Its support is contained in \(\Theta_n\), and
  \begin{equation}
    \label{appendix-eq:trans-prior-info-vanishes}
    \frac{J_{\pi_n}}{n}
    =
    \frac{\lambda_n}{b_n^2}J_\varphi
    \longrightarrow0.
  \end{equation}

  \smallskip
  \noindent\emph{Averaged information bound.}
  Fix \(\ep>0\) and apply \cref{appendix-lem:trans-averaged-information} to this prior.
  This gives \cref{appendix-eq:trans-averaged-info-bound}, with a deterministic remainder
  tending to zero uniformly over mechanisms for this fixed \(\ep\).

  \smallskip
  \noindent\emph{Van Trees step.}
  The output experiment is DQM by \cref{appendix-lem:markov-dqm}.
  The same verification of differentiation under the integral and integration by parts as in \cref{appendix-subsec:matrix-local-lower} applies because the density map is continuously differentiable in \(L^1\), the prior is compactly supported, and the contraction bound makes the integrated output information finite.
  Write
  \[
    \operatorname{BayesRisk}_{\pi_n}(W)
    =
    \int
    \E_\theta
    \norm{\widehat{\theta}-\theta}_W^2
    \pi_n(d\theta).
  \]
  We may therefore apply \cref{appendix-lem:matrix-van-trees} to any mechanism and estimator.
  By \cref{appendix-eq:trans-prior-info-vanishes,appendix-eq:trans-averaged-info-bound} and inverse monotonicity,
  \[
    n\operatorname{BayesRisk}_{\pi_n}(W)
    \ge
    \caI_W(
    L_{n,\ep}+a_{n,\ep} I_p
    ),
  \]
  where
  \[
    a_{n,\ep}
    =
    C_{\theta_0,\lambda}(\ep+\ep^2)
    +r_{n,\ep}
    +\frac{\norm{J_{\pi_n}}_{\mathrm{op}}}{n}.
  \]
  By \cref{appendix-prop:unified-region-geometry}\textup{(iii)}, \(\beta\caU_{\theta_0,\beta}\) is compact and bounded above by \(I(\theta_0)\).
  For fixed \(\ep>0\), take a subsequence realizing the liminf and then a further subsequence on which \(L_{n,\ep} \to L_\ep \in\beta\caU_{\theta_0,\beta}\).
  Since \(a_{n,\ep} \to c_\ep \coloneqq C_{\theta_0,\lambda}(\ep+\ep^2)>0\), continuity of inverse trace on the positive definite cone gives
  \[
    \liminf_{n\to\infty}
    n\operatorname{BayesRisk}_{\pi_n}(W)
    \ge
    \inf_{L\in\beta\caU_{\theta_0,\beta}}
    \caI_W(L+c_\ep I_p).
  \]
  As \(c\downarrow0\), compactness of the region and lower semicontinuity of \(\caI_W\) imply
  \[
    \inf_{L\in\beta\caU_{\theta_0,\beta}}
    \caI_W(L+cI_p)
    \uparrow
    \inf_{L\in\beta\caU_{\theta_0,\beta}}
    \caI_W(L).
  \]
  Indeed, approximate minimizers have a convergent subsequence, and lower semicontinuity gives the lower inequality at the limit; the reverse inequality follows from \(\caI_W(L+cI_p)\le\caI_W(L)\).
  Letting \(\ep\downarrow0\) therefore gives
  \[
    \liminf_{n\to\infty}
    n\operatorname{BayesRisk}_{\pi_n}(W)
    \ge
    \inf_{L\in\beta\caU_{\theta_0,\beta}}\caI_W(L)
    =
    \frac{\uCt(W)}{\beta}.
  \]
  The final equality uses \(\caI_W(\beta H)=\beta^{-1} \caI_W(H)\).
  Since \(\supp(\pi_n)\subset\Theta_n\), the Bayes risk is bounded above by the corresponding local worst-case risk.
  Taking the infimum over all mechanisms proves the lower bound.
   \clearpage
  \section{Completion of main results}
\label{appendix-sec:result-assembly}

This section assembles the lower bounds for the individual regimes and the matching upper bounds into the unified local and compact results stated in the main paper.

\subsection{Proof of \Cref{thm:bounded-stat-upper}}
\label{appendix-subsec:fixed-stat-upper}

Since \(R_n\sigma_n\to0\), the sets \(\widetilde\Theta_n\) shrink to \(\theta_0\).
The result follows from \cref{appendix-lem:unified-fixed-stat} with \(\mathcal T_n=\widetilde\Theta_n\).

\subsection{Proof of \Cref{lem:verify-near-optimal}}
\label{appendix-subsec:near-optimal-stat}

Apply \cref{appendix-lem:unified-near-optimal-stat} at the fixed value of \(\beta\) in the statement.

\subsection{Proof of \Cref{thm:local-attain-seq}}
\label{appendix-subsec:local-attain-seq}

The risk bound follows from \cref{appendix-lem:unified-local-upper}.
Its construction uses the \(\R^p\)-valued diameter-one statistics from \cref{appendix-lem:unified-near-optimal-stat} and therefore has the form asserted in the theorem.

\subsection{Proof of \Cref{thm:gauss-attain-limit}}
\label{appendix-subsec:gauss-limit}

Apply \cref{appendix-lem:unified-gauss-diagonalization} to the specified matrix \(H_{W,\beta}\).
The sequence attains the risk bound and has the uniform Gaussian limit in \cref{eq:attain-seq-gauss-limit}.

\subsection{Proof of \Cref{thm:private-minimax}}
\label{appendix-subsec:privacy-dominated-local}

The lower bound is \cref{thm:matrix-local-lower}.
For the reverse inequality, \cref{thm:local-attain-seq} constructs a sequence of \(\rho_n\)-zCDP estimators whose scaled worst-case risk has limit superior at most \(\mfC_{\theta_0}(W)\).
The two bounds give the asserted equality.

\subsection{Proof of \Cref{thm:local-minimax}}
\label{appendix-subsec:unified-local-lower}

The upper bound is \cref{appendix-lem:unified-local-upper}.
For the matching lower bound, write \(\lambda_n=n\rho_n\) and consider
the three possible limits of \(\beta_n\).

\textit{Privacy-dominated regime: \(\beta=0\).}
If \(\beta=0\), then \(\lambda_n\to0\) and \(\sigma_n^{-2}/\tau_n^2\to1\).
Rewrite \(\widetilde\Theta_n\) on the \(\tau_n\)-scale with radius
\(R_n\tau_n\sigma_n\).
This radius diverges and its ratio to \(\tau_n\) is
\(R_n\sigma_n\to0\).
The lower bound therefore follows from
\cref{thm:matrix-local-lower} and part~\textup{(i)} of
\cref{appendix-lem:unified-info-identification}.

\textit{Transition regime: \(0<\beta<1\).}
If \(0<\beta<1\), put
\(\lambda=\beta/(1-\beta)\).
The same change of scale gives a transition local set with radius
\(R_n\tau_n\sigma_n\).
Apply \cref{appendix-thm:trans-local-lower} and use
\(\sigma_n^{-2}/n=\beta_n\to\beta\).

\textit{Sampling-dominated regime: \(\beta=1\).}
Finally, suppose \(\beta=1\).
Private procedures are a subclass of randomized estimators.
Since \(n\sigma_n^2\to1\), the set \(\widetilde\Theta_n\) is a
root-\(n\) ball whose radius in root-\(n\) coordinates is
\(R_n\sqrt n\,\sigma_n\to\infty\).
Hence every fixed root-\(n\) local ball is eventually contained in
\(\widetilde\Theta_n\).
The classical local asymptotic minimax theorem
\citep[Theorem~8.11]{vaart1998_AsymptoticStatistics}, first for truncated
quadratic loss and then with the truncation level and local radius tending
to infinity, gives the lower bound after using
\(\sigma_n^{-2}/n\to1\).
Part~\textup{(ii)} of
\cref{appendix-lem:unified-info-identification} identifies the constant
with \(\mfC_{\theta_0}^{(1)}(W)\).

\subsection{Proof of \Cref{thm:compact-minimax}}
\label{appendix-subsec:compact-theorem}

The upper bound is \cref{appendix-lem:unified-compact-upper}.
For the lower bound, fix \(\theta_0\in K^\circ\), choose
\(R_n\to\infty\) sufficiently slowly that \(R_n\sigma_n\to0\), and form
the unified local sets
\[
  \widetilde\Theta_n
  =
  \dk{
    \theta_0+\sigma_n u:
    \norm{u}_2\le R_n
  }.
\]
These sets are eventually contained in \(K\).
For every estimator,
\[
  \sup_{\theta\in K}
  \frac{
    \sigma_n^{-2}
    \E_\theta
    \norm{\widehat\theta-\theta}_W^2
  }{
    \uC_\theta(W)
  }
  \ge
  \frac{
    \sigma_n^{-2}
    \sup_{\theta\in\widetilde\Theta_n}
    \E_\theta
    \norm{\widehat\theta-\theta}_W^2
  }{
    \sup_{\theta\in\widetilde\Theta_n}
    \uC_\theta(W)
  }.
\]
By \cref{appendix-lem:unified-constant-continuity}, the denominator is \(\uCt(W)\{1+o(1)\}\).
Taking the infimum over all private estimators and applying
\cref{thm:local-minimax} gives the lower bound.

\subsection{Proof of \Cref{thm:compact-attain-seq}}
\label{appendix-subsec:compact-attainment}

The relative risk bound follows from \cref{appendix-lem:unified-compact-upper}.
Its proof uses \cref{eq:private-localization-estimator,appendix-eq:compact-two-stage-estimator} for the first and second stages, respectively, so the resulting estimators have the asserted two-stage form.

\subsection{Proof of \Cref{cor:compact-minimax}}
\label{appendix-subsec:compact-minimax-corollary}

Put
\[
  C_{K,\beta}(W)
  =
  \max_{\theta\in K}\uC_\theta(W).
\]
The relative upper bound in \cref{appendix-lem:unified-compact-upper}, together with the boundedness in \cref{appendix-lem:unified-constant-continuity}, gives the upper bound \(C_{K,\beta}(W)\).
Conversely, embed a unified shrinking local set around any
\(\theta_0\in K^\circ\) and apply
\cref{thm:local-minimax}.
Taking the supremum over \(K^\circ\), then using
\(K=\overline{K^\circ}\) and \cref{appendix-lem:unified-constant-continuity}, gives the lower bound \(C_{K,\beta}(W)\).

\subsection{Efficiency for differentiable targets}
\label{appendix-subsec:differentiable-target-efficiency}

Let \(\psi:\Theta\to\R^m\) be a differentiable target, and let \(W\succ0\)
specify its quadratic loss.
As in \cref{eq:target-efficiency-constant}, set
\[
  W_{\psi,\theta}
  =
  D\psi(\theta)^\top W D\psi(\theta),
  \qquad
  \mfC_{\psi,\theta}^{(\beta)}(W)
  =
  \mfC_\theta^{(\beta)}(W_{\psi,\theta}).
\]

\begin{theorem}[Local efficiency for a differentiable target]
  \label{appendix-thm:differentiable-target-local}
  Under \cref{ass:regular,ass:score-density-l1}, suppose
  \(I(\theta_0)\succ0\).
  Let \((\rho_n)\) be a positive sequence such that
  \[
    n^2\rho_n\to\infty,
    \qquad
    \beta_n\to\beta\in[0,1].
  \]
  Let \(\psi:\Theta\to\R^m\) be differentiable at \(\theta_0\), and fix
  \(W\succ0\).
  Let \(\widetilde\Theta_n\) be as in
  \cref{eq:unified-local-param-set}, with \(R_n\) chosen sufficiently
  slowly. Then
  \begin{equation}
    \label{appendix-eq:differentiable-target-local-risk}
    \lim_{n\to\infty}
    \sigma_n^{-2}
    \inf_{\widehat\psi_n:\rho_n\text{-zCDP}}
    \sup_{\theta\in\widetilde\Theta_n}
    \E_\theta
    \norm{\widehat\psi_n-\psi(\theta)}_W^2
    =
    \mfC_{\psi,\theta_0}^{(\beta)}(W).
  \end{equation}
\end{theorem}

\begin{theorem}[Compact efficiency for a differentiable target]
  \label{appendix-thm:differentiable-target-compact}
  Let \(K\subset\Theta\) be a nonempty compact set satisfying
  \(K=\overline{K^\circ}\).
  Under \cref{ass:regular,ass:score-density-l1}, suppose
  \(I(\theta)\succ0\) for every \(\theta\in K\),
  \(\theta\mapsto I(\theta)\) is continuous on \(K\), and the model is
  identifiable on \(K\).
  Let \((\rho_n)\) be a positive sequence such that
  \[
    n^2\rho_n\to\infty,
    \qquad
    \beta_n\to\beta\in[0,1].
  \]
  Suppose that \(\psi:\Theta\to\R^m\) is continuously differentiable on a
  neighborhood of \(K\) and \(D\psi(\theta)\ne0\) for every \(\theta\in K\),
  and fix \(W\succ0\).
  Then \(\theta\mapsto\mfC_{\psi,\theta}^{(\beta)}(W)\) is continuous and
  strictly positive on \(K\), and
  \begin{equation}
    \label{appendix-eq:differentiable-target-compact-risk}
    \lim_{n\to\infty}
    \inf_{\widehat\psi_n:\rho_n\text{-zCDP}}
    \sup_{\theta\in K}
    \frac{
      \sigma_n^{-2}
      \E_\theta\norm{\widehat\psi_n-\psi(\theta)}_W^2
    }{
      \mfC_{\psi,\theta}^{(\beta)}(W)
    }
    =1.
  \end{equation}
  Consequently,
  \begin{equation}
    \label{appendix-eq:differentiable-target-compact-minimax}
    \lim_{n\to\infty}
    \sigma_n^{-2}
    \inf_{\widehat\psi_n:\rho_n\text{-zCDP}}
    \sup_{\theta\in K}
    \E_\theta\norm{\widehat\psi_n-\psi(\theta)}_W^2
    =
    \max_{\theta\in K}
    \mfC_{\psi,\theta}^{(\beta)}(W).
  \end{equation}
\end{theorem}

\begin{proof}[Modifications for differentiable targets]
For the local result, put
\[
  B_0=D\psi(\theta_0),
  \qquad
  W_{\psi,0}=B_0^\top W B_0.
\]
For the linearized target
\[
  \psi_0^{\mathrm{lin}}(\theta)
  =
  \psi(\theta_0)+B_0(\theta-\theta_0),
\]
estimating \(\psi_0^{\mathrm{lin}}(\theta)\) under loss \(W\) is equivalent to
estimating \(\theta\) under loss \(W_{\psi,0}\).
In one direction, apply the affine map above to any parameter estimator.
In the other, project a target estimator onto
\(\psi(\theta_0)+\ran(B_0)\) with respect to \(W\), then lift the projection
using a fixed right inverse of \(B_0\) on its range.
Now apply \cref{thm:local-minimax}, which permits a positive semidefinite
loss matrix, and choose \(R_n\) slowly enough that
\[
  \sup_{\norm{u}_2\le R_n}
  \sigma_n^{-1}
  \norm{
    \psi(\theta_0+\sigma_nu)
    -\psi(\theta_0)
    -\sigma_nB_0u
  }_W
  \longrightarrow0.
\]
This uniform linearization allows us to replace the affine target by
\(\psi(\theta)\) without changing the limiting risk.

For the compact result, replace the fixed loss matrix in
\cref{appendix-lem:unified-constant-continuity,appendix-lem:unified-compact-charts} by
\(W_{\psi,\theta}=D\psi(\theta)^\top W D\psi(\theta)\).
Because \(D\psi\) is continuous, the lower semicontinuity argument and the
construction using finitely many charts go through with this replacement.
The condition \(D\psi(\theta)\ne0\) ensures that the target constants are
bounded away from zero on \(K\).
If necessary, shrink the chart balls so that their inverse moment maps take
values in the neighborhood on which \(\psi\) is \(C^1\).
Within each chart, replace the output \(F_j(Y_{j,n})\) by
\(\psi\{F_j(Y_{j,n})\}\) and use the uniform first-order expansion of
\(\psi\) over the compact image of the chart.
The preliminary localization, allocation of the sample and privacy budget,
Gaussian calibration, and adaptive composition require no change.
For the lower bound, apply \cref{appendix-thm:differentiable-target-local} at each
\(\theta_0\in K^\circ\), then use continuity and
\(K=\overline{K^\circ}\) as in the proof of
\cref{thm:compact-minimax}.
\end{proof}

\subsection{Proof of \Cref{prop:unbounded-impossibility}}
\label{appendix-subsec:unbounded-param-impossibility}

Fix \(\theta,\vartheta\in\Theta\), and write
\[
  d^2=\norm{\theta-\vartheta}_W^2.
\]
Group privacy for zCDP~\citep[Proposition~5.3]{bun2016_ConcentratedDifferential}
and joint convexity of the Kullback--Leibler divergence give
\[
  D_{\mathrm{KL}}(Q_\theta^M\|Q_\vartheta^M)
  \le n^2\rho.
\]
Indeed, group privacy gives this bound for every pair of deterministic data sets in \(\caX^n\), and joint convexity preserves it after mixing under \(P_\theta^{\otimes n}\) and \(P_\vartheta^{\otimes n}\).
The two-point testing reduction and the Bretagnolle--Huber inequality therefore yield
\[
  \max_{\eta\in\{\theta,\vartheta\}}
  \E_\eta
  \norm{\widehat{\theta}-\eta}_W^2
  \ge
  \frac{d^2}{16}e^{-n^2\rho}.
\]
For fixed \(n\) and finite \(\rho\), the exponential factor is positive.
Letting \(d^2\to\infty\) proves the claim.
   \clearpage
  \section{Reductions for Various Model Classes}
\label{appendix-sec:model-class-reductions}

This section develops reusable reductions for several general model classes.
\Cref{appendix-sec:concrete-model-constants} then applies these reductions to calculate the constants in concrete models.

\subsection{Finite alphabet model}
\label{appendix-subsec:fin-alphabet-sdp}

For a finite experiment, the unified variational problem admits an exact semidefinite formulation throughout the privacy and transition regimes.
Besides making the constant computable, the formulation yields numerical certificates through primal--dual semidefinite bounds.

At \(\theta_0\), write
\[
  \pi
  =
  \bigl(P_{\theta_0}\{x_1\},\ldots,P_{\theta_0}\{x_m\}\bigr)^\top,
  \qquad
  S_{i\cdot}=s_{\theta_0}(x_i)^\top.
\]
Also write
\[
  B=\operatorname{diag}(\pi)S,
  \qquad
  V_\pi=\operatorname{diag}(\pi)-\pi\pi^\top.
\]
Let \(\Phi_i=\Phi(x_i)\) take values in a real Hilbert space.
Translate the points so that \(\sum_i \pi_i \Phi_i=0\); this changes neither their diameter nor \(A_{\Phi,\theta_0}\).
Their Gram matrix \(G_{ij}=\ang{\Phi_i,\Phi_j}\) satisfies
\[
  G\succeq0,
  \qquad
  G\pi=0,
  \qquad
  \norm{\Phi_i-\Phi_j}_2^2
  =G_{ii}+G_{jj}-2G_{ij}
  \le1.
\]
Conversely, every matrix satisfying these constraints is the Gram matrix of centered points in \(\R^m\) having diameter at most one.
Indeed, after choosing a Gram factorization \(G=FF^\top\), the identity \(G\pi=0\) implies \(F^\top \pi=0\).

If the rows of \(F\) are \(\Phi_i^\top\), then
\[
  A_{\Phi,\theta_0}=F^\top B,
  \qquad
  J_{\Phi,\theta_0}=B^\top FF^\top B=B^\top G B.
\]
Thus the matrices \(B^\top G B\) generated by the first two lines of constraints are exactly \(\caJ_{\theta_0}\).
No closure is lost: centering gives
\[
  \Phi_i=\sum_j \pi_j(\Phi_i-\Phi_j),
  \qquad
  \norm{\Phi_i}_2 \le1,
\]
so the feasible Gram matrices form a bounded closed set.

\begin{proposition}[Finite alphabet unified SDP reduction]
  \label{appendix-prop:fin-alphabet-unified-sdp}
  Suppose \(I(\theta_0)\succ0\) and \(W\succ0\).
  For \(0\le\beta<1\), set \(c_\beta=(1-\beta)/2\), and let \(V_\pi^{1/2}\) be the symmetric positive-semidefinite square root of \(V_\pi\).
  Then
  \begin{equation}
    \label{appendix-eq:fin-alphabet-unified-sdp}
    \begin{aligned}
      \uCt(W)
      ={}&
      \min_{G,T\in\bbS_+^m,\,Z\in\bbS_+^p} \Tr Z \\
      \text{subject to }{}&
      G\pi=0,\quad
      G_{ii}+G_{jj}-2G_{ij}\le1\quad\text{for all }i,j,\\
      &
      \begin{pmatrix}
        c_\beta^{-1}G-T&\sqrt{\beta}\,c_\beta^{-1}GV_\pi^{1/2}\\
        \sqrt{\beta}\,c_\beta^{-1}V_\pi^{1/2}G&I_m+\beta c_\beta^{-1}V_\pi^{1/2}GV_\pi^{1/2}
      \end{pmatrix}
      \succeq0,\\
      &
      \begin{pmatrix}
        B^\top T B&W^{1/2}\\
        W^{1/2}&Z
      \end{pmatrix}
      \succeq0.
    \end{aligned}
  \end{equation}
  The minimum is attained.
  At \(\beta=0\), after eliminating \(T\), \cref{appendix-eq:fin-alphabet-unified-sdp} reduces to \cref{eq:fin-alphabet-sdp}.
\end{proposition}

\begin{proof}
  Fix a feasible \(G\), choose a Gram factorization \(G=FF^\top\) with \(F\in\R^{m\times d}\), and let the rows of \(F\) be the values of a centered diameter-one statistic.
  Conversely, every such statistic produces a feasible \(G\), and the compression statement in \cref{appendix-prop:unified-region-geometry}\textup{(ii)} shows that this parameterization loses no unified information matrices.
  For \(c_\beta>0\), the Woodbury identity gives
  \[
    \begin{aligned}
      &F\xk{\beta F^\top V_\pi F+c_\beta I_d}^{-1}F^\top\\
      &\qquad=
      c_\beta^{-1}G
      -\beta c_\beta^{-2}GV_\pi^{1/2}
      \xk{I_m+\beta c_\beta^{-1}V_\pi^{1/2}GV_\pi^{1/2}}^{-1}
      V_\pi^{1/2}G
      \eqqcolon T_\beta(G).
    \end{aligned}
  \]
  The first block constraint in \cref{appendix-eq:fin-alphabet-unified-sdp} is equivalent to \(T\preceq T_\beta(G)\).
  The map \(G\mapsto T_\beta(G)\) is continuous on the compact feasible Gram set, so the closure in the definition of the unified information region adds no further matrices here.
  The matrix \(B^\top T_\beta(G)B\) is precisely the unified information matrix generated by the statistic represented by \(F\).
  Because \(W\succ0\), the final block constraint requires \(B^\top T B\succ0\), and its Schur complement gives \(Z\succeq W^{1/2}(B^\top T B)^{-1}W^{1/2}\).
  The objective is decreasing in \(B^\top T B\), so it is minimized at \(T=T_\beta(G)\).
  Optimizing over the equivalent Gram and statistic representations proves the claimed equality.
  Compactness of the feasible Gram set and divergence at singular information matrices give attainment.
  When \(\beta=0\), the first block constraint reduces to \(T\preceq2G\), and the same monotonicity gives \(T=2G\).
  Writing \(Z_0=2Z\), a congruence transformation turns the final block constraint into the one in \cref{eq:fin-alphabet-sdp}, while \(\Tr Z=\Tr Z_0/2\).
\end{proof}

\subsection{Fixed statistic decomposition}
\label{appendix-subsec:trans-reductions}

For a fixed statistic, the sampling and privacy contributions separate.
The resulting lower bound is additive, but it is sharp only when the privacy and sampling regime problems admit a common optimizing sequence.
The reductions in this section identify this joint optimizer for binary, spherically symmetric, saturated categorical, and fixed-radius scores.
They yield both elementary closed forms and constants determined by a unique scalar clipping threshold.

\begin{proposition}[Fixed statistic decomposition]
  \label{appendix-prop:trans-stat-decomp}
  Under \cref{ass:regular}, fix \(\theta_0 \in\Theta\) with
  \(I(\theta_0)\succ0\), let \(W\succ0\), and let
  \(\beta\in(0,1)\).
  For \(\Phi:\caX\to\R^p\), abbreviate
  \[
    A_\Phi=A_{\Phi,\theta_0},
    \qquad
    V_\Phi=V_{\Phi,\theta_0}.
  \]
  Then
  \begin{equation}
    \label{appendix-eq:trans-fixed-stat-decomp}
    \uCt(W)=\inf_\Phi\Biggl(\beta\Tr\xk{W A_\Phi^{-1}V_\Phi A_\Phi^{-\top}}+\frac{1-\beta}{2}\Tr\xk{W(A_\Phi^\top A_\Phi)^{-1}}\Biggr),
  \end{equation}
  where the infimum is over diameter-one \(\Phi:\caX\to\R^p\) with
  nonsingular \(A_\Phi\).
  Consequently,
  \begin{equation}
    \label{appendix-eq:trans-regime-lower-bound}
    \uCt(W)
    \ge
    \beta\Tr\xk{WI(\theta_0)^{-1}}
    +(1-\beta)\mfC_{\theta_0}(W).
  \end{equation}
  Equality holds in \cref{appendix-eq:trans-regime-lower-bound} if and only if there is a sequence of admissible statistics that is asymptotically efficient in both the sampling and privacy regimes under the loss \(W\).
  In particular, equality holds if one statistic attains both regime problems.
\end{proposition}

\begin{proof}
  For a finite-dimensional diameter-one statistic, its unified information is
  \[
    H_{\Phi,\theta_0,\beta}
    =
    A_{\Phi,\theta_0}^\top
    \xk{
      \beta V_{\Phi,\theta_0}+\frac{1-\beta}{2}I
    }^{-1}
    A_{\Phi,\theta_0}.
  \]
  Compactness and lower semicontinuity give an optimizing matrix
  \(H_*\in\caU_{\theta_0,\beta}\).
  The region contains a positive-definite matrix, whereas \(W\succ0\) makes
  the objective infinite at every singular matrix, so \(H_*\succ0\).
  Choose generators \(H_{\Phi_m,\theta_0,\beta}\to H_*\).
  They are eventually positive definite, and continuity of matrix inversion
  gives
  \(\caI_W(H_{\Phi_m,\theta_0,\beta})\to\caI_W(H_*)\).
  The compression statement in
  \cref{appendix-prop:unified-region-geometry}\textup{(ii)} replaces each
  generator by an \(\R^p\)-valued statistic with no larger diameter and the
  same information matrix.
  Its square derivative \(A_{\Phi_m}\) is nonsingular.
  This proves that the infimum may be restricted as in
  \cref{appendix-eq:trans-fixed-stat-decomp}.

  For such a statistic, \(V_\Phi \succ0\): indeed,
  \(\ker(V_\Phi)\subset\ker(A_\Phi^\top)\), whereas \(A_\Phi\) is
  nonsingular.
  Direct inversion gives
  \[
    H_{\Phi,\theta_0,\beta}^{-1}
    =
    A_\Phi^{-1}
    \xk{\beta V_\Phi+\frac{1-\beta}{2}I_p}
    A_\Phi^{-\top}.
  \]
  Taking the \(W\)-weighted trace proves
  \cref{appendix-eq:trans-fixed-stat-decomp}.

  The information inequality
  \[
    A_\Phi^\top V_\Phi^{-1} A_\Phi
    \preceq
    I(\theta_0)
  \]
  implies
  \[
    A_\Phi^{-1} V_\Phi A_\Phi^{-\top}
    \succeq
    I(\theta_0)^{-1}.
  \]
  Moreover, \(A_\Phi^\top A_\Phi \in\caJ_{\theta_0}\), and hence
  \[
    \frac{1}{2}
    \Tr\xk{W(A_\Phi^\top A_\Phi)^{-1}}
    \ge
    \mfC_{\theta_0}(W).
  \]
  Weighting the two inequalities by \(\beta\) and \(1-\beta\), respectively,
  proves
  \cref{appendix-eq:trans-regime-lower-bound}.
  The gap between the two sides is the sum of the \(\beta\)-weighted
  nonnegative Fisher gap and the \((1-\beta)\)-weighted nonnegative privacy gap.
  It vanishes along an admissible sequence exactly when both gaps vanish
  along that sequence.
\end{proof}

\subsection{General one-dimensional formula}
\label{appendix-subsec:trans-1d}

At \(\beta=1\), a \(0/0\) ratio in
\cref{eq:1d-unified-info} is interpreted as zero.

\begin{proof}[Proof of the unified formula in \cref{prop:1d-matching}]
  Apply \cref{appendix-lem:unified-directional-envelope} with \(p=1\) and
  \(h=1\).
  Since \(\caU_{\theta_0,\beta}\) is a compact downward-closed subset of
  \(\R_+\), its upper endpoint is exactly the supremum on the right-hand
  side of \cref{eq:1d-unified-info}.
  In the local parameterization \(\eta=\mu(\theta)\), the score is
  \[
    s_\eta(X)
    =
    \frac{s_\theta(X)}{\mu'(\theta)}.
  \]
  Hence the upper endpoint of the one-dimensional unified information region
  is divided by \(\xk{\mu'(\theta_0)}^2\), and its inverse is multiplied by
  that factor.
  Applying \cref{thm:local-minimax} in the \(\eta\)-parameterization
  proves \cref{eq:1d-unified-constant}.
  The identities for the privacy and sampling regimes follow from \cref{appendix-eq:unified-directional-privacy-regime,appendix-eq:unified-directional-sampling-regime}.
\end{proof}

\subsection{Finite-support score classes}
\label{appendix-subsec:finite-support-score-classes}

These reductions cover score classes in which a common statistic or configuration is efficient across the sampling and privacy regimes.

\begin{lemma}[Linear configuration reduction]
  \label{appendix-lem:trans-config-linear}
  Let \(V_0\succ0\), \(\beta\in(0,1)\), and \(W\succeq0\).
  Suppose a nonempty class \(\mathcal A\) of statistics of diameter at most one is
  parameterized by maps \(U:\R^p\to\caH\) of full column rank such that
  \[
    A_\Phi=U, \qquad V_\Phi=UV_0U^*.
  \]
  Put \(J_U=U^*U\).
  Each statistic in this class satisfies
  \[
    H_{\Phi,\theta_0,\beta}^{-1}
    =\beta V_0+\frac{1-\beta}{2}J_U^{-1}.
  \]
  Define the class-restricted constants
  \[
    C_{\mathcal A}^{(\beta)}(W)
    =\inf_{\Phi\in\mathcal A}\Tr(W H_{\Phi,\theta_0,\beta}^{-1}),
    \qquad
    C_{\mathcal A}^{\mathrm{priv}}(W)
    =\frac12\inf_{\Phi\in\mathcal A}\Tr(WJ_U^{-1}).
  \]
  Then
  \[
    C_{\mathcal A}^{(\beta)}(W)
    =\beta\Tr(WV_0)+(1-\beta)C_{\mathcal A}^{\mathrm{priv}}(W).
  \]
  Within \(\mathcal A\), the configurations attaining the transition infimum
  are exactly those attaining the privacy infimum.
\end{lemma}

\begin{proof}
The Woodbury identity gives
\[
  U^*\xk{\beta UV_0 U^*+\frac{1-\beta}{2}I}^{-1}U=\xk{\beta V_0+\frac{1-\beta}{2}(U^*U)^{-1}}^{-1}.
\]
Taking the weighted trace and then the infimum over the same class
\(\mathcal A\) gives the identity for the class-restricted constants.
Since \(1-\beta>0\), the two objectives have the same minimizers in that class.
\end{proof}

\begin{proposition}[Regular simplex scores]
  \label{appendix-prop:simplex-score-reduction}
Let \(\caT=\{v\in\R^k:\mathbf{1}_k^\top v=0\}\), let \(H:\R^{k-1}\to\caT\) be an isometry, and write \(h_j=H^\top e_j\).  If the score takes the values \(S_j=kh_j\), each with probability \(1/k\), then
\[
  \Gamma_{\theta_0}=\frac{k-1}{2}, \qquad \frac{1}{2}I_{k-1}\in\caJ_{\theta_0}.
\]
The latter matrix is generated by the regular-simplex statistic \(H^\top(e_j-k^{-1}\mathbf{1}_k)/\sqrt2\).  Under \cref{ass:regular,ass:score-density-l1}, its exact privacy constant for tangent-coordinate squared Euclidean loss is \(k-1\).
\end{proposition}

\begin{proof}
Centering diameter-one values \(\Phi_j\) gives
\[
  \norm{A_{\Phi,\theta_0}}_{\HS}^2=\frac{1}{k}\sum_{j<\ell}\norm{\Phi_j-\Phi_\ell}^2\le\frac{k-1}{2}.
\]
The regular simplex attains equality, has the information matrix, and has nonsingular derivative.  The trace lower bound and Gaussian release complete the proof.
\end{proof}

\subsection{Scalar score classes}
\label{appendix-subsec:scalar-score-classes}

We begin with the privacy endpoint and then give its unified extension.

\begin{lemma}[Scalar score reduction]
  \label{appendix-lem:scalar-score-reduction}
Let \(S=s_{\theta_0}(X)\) be a scalar score, and set
\[
  a_{\theta_0}=\frac{1}{2}\E_{\theta_0} \abs{S}.
\]
If \(0<a_{\theta_0}<\infty\), then
\[
  \caJ_{\theta_0}=[0,a_{\theta_0}^2].
\]
The upper endpoint is attained by the scalar diameter-one statistic
\(\Phi_*(X)=\frac{1}{2}\sign(S)\), with \(\sign(0)=0\).
\end{lemma}

\begin{proof}
For a Hilbert-valued diameter-one statistic \(\Phi\),
\[
  \norm{A_{\Phi,\theta_0}}_{\HS}
  =
  \sup_{\norm{v}\le1}
  \abs{\E_{\theta_0} \left(\langle v,\Phi\rangle S\right)}.
\]
Every scalar projection has diameter at most one.  Since \(\E_{\theta_0}S=0\), it may be translated without changing its covariance with \(S\), and hence the expression is bounded by \(\E_{\theta_0}\abs{S}/2\).
\(\Phi_*\) attains equality, and scaling it by factors in \([0,1]\) fills the interval.
\end{proof}

Consequently, under the regularity conditions of
\cref{prop:1d-matching}, a smooth scalar target
\(\mu(\theta)\) has privacy constant
\begin{equation}
  \label{appendix-eq:scalar-score-reduction}
  C_{\theta_0}(\mu)
  =
  \frac{\xk{\mu'(\theta_0)}^2}{2a_{\theta_0}^2}.
\end{equation}

\subsubsection{Unified clipping reduction}
\label{appendix-subsubsec:trans-clipping-reductions}

The next proposition reduces every scalar transition calculation to two
clipping thresholds.  Symmetric scores collapse to one threshold, while
binary scores collapse to a single fixed contrast.

\begin{proposition}[Scalar clipped affine reduction]
  \label{appendix-prop:trans-scalar-clipping}
Under \cref{ass:regular}, let \(S=s_{\theta_0}(X)\) be a scalar score
with \(0<\E_{\theta_0} S^2<\infty\), let \(\beta\in(0,1)\), and let
\(\mu\) be a differentiable scalar target with \(\mu'(\theta_0)\ne0\).
For \(a<b\), write
\[
  G_{a,b}(s)
  =
  \min\{(s-a)_+,b-a\}.
\]
Then
\begin{equation}
  \label{appendix-eq:trans-scalar-clipping}
  \uC_{\mu,\theta_0}
  =
  \xk{\mu'(\theta_0)}^2
  \inf_{a<b}
  \frac{
    \beta\Var_{\theta_0}\xk{G_{a,b}(S)}+(1-\beta)(b-a)^2/2
  }{
    \bigl[\E_{\theta_0}\zk{G_{a,b}(S)S}\bigr]^2
  },
\end{equation}
where pairs with zero denominator are omitted.  Every minimizing pair
generates an attaining diameter-one statistic
\[
  \Phi_{a,b}(X)
  =
  \frac{G_{a,b}\xk{S(X)}}{b-a}.
\]
\end{proposition}

\begin{proof}
Applying the one-dimensional unified formula gives
\[
  \frac{\uC_{\mu,\theta_0}}{\xk{\mu'(\theta_0)}^2}
  =
  \inf_{\diam(\phi)\le1}
  \frac{\beta\Var\xk{\phi(X)}+(1-\beta)/2}
  {[\E\zk{\phi(X)S}]^2}.
\]
Replacing \(\phi\) by \(\E\zk{\phi\mid S}\) preserves its score
covariance, weakly decreases its variance, and does not enlarge its diameter.
After changing sign if necessary, we may take
\(\E\zk{\phi S}>0\).  Any nonconstant admissible statistic can be rescaled
to have diameter one, which weakly improves the quotient because
\(0<\beta<1\).  Thus, on setting
\[
  \psi
  =
  \frac{\phi-E\phi}{\E(\phi S)},
\]
the infimum is equivalently
\begin{equation}
  \label{appendix-eq:trans-scalar-influence}
  \inf_{\substack{\E\psi=0\\ \E(\psi S)=1}}
  \xk{
    \beta\E\psi^2+\frac{1-\beta}{2}\diam(\psi)^2
  }.
\end{equation}

Fix the essential lower and upper endpoints of \(\psi\).  With these
endpoints fixed, the diameter term in
\cref{appendix-eq:trans-scalar-influence} is constant, and the
minimum-norm element satisfying the two moment constraints is the
pointwise projection of an affine function of \(S\) onto that interval.
Equivalently, by the Hilbert-space KKT conditions, it has the form
\[
  \psi(S)
  =
  c_0+\gamma G_{a,b}(S)
\]
for some \(a<b\) and \(\gamma>0\).  The sign of \(\gamma\) is positive
because \(\E\zk{\psi(S)S}=1\).  The calibration condition gives
\(\gamma=[\E\zk{G_{a,b}(S)S}]^{-1}\), and substituting this expression into
the objective in \cref{appendix-eq:trans-scalar-influence} gives exactly
the quotient in \cref{appendix-eq:trans-scalar-clipping}.  Conversely,
each clipped affine statistic is admissible, proving the equality.
\end{proof}

The next lemma is the two-point specialization of
\cref{appendix-prop:trans-scalar-clipping}.

\begin{lemma}[Binary score reduction]
  \label{appendix-lem:trans-binary-score}
  Under \cref{ass:regular}, let \(S=s_{\theta_0}(X)\) be a scalar score,
  and suppose that, for a binary statistic \(B\),
  \[
    \Pr_{\theta_0}(B=1)=\pi\in(0,1),
    \qquad
    S=s_1 B+s_0(1-B).
  \]
  Put
  \[
    a=\E_{\theta_0}(BS)
    =
    \pi(1-\pi)(s_1-s_0),
  \]
  and assume \(a\ne0\).
  If \(\mu(\theta)\) is a scalar target with
  \(\mu'(\theta_0)\ne0\), then
  \begin{equation}
    \label{appendix-eq:trans-binary-score}
    \uC_{\mu,\theta_0}
    =
    \xk{\mu'(\theta_0)}^2
    \frac{
      \beta\pi(1-\pi)+(1-\beta)/2
    }{a^2}.
  \end{equation}
  The statistic \(B\) attains the constant.
\end{lemma}

\begin{proof}
  The directional envelope in
  \cref{appendix-lem:unified-directional-envelope} first reduces the unified
  scalar information problem to scalar statistics.
  Conditioning an arbitrary scalar statistic on \(B\) preserves its
  covariance with the score, weakly decreases its variance, and does not
  enlarge its diameter.
  A function of \(B\) has the form \(c+\delta B\), where
  \(\abs{\delta}\le1\).
  Since \(\E_{\theta_0} S=0\), its score covariance and variance are
  respectively \(\delta a\) and \(\delta^2 \pi(1-\pi)\).
  Its target objective is
  \[
    \xk{\mu'(\theta_0)}^2
    \frac{
      \beta\delta^2 \pi(1-\pi)+(1-\beta)/2
    }{\delta^2 a^2}.
  \]
  This is minimized at \(\abs{\delta}=1\), which proves the result.
\end{proof}

\subsection{Symmetric and spherical score classes}
\label{appendix-subsec:symmetric-spherical-score-classes}

\subsubsection{Fixed-radius scores}
\label{appendix-subsubsec:trans-fixed-radius}

\begin{proposition}[Centrally symmetric scores of fixed radius]
  \label{appendix-prop:fixed-radius-score-reduction}
Suppose
\[
  S\overset{d}{=}-S, \qquad \E_{\theta_0}(SS^\top)=\kappa I_p, \qquad \norm{S}_2^2=p\kappa\quad\text{almost surely},
\]
where \(\kappa>0\).  Then
\[
  \Gamma_{\theta_0}=\frac{\kappa}{4}, \qquad \frac{\kappa}{4p}I_p\in\caJ_{\theta_0},
\]
and the latter matrix is generated by \(S/(2\sqrt{p\kappa})\).
Under \cref{ass:regular,ass:score-density-l1}, the exact privacy constant for the identity target under squared Euclidean loss is \(2p^2/\kappa\).
\end{proposition}

\begin{proof}
For every diameter-one statistic \(\Phi\), the support-function argument and Jensen's inequality give
\[
  \norm{A_{\Phi,\theta_0}}_{\HS}\le \frac{1}{2}\sup_{\norm{B}_{\HS}\le1}\E\norm{BS}_2\le \frac{\sqrt\kappa}{2}.
\]
The normalized score statistic has diameter one and derivative \(\sqrt\kappa I_p/(2\sqrt p)\), so it attains the bound.  The trace lower bound and its Gaussian release give the stated constant.
\end{proof}

Under \cref{ass:regular,ass:score-density-l1}, suppose
\[
  S\overset{d}{=}-S, \qquad \E_{\theta_0}(SS^\top)=\kappa I_p, \qquad \norm{S}_2^2=p\kappa\quad\text{almost surely},
\]
where \(\kappa>0\).
The fixed statistic
\[
  \Phi_*(X)=\frac{S}{2\sqrt{p\kappa}}
\]
has
\[
  A_*=\frac{\sqrt\kappa}{2\sqrt p}I_p, \qquad V_*=\frac{1}{4p}I_p.
\]
Redefine \(S\) and \(\Phi_*\) on a \(P_{\theta_0}\)-null set, if necessary,
so that the diameter bound holds pointwise.
By \cref{appendix-prop:fixed-radius-score-reduction}, \(\Phi_*\) is privacy
optimal under squared Euclidean loss.
Moreover,
\[
  A_*^\top V_*^{-1}A_*=\kappa I_p=I(\theta_0),
\]
so it is Fisher-efficient as well.
The fixed statistic decomposition therefore yields
\begin{equation}
  \label{appendix-eq:trans-fixed-radius-constant}
  \uCt(I_p)=\frac{\beta p+2(1-\beta)p^2}{\kappa}.
\end{equation}
Thus one optimal statistic is independent of \(\beta\).
This argument requires only central symmetry, isotropic score covariance, and
constant radius; it does not require a spherically symmetric score law.

\subsubsection{Spherical clipping}

The binary lemma does not cover symmetric scores with more than two
magnitudes.
For those models, and for spherically symmetric vector scores, the optimizer
is a clipped score whose threshold balances sampling variance against
sensitivity.

\begin{proposition}[Spherical clipped score reduction]
  \label{appendix-prop:trans-spherical-clipping}
  Under \cref{ass:regular}, suppose that the score
  \(S=s_{\theta_0}(X)\in\R^p\) is spherically symmetric and
  \[
    0<\E_{\theta_0} \norm{S}_2^2<\infty.
  \]
  Let \(R=\norm{S}_2\).
  For every \(\beta\in(0,1)\), there is a unique
  \(t_{\theta_0,\beta}>0\) satisfying
  \begin{equation}
    \label{appendix-eq:trans-spherical-threshold}
    \E_{\theta_0}
    \xk{
      \frac{R}{t_{\theta_0,\beta}}-1
    }_+
    =
    \frac{2p(1-\beta)}{\beta}.
  \end{equation}
  For the identity target under squared Euclidean loss,
  \begin{equation}
    \label{appendix-eq:trans-spherical-constant}
    \uCt(I_p)
    =
    \frac{\beta p^2}{
      \E_{\theta_0}
      \zk{R\min(R,t_{\theta_0,\beta})}
    }.
  \end{equation}
  The diameter-one statistic
  \begin{equation}
    \label{appendix-eq:trans-spherical-stat}
    \Phi_{\theta_0,\beta}(X)
    =
    \frac{1}{2}
    \frac{S}{R}
    \min\xk{\frac{R}{t_{\theta_0,\beta}},1},
    \qquad
    \Phi_{\theta_0,\beta}=0\ \text{on }\{R=0\},
  \end{equation}
  attains the constant.

  When \(p=1\), the same result applies to a differentiable scalar target
  \(\mu(\theta)\) with \(\mu'(\theta_0)\ne0\), with the right side of
  \cref{appendix-eq:trans-spherical-constant} multiplied by
  \(\xk{\mu'(\theta_0)}^2\).
\end{proposition}

\begin{proof}
  We first rewrite the Euclidean fixed statistic problem in influence
  coordinates.
  For an admissible statistic in
  \cref{appendix-eq:trans-fixed-stat-decomp}, center it and put
  \[
    \psi=A_\Phi^{-1} \Phi,
    \qquad
    Q=(A_\Phi^\top A_\Phi)^{-1}.
  \]
  Conversely, every centered pair \((\psi,Q)\), with \(Q\succ0\), satisfying
  \begin{equation}
    \label{appendix-eq:trans-influence-ellipsoid}
    \E_{\theta_0}(\psi S^\top)=I_p,
    \qquad
    \xk{\psi(x)-\psi(x')}^\top
    Q^{-1}
    \xk{\psi(x)-\psi(x')}
    \le1
  \end{equation}
  is realized by \(\Phi=Q^{-1/2} \psi\), where the principal symmetric square
  root is used.
  Indeed, \(A_\Phi=Q^{-1/2}\), and the two terms in the unified objective
  become \(\beta\E\norm{\psi}_2^2\) and \((1-\beta)\Tr(Q)/2\), respectively.
  Hence
  \begin{equation}
    \label{appendix-eq:trans-influence-quotient}
    \uCt(I_p)
    =
    \inf_{(\psi,Q)}
    \xk{
      \beta\E_{\theta_0} \norm{\psi}\!_2^2
      +\frac{1-\beta}{2}\Tr Q
    },
  \end{equation}
  subject to \cref{appendix-eq:trans-influence-ellipsoid}.

  Choose a conditional expectation version of \(\psi\) given \(S\) that takes
  values in the closed convex hull of the original range, including on null
  score values.
  This reduction preserves the calibration, weakly decreases the quadratic
  term, and preserves the ellipsoid constraint because a convex hull has the
  same diameter in every norm.
  Let \(\nu\) be Haar probability measure on the orthogonal group, and define
  \[
    \overline{\psi}(s)
    =
    \int G^\top \psi(Gs)\,d\nu(G),
    \qquad
    \overline{Q}
    =
    \int G^\top QG\,d\nu(G)
    =
    \frac{\Tr Q}{p}I_p.
  \]
  Spherical symmetry preserves the calibration, Jensen's inequality weakly
  decreases the quadratic term, and \(\Tr\overline{Q}=\Tr Q\).
  Feasibility is also preserved: the ellipsoid inequality is equivalent to
  positivity of
  \[
    \begin{pmatrix}
      Q & \psi(s)-\psi(s')\\
      \xk{\psi(s)-\psi(s')}^\top & 1
    \end{pmatrix},
  \]
  Applying this block inequality at \((Gs,Gs')\), conjugating by
  \(\operatorname{diag}(G^\top,1)\), and integrating gives the corresponding
  matrix for \((\overline{\psi},\overline{Q})\).

  We may therefore restrict to
  \[
    \psi(S)=h(R)\frac{S}{R},
    \qquad
    Q=qI_p.
  \]
  Orthogonal equivariance makes \(\psi\) odd.
  If \(M=\esssup\abs{h(R)}\), the ellipsoid constraint is equivalent, after
  modification on a null set, to \(q\ge4M^2\).
  The objective is increasing in \(q\), so take \(q=4M^2\).
  The calibration becomes \(\E\zk{Rh(R)}=p\).
  Writing \(h=Mf\), where \(\abs{f}\le1\), and then eliminating \(M\), gives
  \begin{equation}
    \label{appendix-eq:trans-radial-quotient}
    \uCt(I_p)
    =
    p^2
    \inf_{\substack{\abs{f}\le1\\\E\zk{Rf(R)}>0}}
    \frac{
      \beta\E f(R)^2+2p(1-\beta)
    }{
      \xk{\E(Rf(R))}^2
    }.
  \end{equation}

  Put \(c=2p(1-\beta)/\beta\).
  The quadratic variational identity gives
  \[
    \sup_{\substack{\abs{f}\le1\\\E(Rf)>0}}
    \frac{\xk{\E(Rf)}^2}{\E f^2+c}
    =
    \sup_{a>0,\ \abs{f}\le1}
    \xk{
      2a\E(Rf)-a^2(\E f^2+c)
    }.
  \]
  For a fixed \(a=t\), the pointwise optimizer is
  \(f_t(r)=\min(r/t,1)\).
  The resulting profile is
  \[
    Q(t)
    =
    \E\left[
      R^2 \ind{R\le t}
      +(2tR-t^2)\ind{R>t}
    \right]
    -t^2 c.
  \]
  Put
  \[
    g(t)=\E\xk{\frac{R}{t}-1}_+.
  \]
  The function \(g\) is continuous and strictly decreasing on
  \((0,\esssup R)\), with limits \(\infty\) and \(0\) at the two ends.
  Moreover,
  \(Q'(t)=2t\xk{g(t)-c}\).
  Thus \cref{appendix-eq:trans-spherical-threshold} has a unique solution,
  and that solution is the global maximizer of the reciprocal quotient.
  At that solution,
  \[
    \E f_t(R)^2+c
    =
    \frac{1}{t}\E\zk{Rf_t(R)}.
  \]
  The maximal reciprocal quotient in
  \cref{appendix-eq:trans-radial-quotient} is consequently
  \[
    t\E\zk{Rf_t(R)}
    =
    \E\zk{R\min(R,t)}.
  \]
  This proves \cref{appendix-eq:trans-spherical-constant}.
  Substitution into the influence representation gives
  \cref{appendix-eq:trans-spherical-stat}.
  The final claim for scalar targets follows by local reparameterization, which is
  valid because \(\mu'(\theta_0)\ne0\).
\end{proof}

\begin{corollary}[Unified factorized spherical scores]
  \label{appendix-cor:trans-factorized-spherical-score}
Under \cref{ass:regular}, suppose that, at the local center,
\[
  s_{\beta_0}(X,Y)=AX,
  \qquad
  X=RU,
\]
where \(U\) is uniform on the unit sphere and independent of \((A,R)\), and
\[
  0<\E(A^2R^2)<\infty,
  \qquad
  \E(\abs{A}R)>0.
\]
Put
\[
  R_S=\abs{A} R.
\]
For every \(\beta\in(0,1)\), let \(t_\beta>0\) solve
\[
  \E\xk{\frac{R_S}{t_\beta}-1}_+
  =
  \frac{2p(1-\beta)}{\beta}.
\]
Then the exact unified constant under squared Euclidean loss is
\[
  \uCt(I_p)
  =
  \frac{\beta p^2}{\E\zk{R_S \min(R_S,t_\beta)}}.
\]
The attaining statistic is the radial clipping in
\cref{appendix-eq:trans-spherical-stat}, with score \(S=AX\).
Its constant in the privacy regime depends only on \(\E R_S\), whereas its interior transition constants depend on the full law of \(R_S\).
\end{corollary}

\begin{proof}
Let \(\varepsilon(A)=1\) when \(A\ge0\) and \(-1\) otherwise.
Conditional on \((A,R)\), the vector \(\varepsilon(A)U\) is uniform on
the sphere.  Thus it is independent of \(R_S=\abs{A}R\), and
\(S=R_S\varepsilon(A)U\) is spherical.
Apply \cref{appendix-prop:trans-spherical-clipping}.
\end{proof}

For \(p=1\), \cref{appendix-prop:trans-spherical-clipping} says that a
symmetric nondegenerate scalar score has the exact constant
\begin{equation}
  \label{appendix-eq:trans-symmetric-score-constant}
  \uC_{\mu,\theta_0}
  =
  \frac{\beta\xk{\mu'(\theta_0)}^2}{
    \E_{\theta_0} \min\{S^2,t_{\theta_0,\beta} \abs{S}\}
  },
\end{equation}
where
\(\E\zk{(\abs{S}/t_{\theta_0,\beta}-1)_+}=2(1-\beta)/\beta\).
Thus every finite symmetric score distribution gives a piecewise rational
constant, while a continuous symmetric score requires only a unique scalar
root.
   \clearpage
  \section{Exact Constants in Concrete Models}
\label{appendix-sec:concrete-model-constants}

This section gives exact constants for concrete models.
The table first collects the privacy-regime endpoints.
Building on the reductions in \cref{appendix-sec:model-class-reductions},
the subsections give model-specific efficiency constants.

\begingroup
\scriptsize
\renewcommand{\arraystretch}{1.8}
\begin{longtable}{
  p{0.44\textwidth}
  p{0.10\textwidth}
  p{0.33\textwidth}
}
  \caption{Exact constants in the privacy regime.}
  \label{appendix-tab:exact-constants}
  \\
  \hline
  Model & Target & Privacy constant \(C_{\mathrm{priv}}\) \\
  \hline
  \endfirsthead
  \hline
  Model & Target & Privacy constant \(C_{\mathrm{priv}}\) \\
  \hline
  \endhead
  \hline
  \endfoot
  One-dimensional regular model \(\{P_\theta\}\) &
  \(\mu(\theta)\) &
  \(\displaystyle
    \frac{\xk{\mu'(\theta)}^2}{2a_\theta^2}\),
  \(a_\theta=\frac{1}{2}\E_\theta \abs{s_\theta(X)}\) \\
  Bernoulli\((q)\), \(q\in(0,1)\) &
  \(q\) &
  \(\displaystyle \frac{1}{2}\) \\
  Gaussian location \(N(\theta,\sigma^2)\) &
  \(\theta\) &
  \(\displaystyle \pi\sigma^2\) \\
  Gaussian location--scale \(N(\mu,\sigma^2)\) &
  \(\mu+z_\alpha \sigma\) &
  \(\displaystyle
    \frac{2\sigma^2}{d_\alpha^2}\),
  \(d_\alpha\) as in
  \cref{appendix-eq:gauss-quantile-constant} \\
  Laplace location \(\theta\), scale \(b\) &
  \(\theta\) &
  \(\displaystyle 2b^2\) \\
  Logistic location \(\theta\), scale \(s\) &
  \(\theta\) &
  \(\displaystyle 8s^2\) \\
  Cauchy location \(\theta\), scale \(\gamma\) &
  \(\theta\) &
  \(\displaystyle \frac{\pi^2 \gamma^2}{2}\) \\
  Student \(t_\nu\) location \(\theta\), scale \(\sigma\) &
  \(\theta\) &
  \(\displaystyle
    \frac{\nu\pi\sigma^2}{2}
    \left(
      \frac{\Gamma(\nu/2)}
      {\Gamma((\nu+1)/2)}
    \right)^2\) \\
  Generalized Gaussian location
  \(\propto\exp(-(\abs{x-\theta}/b)^\alpha)\), \(\alpha>1/2\) &
  \(\theta\) &
  \(\displaystyle
    \frac{2b^2 \Gamma(1/\alpha)^2}{\alpha^2}\) \\
  Exponential rate \(\lambda\) &
  \(\lambda\) &
  \(\displaystyle \frac{e^2 \lambda^2}{2}\) \\
  Exponential mean \(\mu=1/\lambda\) &
  \(\mu\) &
  \(\displaystyle \frac{e^2 \mu^2}{2}\) \\
  Gamma\((\alpha,\beta)\), known shape \(\alpha\), scale \(\beta\) &
  \(\beta\) &
  \(\displaystyle
    \frac{\beta^2 e^{2\alpha} \Gamma(\alpha)^2}
    {2\alpha^{2\alpha}}\) \\
  Gamma\((\alpha,\beta)\), known shape \(\alpha\), mean \(\mu=\alpha\beta\) &
  \(\mu\) &
  \(\displaystyle
    \frac{\mu^2 e^{2\alpha} \Gamma(\alpha)^2}
    {2\alpha^{2\alpha}}\) \\
  Poisson mean \(\lambda\) &
  \(\lambda\) &
  \(\displaystyle
    \frac{1}
    {2\Pr_\lambda\{X=\lfloor\lambda\rfloor\}^2}\) \\
  Distributions \(P\) supported on \([\ell,r]\), \(\Delta=r-\ell\) &
  \(\E_P X\) &
  \(\displaystyle \frac{\Delta^2}{2}\) \\
  Distributions \(P\) supported on \([\ell,r]^p\) &
  \(\E_P X \in\R^p\) &
  \(\displaystyle \frac{p^2 \Delta^2}{2}\) \\
  Categorical model \(q\in\Delta_{k-1}\) on \(\{1,\ldots,k\}\) &
  \(q\) &
  \(\displaystyle k-1\) \\
  Spherical location \(X=\theta+Z\), \(\theta\in\R^p\) &
  \(\theta\) &
  \(\displaystyle
    \frac{2p^3}{m^2}\),
  \(m=\E_\theta \norm{s_\theta(X)}_2\) \\
  Gaussian mean \(N_p(\theta,\sigma^2 I_p)\) &
  \(\theta\) &
  \(\displaystyle
    \frac{2\sigma^2 p^3}{(\E\chi_p)^2}\) \\
  Radial Laplace location
  \(\propto\exp(-\norm{x-\theta}_2/b)\) &
  \(\theta\) &
  \(\displaystyle 2b^2 p^3\) \\
  Radial generalized Gaussian location
  \(\propto\exp(-(\norm{x-\theta}_2/b)^\alpha)\),
  \(\alpha>\max(0,1-p/2)\) &
  \(\theta\) &
  \(\displaystyle
    \frac{
      2p^3 b^2 \Gamma(p/\alpha)^2
    }{
      \alpha^2 \Gamma((p+\alpha-1)/\alpha)^2
    }\) \\
  Multivariate Student \(t_\nu\) location, scale \(\sigma\) &
  \(\theta\) &
  \(\displaystyle \frac{\nu\sigma^2 p^3}{2}\)
  \newline
  \(\displaystyle {}\times
    \left(
      \frac{\Gamma(p/2)\Gamma(\nu/2)}
      {\Gamma((p+1)/2)\Gamma((\nu+1)/2)}
    \right)^2\) \\
  Multivariate Cauchy location, scale \(\sigma\) &
  \(\theta\) &
  \(\displaystyle
    \frac{\pi\sigma^2 p^3}{2}
    \left(
      \frac{\Gamma(p/2)}
      {\Gamma((p+1)/2)}
    \right)^2\) \\
  Logistic regression with spherical design at \(\beta_0=0\) &
  \(\beta\) &
  \(\displaystyle
    \frac{8p^3}{(\E\norm{X}_2)^2}\) \\
  Canonical GLM at \(\beta_0=0\), \(X\sim N_p(0,I_p)\) &
  \(\beta\) &
  \(\displaystyle
    \frac{2p^3}{(\E\chi_p)^2d_0^2}\),
  \(d_0\) as in \cref{eq:glm-absolute-score} \\
  Logistic regression at \(\beta_0=0\), \(X\sim N_p(0,I_p)\) &
  \(\beta\) &
  \(\displaystyle
    \frac{8p^3}{(\E\chi_p)^2}\) \\
  Poisson log-linear regression at \(\beta_0=0\),
  \(X\sim N_p(0,I_p)\) &
  \(\beta\) &
  \(\displaystyle
    \frac{e^2 p^3}{2(\E\chi_p)^2}\) \\
  Linear regression with spherical design &
  \(\beta\) &
  \(\displaystyle
    \frac{2p^3}
    {(\E\abs{\ell_f(\xi)}\E\norm{X}_2)^2}\) \\
  Linear regression with \(X\sim N_p(0,I_p)\),
  \(\xi\sim N(0,\sigma^2)\) &
  \(\beta\) &
  \(\displaystyle
    \frac{\pi\sigma^2 p^3}{(\E\chi_p)^2}\) \\
\end{longtable}
\endgroup

\subsection{Location, scale, and Gamma models under privacy regime}
\label{appendix-subsec:1d-models}

For the models in this subsection, \cref{appendix-lem:scalar-score-reduction}
reduces the privacy constant to the calculation of \(a_{\theta_0}\) and the
derivative of the target.

For an absolutely continuous location family with density \(f(x-\theta)\), suppose \(f\) is symmetric and unimodal with mode zero and the regularity assumptions of the scalar theorem hold.
Then
\[
  a_\theta
  =
  \frac{1}{2}\int_{\R} \abs{f'(x)} dx
  =
  f(0),
  \qquad
  C_\theta(\theta)=\frac{1}{2f(0)^2}.
\]
For the Cauchy location family with scale \(\gamma\),
\[
  f(0)=\frac{1}{\pi\gamma}.
\]
For the Student \(t_\nu\) location family with scale \(\sigma\),
\[
  f(0)
  =
  \frac{\Gamma((\nu+1)/2)}
  {\sigma\sqrt{\nu\pi}\Gamma(\nu/2)}.
\]
For the generalized Gaussian location density proportional to
\(\exp(-(\abs{x-\theta}/b)^\alpha)\),
\[
  f(0)=\frac{\alpha}{2b\Gamma(1/\alpha)}.
\]
The restriction \(\alpha>1/2\) is precisely the condition under which this
location family has finite Fisher information and is differentiable in
quadratic mean at the cusp.

For the exponential and Gamma scale examples, let \(Y\sim\operatorname{Gamma}(\alpha,1)\).
Integration by parts gives
\[
  \E\abs{Y-\alpha}
  =
  \frac{2\alpha^\alpha e^{-\alpha}}{\Gamma(\alpha)}.
\]
If \(X\sim\operatorname{Gamma}(\alpha,\beta)\), with known shape \(\alpha\) and scale \(\beta\), then
\[
  a_\beta
  =
  \frac{\alpha^\alpha e^{-\alpha}}
  {\beta\Gamma(\alpha)}.
\]
Substitution into the scalar formula gives the Gamma scale row, and multiplying
the derivative of the target by \(\alpha\) gives the Gamma mean row.
The exponential scale case is \(\alpha=1\); reparameterizing by the rate \(\lambda\) or the mean \(\mu=1/\lambda\) gives the two exponential rows.

\subsection{One-dimensional models with explicit transition formulas}
\label{appendix-subsec:1d-trans-models}

These models have closed, piecewise, or one-root transition formulas.

\subsubsection{Bernoulli mean}
\label{appendix-subsubsec:trans-bernoulli}

Let \(X\sim\operatorname{Bernoulli}(q)\), fix \(q_0 \in(0,1)\), and estimate
\(q\) under squared loss.
At \(q_0\), take \(B=X\).
Then
\[
  S=\frac{X-q_0}{q_0(1-q_0)},
  \qquad
  \pi=q_0,
  \qquad
  a=\E_{q_0}(XS)=1.
\]
The binary score lemma gives
\begin{equation}
  \label{appendix-eq:trans-bernoulli-constant}
  \begin{aligned}
    C_{\sigma}^{\mathrm{Ber}}(\beta,q_0)
    &=
    \beta q_0(1-q_0)+\frac{1-\beta}{2},\\
    C_n^{\mathrm{Ber}}(\lambda,q_0)
    &=
    q_0(1-q_0)+\frac{1}{2\lambda},\\
    C_{\tau}^{\mathrm{Ber}}(\lambda,q_0)
    &=
    \frac{1}{2}+\lambda q_0(1-q_0).
  \end{aligned}
\end{equation}
The statistic \(\Phi(x)=x\) is efficient in both the privacy and sampling regimes, so the regime lower bound is exact.

\subsubsection{Laplace location}
\label{appendix-subsubsec:trans-laplace}

Let \(X\) have density
\[
  x\longmapsto
  \frac{1}{2b}\exp\xk{-\frac{\abs{x-\theta}}{b}},
  \qquad b>0,
\]
and estimate \(\theta\) under squared loss.
At \(\theta_0\),
\[
  S
  =
  \frac{1}{b}\sign(X-\theta_0).
\]
With \(B=\ind{X\ge\theta_0}\), we have
\(\pi=1/2\) and \(a=1/(2b)\).
Therefore
\begin{equation}
  \label{appendix-eq:trans-laplace-constant}
  \begin{aligned}
    C_\sigma^{\mathrm{Lap}}(\beta,b)
    &=
    b^2(2-\beta),\\
    C_n^{\mathrm{Lap}}(\lambda,b)
    &=
    b^2 \xk{1+\frac{2}{\lambda}},\\
    C_\tau^{\mathrm{Lap}}(\lambda,b)
    &=
    b^2(\lambda+2).
  \end{aligned}
\end{equation}
The indicator \(B\), equivalently the centered statistic
\(\sign(X-\theta_0)/2\), is optimal at every \(\lambda\).

\subsubsection{Logistic location: an elementary nonadditive constant}
\label{appendix-subsubsec:trans-logistic-location}

Consider the logistic location family with known scale \(s>0\) and CDF
\[
  F_\theta(x)
  =
  \frac{1}{1+\exp\xk{-(x-\theta)/s}}.
\]
At \(\theta_0\),
\[
  S
  =
  \frac{2F_{\theta_0}(X)-1}{s},
  \qquad
  U=sS\sim\operatorname{Unif}[-1,1].
\]
Define
\begin{equation}
  \label{appendix-eq:trans-logistic-threshold}
  a_\lambda
  =
  \frac{\sqrt{1+\lambda}-1}{\sqrt{1+\lambda}+1}
  =
  1+\frac{2}{\lambda}
  -
  \sqrt{\xk{1+\frac{2}{\lambda}}^2-1}.
\end{equation}
Since \(\abs{U}\sim\operatorname{Unif}[0,1]\), the threshold equation is
\[
  \frac{(1-a_\lambda)^2}{2a_\lambda}
  =
  \frac{2}{\lambda}.
\]
The clipped score reduction gives
\begin{equation}
  \label{appendix-eq:trans-logistic-constant}
  C_n^{\mathrm{Log}}(\lambda,s)
  =
  \frac{6s^2}{a_\lambda(3-a_\lambda^2)},
  \qquad
  C_\sigma^{\mathrm{Log}}
  =
  \beta C_n^{\mathrm{Log}},
  \qquad
  C_\tau^{\mathrm{Log}}
  =
  \lambda C_n^{\mathrm{Log}}.
\end{equation}
The optimizer is
\[
  \Phi_\lambda(X)
  =
  \frac{1}{2}\sign(U)
  \min\xk{\frac{\abs{U}}{a_\lambda},1}.
\]
The limiting values corresponding to the privacy and sampling regimes are
\[
  \lim_{\lambda\downarrow0}
  \lambda C_n^{\mathrm{Log}}(\lambda,s)
  =8s^2,
  \qquad
  \lim_{\lambda\to\infty}
  C_n^{\mathrm{Log}}(\lambda,s)
  =3s^2.
\]
For every finite \(\lambda\), the joint constant strictly exceeds the sum of the privacy and sampling regime lower bounds:
\begin{equation}
  \label{appendix-eq:trans-logistic-nonadditivity}
  C_n^{\mathrm{Log}}
  -
  \xk{3s^2+\frac{8s^2}{\lambda}}
  =
  s^2
  \frac{(1-a_\lambda)^2(2a_\lambda+3)}{3-a_\lambda^2}
  >0.
\end{equation}

\subsubsection{Binomial\texorpdfstring{\((3,q)\)}{(3,q)}: a piecewise rational constant}
\label{appendix-subsubsec:trans-binomial-three}

Let \(X\sim\operatorname{Binomial}(3,q)\), fix \(q_0=1/2\), and estimate
\(q\) under squared loss.
The score is
\[
  S=4X-6,
  \qquad
  \Pr(\abs{S}=2)=\frac{3}{4},
  \qquad
  \Pr(\abs{S}=6)=\frac{1}{4}.
\]
Solving the clipping equation on its two linear pieces gives
\[
  t_{\theta_0,\lambda}
  =
  \begin{cases}
    \dfrac{3\lambda}{\lambda+2},
    &0<\lambda\le4,\\[6pt]
    \dfrac{6\lambda}{\lambda+8},
    &\lambda\ge4.
  \end{cases}
\]
Consequently,
\begin{equation}
  \label{appendix-eq:trans-binomial-three-constant}
  C_n^{\mathrm{Bin}(3)}(\lambda)
  =
  \begin{cases}
    \dfrac{\lambda+2}{9\lambda},
    &0<\lambda\le4,\\[7pt]
    \dfrac{\lambda+8}{12(\lambda+2)},
    &\lambda\ge4.
  \end{cases}
\end{equation}
As usual,
\(C_\sigma^{\mathrm{Bin}(3)}=\beta C_n^{\mathrm{Bin}(3)}\) and
\(C_\tau^{\mathrm{Bin}(3)}=\lambda C_n^{\mathrm{Bin}(3)}\).
On the four values \(X=0,1,2,3\), an optimal centered statistic is
\[
  \Phi_\lambda
  =
  \begin{cases}
    (-1/2,-1/2,1/2,1/2),
    &0<\lambda\le4,\\[3pt]
    (-1/2,-x_\lambda,x_\lambda,1/2),
    &\lambda\ge4,
  \end{cases}
  \qquad
  x_\lambda=\frac{\lambda+8}{6\lambda}.
\]
The breakpoint \(\lambda=4\) is where the inner score values cease to be
fully clipped.
Moreover,
\[
  \lim_{\lambda\downarrow0}
  \lambda C_n^{\mathrm{Bin}(3)}(\lambda)
  =\frac{2}{9},
  \qquad
  \lim_{\lambda\to\infty}
  C_n^{\mathrm{Bin}(3)}(\lambda)
  =\frac{1}{12}.
\]
The lower bound based on the privacy and sampling regime constants is again strict at every finite \(\lambda\):
\begin{equation}
  \label{appendix-eq:trans-binomial-three-nonadditivity}
  C_n^{\mathrm{Bin}(3)}
  -
  \xk{\frac{1}{12}+\frac{2}{9\lambda}}
  =
  \begin{cases}
    \dfrac{1}{36},
    &0<\lambda\le4,\\[6pt]
    \dfrac{5\lambda-8}{18\lambda(\lambda+2)},
    &\lambda\ge4.
  \end{cases}
\end{equation}

\subsubsection{Poisson mean: asymmetric winsorization}
\label{appendix-subsubsec:trans-poisson-mean}

Let \(X\sim\operatorname{Poisson}(\nu)\), fix \(\nu>0\), and estimate
\(\nu\) under squared loss.
The Poisson family is regular, and its density derivative is \(L^1\)-continuous
on compact subsets of \((0,\infty)\).
Write
\[
  p_\nu(j)=\Pr_\nu(X=j),
  \qquad
  L_\nu(a)=\E_\nu(a-X)_+,
  \qquad
  U_\nu(b)=\E_\nu(X-b)_+.
\]
The two stop-loss functions meet at
\[
  q_\nu^\star
  =
  L_\nu(\nu)
  =
  U_\nu(\nu)
  =
  \nu p_\nu(\lfloor\nu\rfloor).
\]
For \(q\in(0,q_\nu^\star)\), let \(a_\nu(q)\in(0,\nu)\) and
\(b_\nu(q)\in(\nu,\infty)\) be determined by
\[
  L_\nu\xk{a_\nu(q)}
  =
  U_\nu\xk{b_\nu(q)}
  =q.
\]

\begin{proposition}[Poisson mean transition constant]
  \label{appendix-prop:trans-poisson-mean}
  For every \(\lambda\in(0,\infty)\), there is a unique
  \(q_{\nu,\lambda} \in(0,q_\nu^\star)\) satisfying
  \begin{equation}
    \label{appendix-eq:trans-poisson-tail-balance}
    b_\nu(q_{\nu,\lambda})-a_\nu(q_{\nu,\lambda})
    =
    2\lambda q_{\nu,\lambda}.
  \end{equation}
  Put
  \[
    a_{\nu,\lambda}=a_\nu(q_{\nu,\lambda}),
    \qquad
    b_{\nu,\lambda}=b_\nu(q_{\nu,\lambda}),
    \qquad
    w_{\nu,\lambda}=b_{\nu,\lambda}-a_{\nu,\lambda},
  \]
  and define
  \[
    G_{\nu,\lambda}(x)
    =
    \min\{(x-a_{\nu,\lambda})_+,w_{\nu,\lambda}\},
    \qquad
    D_{\nu,\lambda}
    =
    \E_\nu\zk{G_{\nu,\lambda}(X+1)-G_{\nu,\lambda}(X)}.
  \]
  Then the attaining diameter-one statistic is
  \begin{equation}
    \label{appendix-eq:trans-poisson-stat}
    \Phi_{\nu,\lambda}(x)
    =
    \frac{G_{\nu,\lambda}(x)}{w_{\nu,\lambda}},
  \end{equation}
  and the three normalizations are
  \begin{equation}
    \label{appendix-eq:trans-poisson-constant}
    C_n^{\mathrm{Poi}}(\lambda,\nu)
    =
    \frac{\nu}{D_{\nu,\lambda}},
    \qquad
    C_\sigma^{\mathrm{Poi}}(\beta,\nu)
    =
    \frac{\beta\nu}{D_{\nu,\lambda}},
    \qquad
    C_\tau^{\mathrm{Poi}}(\lambda,\nu)
    =
    \frac{\lambda\nu}{D_{\nu,\lambda}}.
  \end{equation}
  The continuous extensions of the unified constant are
  \begin{equation}
    \label{appendix-eq:trans-poisson-regime-limits}
    C_\sigma^{\mathrm{Poi}}(0,\nu)
    =
    \frac{1}{2p_\nu(\lfloor\nu\rfloor)^2},
    \qquad
    C_\sigma^{\mathrm{Poi}}(1,\nu)
    =
    \nu.
  \end{equation}
\end{proposition}

\begin{proof}
  Apply \cref{appendix-prop:trans-scalar-clipping} to the score
  \(S=(X-\nu)/\nu\).  After the affine change from score coordinates back
  to count coordinates, an optimal statistic has the form
  \(G(x)=\min\{(x-a)_+,b-a\}\), up to affine rescaling.
  Write \(w=b-a\) and \(\Phi=G/w\).
  Differentiating the resulting criterion with two thresholds, equivalently using
  the threshold KKT conditions, gives the identities.
  If \(m=\E_\nu \Phi(X)\) and \(c\) is the slope of the affine part as a
  function of the score \((X-\nu)/\nu\), the KKT conditions give
  \[
    m=\E_\nu \Phi(X),
    \qquad
    c=\frac{\nu}{w},
    \qquad
    \Var_\nu\xk{\Phi(X)}+\frac{1}{2\lambda}
    =
    c\E_\nu\zk{\Phi(X)\frac{X-\nu}{\nu}}.
  \]
  The intercept identity is
  \(m=(\nu-a)/w\).
  Since
  \[
    \E_\nu G(X)
    =
    \nu-a+L_\nu(a)-U_\nu(b),
  \]
  it is equivalent to \(L_\nu(a)=U_\nu(b)=q\).
  Under this tail balance, direct expansion gives
  \[
    \nu\E_\nu\zk{G(X+1)-G(X)}
    -
    \Var_\nu\xk{G(X)}
    =
    wq.
  \]
  Thus the scale condition reduces to \(w=2\lambda q\).
  Since \(a_\nu(q)\) is strictly increasing and \(b_\nu(q)\) is strictly
  decreasing, their difference decreases continuously from \(\infty\) to
  zero as \(q\) increases from zero to \(q_\nu^\star\).
  This proves the existence and uniqueness in
  \cref{appendix-eq:trans-poisson-tail-balance}, and the convex formulation
  proves global optimality of \cref{appendix-eq:trans-poisson-stat}.

  The Poisson Stein identity gives
  \[
    D_{\nu,\lambda}
    =
    \frac{1}{\nu}
    \E_\nu\zk{(X-\nu)G_{\nu,\lambda}(X)}.
  \]
  By \cref{appendix-eq:trans-poisson-tail-balance}, this is
  \[
    \Var_\nu\xk{G_{\nu,\lambda}(X)}
    +
    \frac{w_{\nu,\lambda}^2}{2\lambda}
    =
    \nu D_{\nu,\lambda}.
  \]
  Substituting \(\Phi_{\nu,\lambda}=G_{\nu,\lambda}/w_{\nu,\lambda}\)
  into \cref{appendix-eq:trans-scalar-clipping} gives
  \cref{appendix-eq:trans-poisson-constant}.

  As \(\lambda\downarrow0\),
  \[
    a_{\nu,\lambda},b_{\nu,\lambda} \to\nu,
    \qquad
    w_{\nu,\lambda} \sim2\lambda\nu p_\nu(\lfloor\nu\rfloor),
    \qquad
    \frac{D_{\nu,\lambda}}{w_{\nu,\lambda}}
    \to p_\nu(\lfloor\nu\rfloor).
  \]
  Hence
  \[
    \lambda C_n^{\mathrm{Poi}}(\lambda,\nu)
    \to
    \frac{1}{2p_\nu(\lfloor\nu\rfloor)^2}.
  \]
  As \(\lambda\to\infty\),
  \[
    a_{\nu,\lambda} \to0,
    \qquad
    b_{\nu,\lambda} \to\infty,
    \qquad
    D_{\nu,\lambda} \to1,
  \]
  so \(C_n^{\mathrm{Poi}}(\lambda,\nu)\to\nu\).
  Together with \(C_\sigma^{\mathrm{Poi}}=\beta C_n^{\mathrm{Poi}}\), the claims for the two limiting regimes follow.
\end{proof}

The calculation is explicit on each integer clipping phase.
Let
\[
  F_j=\Pr_\nu(X\le j),
  \qquad
  \overline{F}_j=\Pr_\nu(X\ge j),
  \qquad
  F_{-1}=0.
\]
Away from a phase boundary, let
\(j=\lfloor a_{\nu,\lambda} \rfloor\) and
\(k=\lfloor b_{\nu,\lambda} \rfloor\).
Solving the two tail balance equations on this phase gives
\[
  q_{jk}(\lambda)
  =
  \frac{
    \nu\xk{F_j p_\nu(k)+p_\nu(j)\overline{F}_{k+1}}
  }{
    F_j+\overline{F}_{k+1}+2\lambda F_j \overline{F}_{k+1}
  },
\]
\[
  a_{\nu,\lambda}
  =
  \frac{q_{jk}(\lambda)+\nu F_{j-1}}{F_j},
  \qquad
  b_{\nu,\lambda}
  =
  \frac{\nu\overline{F}_k-q_{jk}(\lambda)}{\overline{F}_{k+1}}.
\]
The valid pair is the unique \((j,k)\) satisfying
\(j\le a_{\nu,\lambda}<j+1\) and
\(k\le b_{\nu,\lambda}<k+1\).
On that phase,
\[
  D_{\nu,\lambda}
  =
  \begin{cases}
    (j+1-a_{\nu,\lambda})p_\nu(j)
    +\displaystyle\sum_{\ell=j+1}^{k-1} p_\nu(\ell)
    +(b_{\nu,\lambda}-k)p_\nu(k),
    &j<k,\\[8pt]
    (b_{\nu,\lambda}-a_{\nu,\lambda})p_\nu(j),
    &j=k.
  \end{cases}
\]
Thus the Poisson constant is piecewise rational in \(\lambda\) over countably
many clipping phases, with no numerical optimization required.

\subsection{Bounded-support mean models}
\label{appendix-subsec:bounded-finite-support}

The next results are global minimax consequences of local submodels rather than local parametric statements.

\begin{theorem}[Bounded scalar and hypercube means]
  \label{appendix-thm:bounded-means}
Let \(\ell<r\), let \(\Delta=r-\ell\), and let \(\mu(P)=\E_P X\).
Suppose that \(n^2\rho_n\to\infty\) and \(n\rho_n\to0\).
Then
\[
  \lim_{n\to\infty}
  \tau_n^2
  \inf_{\widehat{\mu}:\rho_n \text{-zCDP}}
  \sup_{P:\operatorname{supp}(P)\subset[\ell,r]}
  \E_P(\widehat{\mu}-\mu(P))^2
  =
  \frac{\Delta^2}{2}.
\]
For fixed \(p\),
\[
  \lim_{n\to\infty}
  \tau_n^2
  \inf_{\widehat{\mu}:\rho_n \text{-zCDP}}
  \sup_{P:\operatorname{supp}(P)\subset[\ell,r]^p}
  \E_P \norm{\widehat{\mu}-\mu(P)}_2^2
  =
  \frac{p^2 \Delta^2}{2}.
\]
\end{theorem}

\begin{proof}
For the scalar upper bound, release the empirical mean with independent noise
\[
  G_n \sim N\left(0,\frac{\Delta^2}{2\tau_n^2}\right).
\]
The query sensitivity is \(\Delta/n\), so the release is \(\rho_n\)-zCDP.
Its sampling risk is at most \(\Delta^2/(4n)=o(\tau_n^{-2})\).
For the lower bound, restrict to \(X=\ell+\Delta Y\), where
\(Y\sim\operatorname{Bernoulli}(q)\), locally around \(q_0=1/2\).
At \(q_0\), \(a_{q_0}=1\), while the mean derivative is \(\Delta\), so
\cref{prop:1d-matching} gives the constant
\(\Delta^2/2\).

For the vector upper bound, the empirical mean has replacement sensitivity
\(\sqrt p \Delta/n\).
Adding Gaussian noise with covariance
\[
  \frac{p\Delta^2}{2\tau_n^2}I_p
\]
gives privacy risk \(p^2 \Delta^2/(2\tau_n^2)\); the sampling risk is
negligible.

For the lower bound, restrict to
\(X=\ell\mathbf{1}_p+\Delta Y\), where the coordinates of \(Y\) are independent
Bernoulli variables, and work locally at \(q_0=\mathbf{1}_p/2\).
With \(V=2Y-\mathbf{1}_p\), the score is \(2V\).
This is the fixed radius score law in
\cref{appendix-prop:fixed-radius-score-reduction} with \(\lambda=4\).
Thus \(\Gamma_{q_0}=1\), \(I_p/p\) is feasible, and
\cref{appendix-cor:euclidean-trace-lower}, together with
\(D\mu(q_0)=\Delta I_p\), gives \(p^2 \Delta^2/2\).
The product Bernoulli submodel is contained in the full hypercube distribution
class, completing the lower bound.
\end{proof}

Let \(\Delta=r-\ell\).
For fixed \(p\) and \(n\rho_n \to\lambda\in(0,\infty)\), the exact global transition
constants over bounded distribution classes are
\begin{equation}
  \label{appendix-eq:trans-bounded-mean-constants}
  \begin{aligned}
    \lim_{n\to\infty}
    n
    \inf_{\widehat{\mu}:\rho_n \text{-zCDP}}
    \sup_{\operatorname{supp}(P)\subset[\ell,r]}
    \E_P(\widehat{\mu}-\mu(P))^2
    &=
    \Delta^2 \xk{\frac{1}{4}+\frac{1}{2\lambda}},\\
    \lim_{n\to\infty}
    n
    \inf_{\widehat{\mu}:\rho_n \text{-zCDP}}
    \sup_{\operatorname{supp}(P)\subset[\ell,r]^p}
    \E_P \norm{\widehat{\mu}-\mu(P)}_2^2
    &=
    \Delta^2 \xk{\frac{p}{4}+\frac{p^2}{2\lambda}}.
  \end{aligned}
\end{equation}
For the upper bounds, release the empirical mean with Gaussian noise having
variance \(\Delta^2/(2\tau_n^2)\) in the scalar case and covariance
\(p\Delta^2 I_p/(2\tau_n^2)\) in the vector case.
These are calibrated to sensitivities \(\Delta/n\) and
\(\sqrt p\Delta/n\), respectively.
The maximal sampling variances are \(\Delta^2/4\) and
\(p\Delta^2/4\).
For the lower bounds, restrict to the two-point Bernoulli model in the scalar
case and to the product two-point Bernoulli model at
\(q_0=\mathbf{1}_p/2\) in the vector case.
The latter has score \(S=2V\), where
\(V\in\{-1,1\}^p\), so it is the fixed radius model with \(\kappa=4\).
Multiplication by the target derivative \(\Delta I_p\) gives the second
constant.

\subsection{Categorical models}

\subsubsection{Saturated categorical models}
\label{appendix-subsubsec:categorical-prob}

The categorical model provides a finite support vector problem with simplex rather than hypercube geometry.

Let \(q\) be an interior point of the \(k\)-category simplex, use category
\(k\) as baseline, and put \(d=k-1\).
Write
\[
  V_q
  =
  \operatorname{diag}(q_1,\ldots,q_d)
  -q_{1:d} q_{1:d}^\top
\]
and the configuration Gram region \(\mathcal{G}_k\) in
\cref{appendix-eq:trans-config-region}.

\begin{proposition}[Saturated categorical constants]
  \label{appendix-prop:trans-categorical-constant}
  For \(W\succ0\), define the privacy constant
  \[
    C_{\mathrm{priv}}(W)
    =
    \frac{1}{2}
    \inf_{J\in\mathcal{G}_k} \caI_W(J).
  \]
  Then
  \begin{equation}
    \label{appendix-eq:trans-categorical-constant}
    \begin{aligned}
      \uC_q(W)
      & =
      \beta\Tr(WV_q)
      +(1-\beta)C_{\mathrm{priv}}(W),\\
      C_n^{\mathrm{cat}}(\lambda,q,W)
      & =
      \Tr(WV_q)
      +\frac{C_{\mathrm{priv}}(W)}{\lambda},\\
      C_\tau^{\mathrm{cat}}(\lambda,q,W)
      & =
      \lambda\Tr(WV_q)
      +C_{\mathrm{priv}}(W).
    \end{aligned}
  \end{equation}
  At every finite \(\lambda\), the configurations that are optimal in the
  transition regime are exactly those that are optimal in the privacy regime.
\end{proposition}

\begin{proof}
  Every statistic on the categorical sample space is specified by values
  \(\phi_1,\ldots,\phi_k\).
  Subtracting \(\phi_k\) preserves its diameter, covariance, and expectation derivative.
  Define \(Ue_a=\phi_a-\phi_k\), \(a=1,\ldots,d\).
  Its translated values lie in a space of dimension at most \(d\), and
  \[
    J=U^*U\in\mathcal G_k,
    \qquad
    A_{\Phi,q}=U,
    \qquad
    V_{\Phi,q}=UV_qU^*.
  \]
  Conversely, every \(J\in\mathcal G_k\) has a factorization \(J=U^*U\)
  whose columns, together with the zero baseline, form an admissible configuration.
  Thus these configurations exhaust the admissible statistics up to translation
  and an isometry of their span.
  For \(0<\beta<1\), the unified information has kernel \(\ker U\).
  Since \(W\succ0\), a rank-deficient \(U\) therefore has infinite inverse-information
  loss in both the privacy and transition problems.
  Full-rank admissible configurations exist, so the infima may be restricted to them.
  Moreover, the configuration Gram region is compact and its map to unified
  information is continuous for \(0<\beta<1\), so taking the information-region
  closure adds no further matrices.
  Applying \cref{appendix-lem:trans-config-linear} to the full-rank class with
  \(V_0=V_q\) now identifies its class-restricted constants with the
  unrestricted categorical constants and gives the first identity.
  The rescaled constants follow from
  \(C_n^{\mathrm{cat}}(\lambda,q,W)=\beta^{-1}\uC_q(W)\) and
  \(C_\tau^{\mathrm{cat}}(\lambda,q,W)=\lambda C_n^{\mathrm{cat}}(\lambda,q,W)\).
\end{proof}

\paragraph{Euclidean global privacy endpoint.}
Release the empirical probability vector with Gaussian noise supported on
\[
  \caT=\dk{v\in\R^k:\sum_{j=1}^k v_j=0}
\]
and covariance \(\tau_n^{-2} I_{\caT}\), followed by projection onto the
simplex.
The empirical vector has replacement sensitivity \(\sqrt2/n\), and the
privacy contribution to squared Euclidean risk is
\((k-1)/\tau_n^2\); its sampling contribution is \(O(n^{-1})\).

For the lower bound, work locally around
\(q_0=k^{-1} \mathbf{1}_k\).
Let \(H:\R^{k-1} \to\caT\) be an isometry and parameterize
\(q(\vartheta)=q_0+H\vartheta\).
Writing \(h_j=H^\top e_j\), the score at zero is \(s_0(X=j)=kh_j\).
This is exactly the score law in
\cref{appendix-prop:simplex-score-reduction}, which gives the constant \(k-1\) on
the tangent space.
The local categorical submodel is contained in the global simplex problem.

\paragraph{Euclidean transition and global simplex bound.}
For squared Euclidean loss on the full probability vector, the baseline
coordinate loss is \(W=I_d+\mathbf{1}_d \mathbf{1}_d^\top\).
The sampling and privacy terms are
\[
  \Tr(WV_q)=1-\sum_{j=1}^k q_j^2,
  \qquad
  C_{\mathrm{priv}}(W)=k-1.
\]
Therefore
\begin{equation}
  \label{appendix-eq:trans-categorical-euclidean-constant}
  C_n^{\mathrm{cat}}(\lambda,q)
  =
  1-\sum_{j=1}^k q_j^2+\frac{k-1}{\lambda}.
\end{equation}
The same calculation and a local lower bound at the uniform distribution give
the global simplex result
\begin{equation}
  \label{appendix-eq:trans-categorical-global-constant}
  \lim_{n\to\infty}
  n
  \inf_{\widehat{q}:\rho_n \text{-zCDP}}
  \sup_{q\in\Delta_{k-1}}
  \E_q \norm{\widehat{q}-q}_2^2
  =
  \frac{k-1}{k}+\frac{k-1}{\lambda},
  \qquad
  n\rho_n \to\lambda.
\end{equation}
For the upper bound, release the empirical probability vector with Gaussian
noise of covariance \(\tau_n^{-2} I_{\caT}\) on the tangent space \(\caT\),
and project onto the simplex.
The sampling trace is \(1-\sum_{j=1}^k q_j^2\), whose maximum is
\((k-1)/k\).
The local saturated model at the uniform distribution gives the matching
lower bound.

\begin{proof}[Proof of \cref{prop:categorical-prob}]
  By the definition of \(C_{\mathrm{priv}}(W)\),
  \(C_{\mathrm{priv}}(W)=\mfC_q(W)\).
  For \(0<\beta<1\), the first formula in
  \cref{appendix-prop:trans-categorical-constant} is therefore exactly the
  general-\(W\) unified formula in the theorem, and its optimizer statement
  gives the final assertion.
  At \(\beta=0\) and \(\beta=1\), the same identity follows directly from
  the privacy and Fisher endpoint definitions.
  For squared Euclidean loss,
  \cref{appendix-eq:trans-categorical-euclidean-constant,appendix-eq:trans-categorical-global-constant}
  give the stated specialization.
\end{proof}

\paragraph{Three-category geometry and loss dependence.}
Write \(q_3=1-q_1-q_2\), and let \(\Phi_j=\Phi(j)\) for a Hilbert-valued statistic \(\Phi\).
The coordinate matrices of the two losses are
\[
  W_{\mathrm{iso}}
  =
  \begin{pmatrix}
    2 & 1\\
    1 & 2
  \end{pmatrix},
  \qquad
  W_{\mathrm{aniso}}
  =
  \begin{pmatrix}
    13 & 1\\
    1 & 2
  \end{pmatrix}.
\]
Its mean map in the \((q_1,q_2)\) coordinates is
\[
  g_\Phi(q)
  =
  q_1 \Phi_1+q_2 \Phi_2+(1-q_1-q_2)\Phi_3.
\]
Hence, with
\[
  u=\Phi_1-\Phi_3,
  \qquad
  v=\Phi_2-\Phi_3,
\]
the score--statistic operator satisfies
\[
  A_{\Phi,q} e_1=u,
  \qquad
  A_{\Phi,q} e_2=v.
\]
Thus \(J_{\Phi,q}\) is the Gram matrix of \(u\) and \(v\).
Moreover, \(\diam(\Phi)\le1\) is equivalent to
\[
  \norm{u}\le1,
  \qquad
  \norm{v}\le1,
  \qquad
  \norm{u-v}\le1.
\]
If
\[
  J_{\Phi,q}
  =
  \begin{pmatrix}
    a & c\\
    c & b
  \end{pmatrix},
\]
these inequalities become
\[
  a\le1,
  \qquad
  b\le1,
  \qquad
  a+b-2c\le1.
\]
Conversely, every positive semidefinite matrix satisfying these three inequalities admits a Gram representation by vectors \(u\) and \(v\) with the stated distance constraints.
Taking \(\Phi_1=u\), \(\Phi_2=v\), and \(\Phi_3=0\) realizes that matrix.
This gives the information region
\begin{equation}
  \label{appendix-eq:categorical-three-info-region}
  \caJ_{q_0}
  =
  \dk{
    \begin{pmatrix}
      a&c\\c&b
    \end{pmatrix}
    \succeq0:
    a\le1,\ b\le1,\ a+b-2c\le1
  }.
\end{equation}

The two loss-specific candidate matrices are
\begin{equation}
  \label{appendix-eq:categorical-loss-dependent-optimizers}
  J_{\mathrm{iso}}
  =
  \begin{pmatrix}
    1&1/2\\1/2&1
  \end{pmatrix},
  \qquad
  J_{\mathrm{aniso}}
  =
  \begin{pmatrix}
    1&1/3\\1/3&2/3
  \end{pmatrix}.
\end{equation}

Define the diameter-one statistics
\[
  \begin{aligned}
    \Phi_{\mathrm{iso}}(1)&=(1,0),
    &
    \Phi_{\mathrm{iso}}(2)&=(1/2,\sqrt{3}/2),
    &
    \Phi_{\mathrm{iso}}(3)&=(0,0),
    \\
    \Phi_{\mathrm{aniso}}(1)&=(1,0),
    &
    \Phi_{\mathrm{aniso}}(2)&=(1/3,\sqrt{5}/3),
    &
    \Phi_{\mathrm{aniso}}(3)&=(0,0).
  \end{aligned}
\]
They generate the two candidate matrices in
\cref{appendix-eq:categorical-loss-dependent-optimizers}.
Indeed, their respective vectors \(u\) and \(v\) satisfy
\[
  \begin{array}{c|ccc}
    & \norm{u}^2 & \norm{v}^2 & \norm{u-v}^2\\
    \hline
    \Phi_{\mathrm{iso}} & 1 & 1 & 1\\
    \Phi_{\mathrm{aniso}} & 1 & 2/3 & 1
  \end{array}
\]
and their inner products are \(1/2\) and \(1/3\), respectively.
Both candidates are therefore positive definite feasible matrices.

It remains to prove optimality.
For \(W\succ0\), singular matrices have infinite objective, so it is enough to minimize
\[
  F_W(J)=\Tr(WJ^{-1})
\]
over the positive definite part of the information region.
At a candidate \(J_*\), put
\[
  B_*=J_*^{-1}WJ_*^{-1}.
\]
Since \(\nabla F_W(J_*)=-B_*\), convexity shows that \(J_*\) is optimal whenever
\(\Tr(B_*J)\le\Tr(B_*J_*)\) for every feasible \(J\).
For the two candidates,
\[
  B_{\mathrm{iso}}
  =
  \frac{4}{3}
  \begin{pmatrix}
    2 & -1\\
    -1 & 2
  \end{pmatrix},
  \qquad
  B_{\mathrm{aniso}}
  =
  9
  \begin{pmatrix}
    2 & -1\\
    -1 & 1
  \end{pmatrix}.
\]
For any feasible \(J=\left(\begin{smallmatrix}a&c\\c&b\end{smallmatrix}\right)\),
\[
  \Tr(B_{\mathrm{iso}}J)
  =
  \frac{4}{3}\xk{a+b+(a+b-2c)}
  \le4
  =
  \Tr(B_{\mathrm{iso}}J_{\mathrm{iso}})
\]
and
\[
  \Tr(B_{\mathrm{aniso}}J)
  =
  9\xk{a+(a+b-2c)}
  \le18
  =
  \Tr(B_{\mathrm{aniso}}J_{\mathrm{aniso}}).
\]
The map \(F_W\) is strictly convex for \(W\succ0\), so both optimizers are unique.

The inverses are
\[
  J_{\mathrm{iso}}^{-1}
  =
  \begin{pmatrix}
    4/3 & -2/3\\
    -2/3 & 4/3
  \end{pmatrix},
  \qquad
  J_{\mathrm{aniso}}^{-1}
  =
  \begin{pmatrix}
    6/5 & -3/5\\
    -3/5 & 9/5
  \end{pmatrix}.
\]
They give
\[
  \frac{1}{2}\Tr(W_{\mathrm{iso}}J_{\mathrm{iso}}^{-1})=2,
  \qquad
  \frac{1}{2}\Tr(W_{\mathrm{aniso}}J_{\mathrm{aniso}}^{-1})=9,
\]
as well as the values under the other loss
\[
  \frac{1}{2}\Tr(W_{\mathrm{iso}}J_{\mathrm{aniso}}^{-1})
  =
  \frac{12}{5},
  \qquad
  \frac{1}{2}\Tr(W_{\mathrm{aniso}}J_{\mathrm{iso}}^{-1})
  =
  \frac{28}{3}.
\]
Finally,
\[
  \det(J_{\mathrm{iso}}-J_{\mathrm{aniso}})
  =
  -\frac{1}{36}<0,
\]
so the two optimizing matrices are not comparable in the Loewner order.

The three-category comparison shows how the optimal geometry changes with the loss.
At \(q_0=(1/3,1/3,1/3)\), the two losses in that example satisfy
\[
  \Tr(W_{\mathrm{iso}}V_{q_0})=\frac{2}{3},
  \qquad
  \Tr(W_{\mathrm{aniso}}V_{q_0})=\frac{28}{9},
\]
while their privacy constants are \(2\) and \(9\), respectively.
Hence
\begin{equation}
  \label{appendix-eq:trans-three-category-constants}
  \begin{aligned}
    C_{n,\mathrm{iso}}(\lambda)
    &=
    \frac{2}{3}+\frac{2}{\lambda},
    &
    C_{\tau,\mathrm{iso}}(\lambda)
    &=
    2+\frac{2\lambda}{3},\\
    C_{n,\mathrm{aniso}}(\lambda)
    &=
    \frac{28}{9}+\frac{9}{\lambda},
    &
    C_{\tau,\mathrm{aniso}}(\lambda)
    &=
    9+\frac{28\lambda}{9}.
  \end{aligned}
\end{equation}
The two different matrices in
\cref{appendix-eq:categorical-loss-dependent-optimizers} remain the respective
unique optimizers for all \(\lambda\in(0,\infty)\).
Thus the matrix geometry continues to depend on the loss throughout the
transition regime even though the constant for each fixed loss is additive.

\subsection{Multivariate location and quantile models}
\label{appendix-subsec:multivariate-location-quantile-models}

\subsubsection{Gaussian quantiles with unknown scale}
\label{appendix-subsubsec:gauss-quantiles-unknown-scale}

\begin{proposition}[Gaussian quantile with unknown scale]
  \label{appendix-prop:gauss-quantile-unknown-scale}
  Consider \(X\sim N(\mu,\sigma^2)\), fix \((\mu_0,\sigma_0)\), and let
  \(q_\alpha(\mu,\sigma)=\mu+z_\alpha \sigma\), where \(z_\alpha\) is the \(\alpha\)-quantile of the standard normal law.
  Put
  \begin{equation}
    \label{appendix-eq:gauss-quantile-constant}
    d_\alpha
    =
    \inf_{a\in\R}
    \E\abs{Z-a\xk{Z^2-1-z_\alpha Z}},
    \qquad Z\sim N(0,1).
  \end{equation}
  The exact local minimax constant under squared loss in the privacy regime is
  \[
    \frac{2\sigma_0^2}{d_\alpha^2}.
  \]
\end{proposition}

\begin{proof}
Parameterize the model by \(\theta=(\mu,\sigma)\), and fix
\(\theta_0=(\mu_0,\sigma_0)\).
For
\[
  g_\alpha=(1,z_\alpha)^\top,
  \qquad
  W_\alpha=g_\alpha g_\alpha^\top,
\]
the quadratic loss induced by \(W_\alpha\) is exactly the squared error for
\(q_\alpha(\mu,\sigma)=\mu+z_\alpha \sigma\).
Every vector estimator induces a scalar estimator by taking its inner product
with \(g_\alpha\), and every scalar estimator can be embedded into a vector
estimator with the same loss.  Thus the scalar minimax problem is equivalent
to the \(W_\alpha\)-problem in \cref{thm:private-minimax}.
Writing \(Z=(X-\mu_0)/\sigma_0\), the score at \(\theta_0\) is
\begin{equation}
  \label{appendix-eq:gauss-location-scale-score}
  s_{\theta_0}(X)
  =
  \frac{1}{\sigma_0}
  \begin{pmatrix}
    Z\\ Z^2-1
  \end{pmatrix}.
\end{equation}

We first reduce the rank-one information problem to scalar statistics.
Let \(\Phi\) be a diameter-one Hilbert-valued statistic, write
\(A=A_{\Phi,\theta_0}\), and suppose
\(\caI_{W_\alpha}(A^*A)<\infty\).
Then \(g_\alpha \in\ran(A^*)\), and
\[
  \caI_{W_\alpha}(A^*A)
  =
  g_\alpha^\top(A^*A)^\dagger g_\alpha
  =
  \min\dk{
    \norm{v}_{\caH}^2:A^*v=g_\alpha
  }.
\]
If \(v\) is the minimum-norm solution, then
\(\phi(x)=\ang{\Phi(x),v/\norm{v}_{\caH}}_{\caH}\) has diameter at most one and
\[
  \E_{\theta_0}\zk{\phi(X)s_{\theta_0}(X)}
  =
  \frac{g_\alpha}{\norm{v}_{\caH}}.
\]
Conversely, a scalar diameter-one statistic satisfying
\(\E_{\theta_0}\zk{\phi(X)s_{\theta_0}(X)}=a g_\alpha\)
generates the rank-one matrix \(a^2 g_\alpha g_\alpha^\top\), whose
\(W_\alpha\)-information objective is \(a^{-2}\).
To pass from actual statistics to the closure defining \(\caJ_{\theta_0}\),
first mix any limiting matrix with an arbitrarily small positive-definite
element of \(\caJ_{\theta_0}\), whose existence follows from
\cref{prop:dc-info-prop}\textup{(iv)}, and then approximate the resulting
positive-definite matrix by actual statistics.
The objective is continuous at the perturbed matrix and converges to its
pseudoinverse value as the perturbation vanishes, so the closure does not
change the infimum.
Consequently,
\begin{equation}
  \label{appendix-eq:gauss-quantile-rank-one-reduction}
  \inf_{J\in\caJ_{\theta_0}}\caI_{W_\alpha}(J)
  =
  \frac{1}{a_\alpha^2},
\end{equation}
where \(a_\alpha\) is the largest \(a\ge0\) generated by such a scalar statistic.

Set
\[
  H_\alpha(Z)=Z^2-1-z_\alpha Z.
\]
By \cref{appendix-eq:gauss-location-scale-score}, alignment with
\(g_\alpha\) is equivalent to
\[
  \E\zk{\phi(Z)H_\alpha(Z)}=0,
  \qquad
  a=\frac{1}{\sigma_0}\E\zk{\phi(Z)Z}.
\]
Since \(\E Z=\E H_\alpha(Z)=0\), translating \(\phi\) does not change these
moments, so the diameter constraint may be written as \(\abs{\phi}\le1/2\).
For every \(\lambda\in\R\), any feasible \(\phi\) satisfies
\[
  \E\zk{\phi(Z)Z}
  =
  \E\left[\phi(Z)\xk{Z-\lambda H_\alpha(Z)}\right]
  \le
  \frac{1}{2}
  \E\abs{Z-\lambda H_\alpha(Z)}.
\]
The function on the right is convex and coercive in \(\lambda\), and hence
has a minimizer \(\lambda_\alpha^*\).
The polynomial
\(Z-\lambda_\alpha^*H_\alpha(Z)\) is not identically zero, so its zero set
has Gaussian probability zero.
Differentiability at the minimizer gives
\[
  \E\left[
    H_\alpha(Z)
    \sign\xk{Z-\lambda_\alpha^*H_\alpha(Z)}
  \right]
  =0.
\]
Therefore
\[
  \phi_\alpha^*(Z)
  =
  \frac{1}{2}
  \sign\xk{Z-\lambda_\alpha^*H_\alpha(Z)}
\]
viewed as a function of \(x\) through \(Z=(x-\mu_0)/\sigma_0\), is feasible
and attains the preceding upper bound.  It follows that
\begin{equation}
  \label{appendix-eq:gauss-quantile-optimal-slope}
  a_\alpha
  =
  \frac{d_\alpha}{2\sigma_0}.
\end{equation}
Combining \cref{appendix-eq:gauss-quantile-rank-one-reduction,appendix-eq:gauss-quantile-optimal-slope}
with the factor \(1/2\) in the definition of \(\mfC_{\theta_0}\) yields
\[
  \mfC_{\theta_0}(W_\alpha)
  =
  \frac{2\sigma_0^2}{d_\alpha^2}.
\]

The rank-one optimum can be approached by locally invertible
Gaussian releases.
Fix \(c>0\) and let
\(\psi_c(x)=\ind{\abs{x-\mu_0}\ge c\sigma_0}\).
If \(\varphi\) denotes the standard normal density, its mean derivative at
\(\theta_0\) is
\[
  \E_{\theta_0}\zk{\psi_c(X)s_{\theta_0}(X)}
  =
  \frac{2c\varphi(c)}{\sigma_0}(0,1)^\top,
\]
which is linearly independent of \(g_\alpha\).
For \(\epsilon\in(0,1)\), the two-dimensional statistic
\[
  \Phi_{\alpha,\epsilon}(x)
  =
  \left(
    \sqrt{1-\epsilon^2}\,\phi_\alpha^*(x),
    \epsilon\psi_c(x)
  \right)
\]
has Euclidean diameter at most one and a nonsingular mean derivative.
Its moment map inverse estimator has constant tending to
\(2\sigma_0^2/d_\alpha^2\) as \(\epsilon\downarrow0\), and a diagonal choice
gives the attaining sequence in the main theorem.

When \(\alpha=1/2\), we have \(z_\alpha=0\).
Symmetry and convexity imply that \(\lambda=0\) minimizes
\(\E\abs{Z-\lambda(Z^2-1)}\), so
\(d_{1/2}=\E\abs{Z}=\sqrt{2/\pi}\), and the constant reduces to
\(\pi\sigma_0^2\).
\end{proof}

\subsubsection{Fixed-radius radial Laplace location}
For the radial Laplace location density proportional to
\(\exp\xk{-\norm{x-\theta}_2/b}\), the score is \(S=U/b\), where \(U\) is
uniform on the unit sphere.
Thus \(\kappa=1/(pb^2)\), and
\begin{equation}
  \label{appendix-eq:trans-radial-laplace-constant}
  C_n^{\mathrm{radLap}}(\lambda,b)
  =
  p^2 b^2 \xk{1+\frac{2p}{\lambda}},
  \qquad
  C_\tau^{\mathrm{radLap}}(\lambda,b)
  =
  p^2 b^2(\lambda+2p).
\end{equation}
The optimizer is the spatial sign
\[
  \Phi_{\theta_0}(x)
  =
  \frac{1}{2}
  \frac{x-\theta_0}{\norm{x-\theta_0}_2}.
\]
Set \(\Phi_{\theta_0}(\theta_0)=0\).
For \(p=1\), this agrees with
\cref{appendix-eq:trans-laplace-constant} up to translation of the statistic.

\subsubsection{Gaussian location: a vector constant defined by one root}
\label{appendix-subsubsec:trans-gauss-location}

Let \(X\sim N_p(\theta,\sigma^2 I_p)\), and estimate \(\theta\) under squared
Euclidean loss.
Write
\[
  Z=\frac{X-\theta_0}{\sigma},
  \qquad
  R=\norm{Z}_2 \sim\chi_p.
\]
The score is \(S=Z/\sigma\).  By
\cref{appendix-eq:trans-spherical-threshold}, let \(b_{p,\lambda}>0\) be the unique solution of
\begin{equation}
  \label{appendix-eq:trans-gauss-clipping-threshold}
  \E\xk{\frac{R}{b_{p,\lambda}}-1}_+
  =
  \frac{2p}{\lambda}.
\end{equation}
Then \cref{appendix-prop:trans-spherical-clipping} gives the constant,
determined by a scalar root:
\begin{equation}
  \label{appendix-eq:trans-gauss-location-constant}
  C_n^{\mathrm{Gau},p}(\lambda,\sigma)
  =
  \frac{p^2 \sigma^2}{
    \E\zk{R\min(R,b_{p,\lambda})}
  },
  \qquad
  C_\sigma^{\mathrm{Gau},p}
  =
  \beta C_n^{\mathrm{Gau},p},
  \qquad
  C_\tau^{\mathrm{Gau},p}
  =
  \lambda C_n^{\mathrm{Gau},p}.
\end{equation}
The optimal statistic is the radial clipping
\[
  \Phi_{p,\lambda}(X)
  =
  \frac{1}{2b_{p,\lambda}}
  \min\xk{1,\frac{b_{p,\lambda}}{\norm{Z}_2}}Z,
\]
with value zero at \(Z=0\).
The limits corresponding to the sampling and privacy regimes recover the Fisher and privacy constants:
\begin{equation}
  \label{appendix-eq:trans-gauss-location-regime-limits}
  \lim_{\lambda\to\infty}
  C_n^{\mathrm{Gau},p}(\lambda,\sigma)
  =p\sigma^2,
  \qquad
  \lim_{\lambda\downarrow0}
  \lambda C_n^{\mathrm{Gau},p}(\lambda,\sigma)
  =
  \frac{2\sigma^2 p^3}{(\E\chi_p)^2}.
\end{equation}
Since \(C_\sigma^{\mathrm{Gau},p}=\beta C_n^{\mathrm{Gau},p}\), this calculation gives
\cref{eq:gauss-unified-threshold,eq:gauss-unified-constant}
and completes the transition part of \cref{prop:gauss-mean}.

For \(p=1\), let \(\varphi_{\mathrm{N}}\) denote the standard normal density
and put \(\overline{F}_{\mathrm{N}}(b)=\Pr(Z>b)\).
The threshold and the constant simplify to
\begin{equation}
  \label{appendix-eq:trans-scalar-gauss-constant}
  \frac{1}{\lambda}
  =
  \frac{\varphi_{\mathrm{N}}(b_{1,\lambda})}{b_{1,\lambda}}
  -
  \overline{F}_{\mathrm{N}}(b_{1,\lambda}),
  \qquad
  C_n^{\mathrm{Gau},1}(\lambda,\sigma)
  =
  \frac{\sigma^2}{\Pr(\abs{Z}\le b_{1,\lambda})}.
\end{equation}
Indeed,
\(\E\zk{\abs{Z}\min(\abs{Z},b)}=\Pr(\abs{Z}\le b)\) for a standard normal
variable.
This is a genuine transition calculation: the optimal clipping level tends
to zero in the privacy limit and to infinity in the sampling limit.

\subsubsection{Other spherical location families}
\label{appendix-subsubsec:spherical-location-models}

The formula in \cref{appendix-prop:trans-spherical-clipping}, determined by a single root, applies to every fixed-dimensional location family whose score is spherical.
Thus its full unified constant is determined by the law of \(R=\norm{s_{\theta_0}(X)}_2\), while the constant in the privacy regime is \(2p^3/\xk{\E R}^2\).
For the radial generalized Gaussian density proportional to
\(\exp\xk{-(\norm{x-\theta}_2/b)^\alpha}\),
\[
  \E R
  =
  \frac{\alpha}{b}
  \frac{\Gamma((p+\alpha-1)/\alpha)}{\Gamma(p/\alpha)}.
\]
For the multivariate Student location family with \(\nu\) degrees of freedom
and scale \(\sigma\),
\[
  \E R
  =
  \frac{2}{\sigma\sqrt\nu}
  \frac{\Gamma((p+1)/2)\Gamma((\nu+1)/2)}
  {\Gamma(p/2)\Gamma(\nu/2)}.
\]
These give the radial generalized Gaussian, Student, and Cauchy rows of
Table~S1 in the privacy regime.  The regularity restriction
\(\alpha>\max(0,1-p/2)\) is precisely the condition needed for finite Fisher
information and differentiability in quadratic mean at the origin.

\subsection{Regression models with explicit constants}
\label{appendix-subsec:regression-explicit-constants}

These models have factorized spherical scores, so their explicit unified constants follow from the corresponding reduction in \cref{appendix-sec:model-class-reductions}.

\subsubsection{Generalized linear regression with Gaussian design}
\label{appendix-subsubsec:trans-gauss-design-glm}

\begin{proof}[Proof of the unified formula in \cref{prop:gauss-design-glm}]
  At the null center, write
  \[
    S=s_0(X,Y_0)=A_0 X,
    \qquad
    A_0=\frac{Y_0-b'(0)}{\kappa}.
  \]
  The response variable \(A_0\) is independent of \(X\), so the factorized
  spherical score reduction gives the radial variable
  \(R_0=\abs{A_0}\chi_p\).  Thus
  \cref{appendix-cor:trans-factorized-spherical-score} applies directly.
  For \(0<\beta<1\), it gives
  \[
    \uC_0(I_p)
    =
    \frac{\beta p^2}{\E\zk{R_0 \min(R_0,t_\beta)}},
    \qquad
    \E\xk{\frac{R_0}{t_\beta}-1}_+
    =
    \frac{2p(1-\beta)}{\beta}.
  \]
  This is the unified constant in
  \cref{prop:gauss-design-glm}, and the attaining statistic is
  the corresponding radial clipping of \(S\).

  As \(\beta\downarrow0\), the formula gives
  \[
    \mfC_0^{(0)}(I_p)
    =
    \frac{2p^3}{\xk{\E R_0}^2}
    =
    \frac{2p^3}{(\E\chi_p)^2d_0^2}.
  \]
  As \(\beta\uparrow1\), it gives
  \[
    \mfC_0^{(1)}(I_p)
    =
    \frac{p^2}{\E R_0^2}
    =
    \frac{p\kappa}{b''(0)},
  \]
  because \(\E A_0^2=b''(0)/\kappa\).
\end{proof}

\subsubsection{Other factorized spherical regression models}

More generally, let \(X=RU\), where \(U\) is uniform on the unit sphere and
independent of \(R\), and suppose that the score at the local center is
\[
  s_{\beta_0}(X,Y)=AX,
\]
where \(A\) is independent of \(X\).
The full unified constant is obtained by substituting the radial variable \(\abs{A}\norm{X}_2\) into \cref{appendix-cor:trans-factorized-spherical-score}; its constant in the privacy regime uses only
\[
  \E(\abs{A}\norm{X}_2)
  =
  \E\abs{A}\,\E\norm{X}_2.
\]
For linear regression with spherical design, \(A=\ell_f(\xi)\).  With Gaussian
noise,
\[
  \E\abs{\ell_f(\xi)}
  =
  \frac{1}{\sigma}\sqrt{\frac{2}{\pi}}.
\]
For Gaussian design, \(\E\norm{X}_2=\E\chi_p\).  In Gaussian linear
regression this factorization holds at every \(\beta_0\), since
\[
  s_{\beta_0}(X,Y)
  =
  \frac{X(Y-X^\top \beta_0)}{\sigma^2}
  =
  \frac{X\xi}{\sigma^2}.
\]
For the null canonical models with Gaussian design in
\cref{prop:gauss-design-glm}, the three values
\[
  d_0
  =
  \frac{1}{\sigma}\sqrt{\frac{2}{\pi}},
  \qquad
  d_0=\frac{1}{2},
  \qquad
  d_0=\E\abs{Y_0-1}=\frac{2}{e}
\]
correspond respectively to Gaussian linear, logistic, and Poisson log-linear
regression.  Gaussian-polynomial and Gaussian-integrable exponential
envelopes verify \cref{ass:regular,ass:score-density-l1} for the
last two models, respectively.

\subsection{Patterns in the explicit constants}
\label{appendix-subsec:explicit-constant-patterns}

The examples exhibit two distinct phenomena.
In the Bernoulli, Laplace, saturated categorical, fixed radius, and bounded
mean models, one statistic or one configuration is simultaneously optimal for
sampling and privacy.
Their unified constants are therefore the affine interpolations of their
sampling and privacy regime constants; equivalently, after root-\(n\)
normalization they equal the corresponding additive expressions.

In the logistic, Gaussian, and Binomial\((3,q)\) models, the score that is
efficient for sampling is not the statistic that is efficient for privacy.
The common statistic is instead a clipped score, and its clipping threshold
changes with \(\lambda\).
The logistic and Binomial\((3,q)\) examples give elementary nonadditive
closed forms, while the Gaussian example reduces the full vector optimization
to a unique one-dimensional root.
More generally, \cref{appendix-prop:trans-spherical-clipping} gives the same
reduction involving a scalar root for every spherically symmetric score law,
including the spherically symmetric location families calculated in the privacy regime
in \cref{appendix-subsubsec:spherical-location-models}.

The Poisson mean model is different again: its two clipping points are
asymmetric and jointly balance the lower and upper tail overshoots.
As they cross the integer lattice, the constant is piecewise rational over
countably many phases rather than an expression with a single symmetric
threshold.
More generally, \cref{appendix-prop:trans-scalar-clipping} reduces every
one-dimensional transition calculation to two clipping thresholds; the
Poisson stop-loss and Stein identities are what make those thresholds and the
resulting constant explicitly piecewise rational.

\bibliographystyle{plainnat}
  \bibliography{main}

\end{document}